\documentclass{amsart}
\usepackage{amsmath,amssymb,amscd,latexsym}

\usepackage[all]{xy}

\newcommand{\F}{{\mathbb{F}}}

\newcommand{\N}{{\mathbb{N}}}

\newcommand{\Sph}{{\mathbb{S}}}

\newcommand{\Z}{{\mathbb{Z}}}

\newcommand{\Eni}{E'_{n,I}}

\newcommand{\Ev}{\mathrm{Ev}}
\newcommand{\op}{\mathrm{op}}
\newcommand{\pro}{\mathrm{pro}}

\newcommand{\colim}{\operatorname*{colim}}
\newcommand{\hocolim}{\operatorname*{hocolim}}
\newcommand{\holim}{\operatorname*{holim}}
\newcommand{\cosk}{\mathrm{cosk}}
\newcommand{\compl}{\hat{(\cdot)}}

\newcommand{\Gal}{\mathrm{Gal}}

\newcommand{\Hom}{\mathrm{Hom}}
\newcommand{\homp}{\mathrm{hom}_{\hShp}}

\newcommand{\hompg}{\mathrm{hom}_{\hShpg}}
\newcommand{\map}{\mathrm{map}}
\newcommand{\mapp}{\mathrm{map}_{\ast}}
\newcommand{\Map}{\mathrm{Map}}

\newcommand{\Mapg}{\mathrm{Map}_G}

\newcommand{\sk}{\mathrm{sk}}

\newcommand{\Tot}{\mathrm{Tot}\,}
\newcommand{\Ah}{{\mathcal A}}
\newcommand{\Ch}{{\mathcal C}}

\newcommand{\Eh}{{\mathcal E}}
\newcommand{\hEh}{\hat{\mathcal E}}
\newcommand{\Fh}{{\mathcal F}}

\newcommand{\Hh}{{\mathcal H}}
\newcommand{\Hhf}{{\mathcal H}_{\mathrm{fin}}}
\newcommand{\Hhp}{{\mathcal H}_{\ast}}
\newcommand{\hHh}{\hat{{\mathcal H}}}
\newcommand{\hHhp}{\hat{{\mathcal H}}_{\ast}}
\newcommand{\hHhg}{\hat{{\mathcal H}}_G}
\newcommand{\hHhpg}{\hat{{\mathcal H}}_{\ast G}}

\newcommand{\hL}{\hat{L}}

\newcommand{\Mh}{{\mathcal M}}

\newcommand{\Ph}{\mathcal{P}}

\newcommand{\Rh}{{\mathcal R}}
\newcommand{\Sh}{{\mathcal S}}
\newcommand{\Shp}{{\mathcal S}_{\ast}}
\newcommand{\SHh}{{\mathcal{SH}}}

\newcommand{\hSHh}{{\hat{\mathcal{SH}}}}
\newcommand{\hSHhg}{{\hat{\mathcal{SH}}_G}}
\newcommand{\hSh}{\hat{\mathcal S}}
\newcommand{\hShg}{\hat{\mathcal S}_G}

\newcommand{\hShp}{\hat{\mathcal S}_{\ast}}
\newcommand{\hShpg}{\hat{\mathcal S}_{\ast G}}
\newcommand{\hShpk}{\hat{\mathcal S}_{\ast K}}

\newcommand{\hSp}{\mathrm{Sp}(\hShp)}
\newcommand{\hSpg}{\mathrm{Sp}(\hShpg)}
\newcommand{\hSpk}{\mathrm{Sp}(\hShpk)}

\newcommand{\Spk}{\mathrm{Sp}(\Sh_{\ast K})}

\newcommand{\Sp}{\mathrm{Sp}(\Sh_{\ast})}

\newtheorem{theorem}{Theorem}[section]
\newtheorem{lemma}[theorem]{Lemma}
\newtheorem{prop}[theorem]{Proposition}
\newtheorem{defn}[theorem]{Definition}

\newtheorem{cor}[theorem]{Corollary}
\theoremstyle{definition}
\newtheorem{convention}[theorem]{Convention}
\newtheorem{example}[theorem]{Example}

\newtheorem{remark}[theorem]{Remark}
\begin{document}
\title{Continuous homotopy fixed points for Lubin-Tate spectra}
\author{Gereon Quick}\thanks{The author was supported by Research Fellowship QU 317/1 of the German Research Foundation (DFG)}
\address{Department of Mathematics, Harvard University, Cambridge, MA 02138, USA}
\email{gquick@math.harvard.edu}
\date{}
\begin{abstract}
We construct a stable model structure on profinite spectra with a continuous action of an arbitrary profinite group. This provides a natural framework for a new and conceptually simplified construction of continuous homotopy fixed point spectra and of continuous homotopy fixed point spectral sequences for Lubin-Tate spectra under the action of the extended Morava stabilizer group. 
\end{abstract}
\maketitle

\section{Introduction}

For the action of a discrete group $G$ on a spectrum $X$ there are well-known constructions for the homotopy fixed point spectrum $X^{hG}$ and for the homotopy fixed point spectral sequence. For fibrant $X$, the spectrum $X^{hG}$ is given by the $G$-fixed points of the function spectrum $F(EG_+,X)$, where $EG$ is a contractible free $G$-space. For each spectrum $Z$, the spectral sequence 
\begin{equation}\label{dssintro}
H^{\ast}(G;X^{\ast}Z)\Rightarrow [Z,X^{hG}]^{\ast}
\end{equation}
is induced by the filtration by the finite subskeleta of $EG$. But in some cases of interest, the group $G$ and the spectrum $X$ carry additional structures that one would like to take care of. For example, this is the case for the most important group action in the chromatic approach to stable homotopy theory, the action of the extended Morava stabilizer group $G_n$ on the $p$-local Landweber exact spectrum $E_n$. Let us shortly describe this famous example.\\ 
Let $p$ be a fixed prime, $n\geq 1$ an integer and $\F_{p^n}$ the field with $p^n$ elements. Let $S_n$ be the $n$th Morava stabilizer group, i.e. the automorphism group of the height $n$ Honda formal group law $\Gamma_n$ over $\F_{p^n}$. We denote by $\Gal(\F_{p^n}/\F_p)$ the Galois group of $\F_{p^n}$ over $\F_p$ and let $G_n=S_n \rtimes \Gal(\F_{p^n}/\F_p)$ be the semi-direct product.
By the work of Lubin and Tate \cite{lubintate}, there is a universal ring of deformations $E(\F_{p^n},\Gamma_n)=W(\F_{p^n})[[u_1,\ldots,u_{n-1}]]$ of $(\F_{p^n},\Gamma_n)$, where $W(\F_{p^n})$ denotes the ring of Witt vectors of $\F_{p^n}$.  The $MU_{\ast}$-module $E(\F_{p^n},\Gamma_n)[u,u^{-1}]$ induces via the Landweber exact functor theorem a homology theory and hence a spectrum, denoted by $E_n$ and called Lubin-Tate spectrum, with $E_{n\ast}=E(\F_{p^n},\Gamma_n)[u,u^{-1}]$, $|u|=-2$. The profinite group $G_n$ acts on the ring $E_{n\ast}$ (cf. \cite{dh2}). By Brown representability, this induces an action of $G_n$ by maps of rings in the stable homotopy category. Furthermore, Goerss, Hopkins and Miller have shown the crucial fact that there is even a $G_n$-action on the spectrum-level on $E_n$ that induces the action in the stable category (see \cite{goersshopkins} and \cite{rezk}).\\
Now $S_n$, $\Gal(\F_{p^n}/\F_p)$ and hence also $G_n$ are profinite groups. Moreover each homotopy group $\pi_t E_n$ has the structure of a continuous profinite $G_n$-module. The continuity of the action of $G_n$ on each $\pi_tE_n$ is an important property for stable homotopy theory. For by Morava's change of rings theorem, the $K(n)_{\ast}$-local $E_n$-Adams spectral sequence for the sphere spectrum $S^0$ has the form
\begin{equation}\label{K(n)AdamsSS}
H^{\ast}(G_n;E_{n\ast})\Rightarrow \pi_{\ast}L_{K(n)}S^0
\end{equation}
where the $E_2$-term is continuous cohomology of $G_n$ with profinite coefficients $E_{n\ast}$. Here $K(n)$ denotes the $n$th Morava $K$-theory and $L_{K(n)}$ denotes $K(n)_{\ast}$-localization, cf. \cite{morava} and \cite{devinatz}. Hence $L_{K(n)}S^0$ looks like a {\em continuous} $G_n$-homotopy fixed point spectrum of $E_n$ and one would like to interpret the above spectral sequence as a continuous homotopy fixed point spectral sequence of the $G_n$-action. \\
But the classical construction of homotopy fixed points and its spectral sequence (\ref{dssintro}) do not reflect the topology on $G_n$. The function spectrum $F(EG_+,E_n)$ should consist of continuous maps in some sense and the $E_2$-term of the spectral sequence (\ref{dssintro}) should be continuous cohomology of $G$. Hence it is a fundamental question in stable homotopy theory to understand in which way $E_n$ can be viewed as an object with a {\em continuous} action under $G_n$, see \cite{devinatzhopkins}.\\ 
Devinatz and Hopkins \cite{devinatzhopkins} have circumvented this problem and given an ad hoc argument for the construction of continuous homotopy fixed points of $E_n$. They proceeded in two steps by first constructing homotopy fixed points, here denoted by $E_n^{dhU}$ by adopting the notation in \cite{davis}, for an open subgroup $U$ of $G_n$ using that $G_n/U$ is finite. In a second step they defined $E_n^{dhG}$ for a closed subgroup $G$. Since $G_n$ is a $p$-adic analytic group, it is possible to find a sequence of open subgroups $G_n=U_0\supset U_1 \supset \ldots$ whose intersection is the trivial subgroup. Then $E_n^{dhG}$ is defined as an appropriate homotopy colimit of the $E_n^{dhU_iG}$'s. Moreover, for every closed subgroup $G$ of $G_n$, they provided a construction of a $K(n)_*$-local Adams spectral sequence for the homotopy fixed point spectrum $E_n^{dhG}$ whose $E_2$-term equals the desired continuous cohomology.\\ 
But since the argument of \cite{devinatzhopkins} did not explain in which sense $G_n$ acts continuously on $E_n$, the question remained how to view $E_n$ as an actual continuous spectrum and to find a natural framework for the continuous homotopy fixed point spectral sequence. The purpose of this paper is to give a new and complete answer to this question.\\
A previous approach has been developed by Davis in \cite{davis} and by Behrens and Davis in \cite{behrensdavis} by studying discrete $G$-spectra. Davis used the idea of Devinatz and Hopkins to start with the homotopy fixed point spectrum $E_n^{dhU}$ of \cite{devinatzhopkins} for an open subgroup $U\subset G_n$ and defined a new spectrum $F_n:=\colim_iE_n^{dhU_i}$ where the $U_i$ run through a fixed sequence of open subgroups of $G_n$ as above. The $K(n)_{\ast}$-localization of $F_n$ is equivalent to $E_n$, where $K(n)_*$ denotes the $n$th Morava $K$-theory. One can regard the localization of $F_n$ as a continuous discrete $G_n$-spectrum. Furthermore, Davis developed a stable homotopy theory for continuous discrete $G$-spectra, for an arbitrary profinite group $G$, based on the work of Goerss \cite{goerss} and Jardine \cite{jardine}. Then he defined systematically the homotopy fixed points for closed subgroups of $G_n$ and constructed a continuous homotopy fixed point spectral sequence.\\ 
A different method has been used by Fausk. In \cite{fausk}, Fausk constructed a model structure for pro-$G$-spectra, where $G$ denotes a compact Hausdorff topological group, e.g. a profinite group. He also obtained results on homotopy fixed points, descent spectral sequences and iterated homotopy fixed points. These results are equivalent to those of \cite{davis} if $G$ has finite virtual cohomological dimension.\\ 
But the crucial point is that if one wants to use the methods of Davis or Fausk for Lubin-Tate spectra $E_n$, one first has to apply the construction of \cite{devinatzhopkins} for open subgroups and has to rewrite $E_n$ as the $K(n)_*$-localization of a suitable colimit of the $E_n^{dhU_i}$'s as in \cite{devinatzhopkins}. Hence the above question still remained open how to view $E_n$ as a continuous spectrum without using \cite{devinatzhopkins} for open subgroups of $G_n$ and to give a unified construction for all closed subgroups of $G_n$ without \cite{devinatzhopkins}.\\
The approach of the present paper provides a new unified natural construction of continuous homotopy fixed points for any closed subgroup independent of \cite{devinatzhopkins} and hence, in particular, also a new construction for open subgroups of $G_n$. The idea is straightforward. Since the homotopy groups $\pi_tE_n$ are not discrete but profinite $G_n$-modules, a natural guess would be to look for a profinite structure on $E_n$. And, in fact, there is one in the following sense. There is a model for $E_n$ that is built out of a sequence of simplicial profinite sets that carry a continuous $G_n$-action. Consequently, a natural setting to study the action of $G_n$ on $E_n$ is a suitable category of continuous profinite $G_n$-spectra.\\

Let us give a quick outline of the strategy and provide precise statements of the main results of this paper. We will study continuous actions on profinite spectra of an arbitrary profinite group $G$ in general. A pointed profinite $G$-space is a simplicial object in the category of profinite sets with the limit topology and a continuous $G$-action together with a choice of basepoint. These pointed profinite $G$-spaces form a category $\hShpg$ with levelwise continuous $G$-equivariant maps as morphisms. A profinite $G$-spectrum $X$ is then a sequence of pointed profinite $G$-spaces $X_n$ with maps $S^1\wedge X_n \to X_{n+1}$ for all $n$, where the simplicial circle $S^1$ is a simplicial finite set with trivial $G$-action. \\ 
We will start with fundamental properties of profinite $G$-spaces. Then we introduce the category $\hSpg$ of profinite $G$-spectra and show that $\hSpg$ is equipped with a natural stable model structure.  \\ 
The fibrant replacement functor $R_G$ in $\hSpg$ will allow us to give a natural definition for {\em continuous} homotopy fixed point spectra. In fact, the homotopy fixed point spectrum $X^{hG}$ of a profinite $G$-spectrum $X$ is defined as a continuous  mapping spectrum $X^{hG}:=\Mapg(EG_+,R_G X)$ of $G$-equivariant and levelwise continuous maps in $\hSpg$ (see Definitions \ref{stableMapgdefn} and \ref{homotopyfixedpoints}). We will show that these homotopy fixed point spectra are equipped with a convergent spectral sequence
$$H^{\ast}(G;\pi_*X) \Rightarrow \pi_*X^{hG}$$
whose $E_2$-terms are given by the continuous cohomology groups of $G$. \\
The striking advantage of studying profinite actions in the category of profinite spectra is that $G$ and its classifying space $EG$ yield natural objects in $\hShg$. The homotopy fixed point spectral sequence is then obtained just as for a finite group by filtering $EG$ by its finite subskeleta. (But one should note that although $EG$ and $X$ are profinite, the function spectrum $\Mapg(EG_+,X)$ does not in general inherit a profinite structure, since, roughly speaking, the limit of $EG$ is turned into a colimit. Hence the homotopy groups of $X^{hG}$ are not profinite anymore in general.)\\
%

In order to be able to apply these techniques to the action of $G_n$ on the Lubin-Tate spectrum $E_n$, we have to prove that we may consider $E_n$ as a continuous profinite $G_n$-spectrum. The starting observation is that $E_n$ has a decomposition as a homotopy limit of spectra $\holim_I E_n\wedge M_I$, where the $M_I$ denote generalized Moore spectra corresponding to an inverse system of ideals $I$ in $BP_*$, cf.\,\cite{homasa}. These spectra have the important property that, for each such ideal $I$ and for every $t$, the homotopy groups $\pi_t(E_n\wedge M_I)$ are finite (being trivial if $t$ is odd). For any spectrum with finite homotopy groups, we will construct a concrete stably equivalent model as a profinite spectrum. This will provide a model of $E_n\wedge M_I$ in the category of profinite $G_n$-spectra. Taking the homotopy limit over all $I$ will yield a model of $E_n$ as a profinite $G_n$-spectrum. The following results will be proven in the last section.
\begin{theorem}\label{Endecompintro}
$E_n$ has a canonical model in the category of continuous profinite $G_n$-spectra, i.e. there is a $G_n$-equivariant map of spectra $\psi:E_n \to E'_n$ such that $E_n'$ is a profinite $G_n$-spectrum and $\psi$ is a stable equivalence of underlying spectra.
\end{theorem}
This allows us to apply the techniques described above to $E_n$ and to prove the following theorem where $K(n)_*$ denotes the $n$th Morava $K$-theory and $L_{K(n)}$ the $K(n)_*$-localization functor of spectra.  
\begin{theorem}\label{descentssintro}
Let $G$ be a closed subgroup of $G_n$.\\
{\rm (i)} There is a $K(n)_*$-local continuous homotopy fixed point spectrum $E_n^{hG}$ of $E_n$ which is natural in $G$ and equivalent to the fixed point spectrum $E_n^{dhG}$ of \cite{devinatzhopkins}. In particular, we obtain an equivalence $E_n^{hG_n}\simeq E_n^{dhG_n} \simeq L_{K(n)}S^0$.\\
{\rm (ii)} There is a natural strongly convergent continuous homotopy fixed point spectral sequence starting from continuous cohomology
$$H^{\ast}(G;\pi_*E_n) \Rightarrow \pi_*E_n^{hG}$$
which is isomorphic to the $K(n)_*$-local $E_n$-Adams spectral sequence converging to $E_n^{dhG}$. 
\end{theorem}
This theorem restates some of the main results of \cite{devinatzhopkins} in the setting of profinite $G$-spectra. The point is that it has a conceptually simpler proof. While in \cite{devinatzhopkins} the special properties of $G_n$ and $E_n$ have been taken advantage of to define homotopy fixed points, the methods we develop to prove Theorem \ref{descentssintro} are general and work for any continuous action of a profinite group on any profinite spectrum. (The only place where we use that $G_n$ is a finitely generated profinite group is when we show that $E_n$ has a model in the category of profinite $G_n$-spectra.) In particular, the definition of $E_n^{hG}$ is the same for all closed (or open) subgroups. \\

Finally, in order to prove that the canonical comparison map $E_n^{dhG}\to E_n^{hG}$ is an equivalence, we will show that continuous mapping spectra yield a $K(n)_*$-local $E_n$-Adams resolution of $E_n^{dhG}$. The key point is that we show the following result, where $\Map(-,-)$ denotes the continuous mapping spectrum functor of Definition \ref{stableMapdefn} and Remark \ref{stableMapgactionrem}.
\begin{theorem}
Let $S$ be a profinite set. There is a map of spectra in the stable homotopy category 
$$L_{K(n)}(E_n\wedge \Map(S,E'_n)) \to \Map(S\times G_n,E'_n)$$
that induces an isomorphism on homotopy groups. In particular, there is an isomorphism in the stable homotopy category
$$L_{K(n)}(E_n\wedge E_n)\cong \Map(G_n,E'_n).$$
\end{theorem}
The last isomorphism of spectra corresponds to the isomorphism of homotopy groups $\pi_*\hL(E_n\wedge E_n)\cong \Map(G_n,\pi_*E_n)$ of \cite{devinatzhopkins}, \S 2. It is an important feature of the category of profinite $G_n$-spectra that we are able to express this relation on the level of spectra and not only for homotopy groups.\\
%
%
%

{\bf Acknowledgements.} The starting point of this project was a question by Dan Isaksen if the profinite spectra of \cite{etalecob} fit in the theory of Lubin-Tate spectra. I would like to thank him very much for this generous hint. I am very grateful to Mike Hopkins for very helpful discussions. I would like to thank Daniel Davis for many comments.
\section{Homotopy theory of profinite spaces}
\subsection{Profinite spaces}\label{secprofinspaces}
We start with basic definitions and invariants for profinite spaces that will be necessary to construct a homotopy category for profinite spectra. \\
For a category $\Ch$ with small limits, the pro-category of $\Ch$, denoted pro-$\Ch$, has as objects all cofiltering diagrams $X:I \to \Ch$. Its sets of morphisms are defined as
$$\Hom_{\pro-\Ch}(X,Y):=\lim_{j\in J}\colim_{i\in I} \Hom_{\Ch}(X(i),Y(j)).$$
A constant pro-object is one indexed by the category with one object and one identity map. The functor sending an object $X$ of $\Ch$ to the constant pro-object with value $X$ makes $\Ch$ a full subcategory of pro-$\Ch$. The right adjoint of this embedding is the limit functor $\lim$: pro-$\Ch$ $\to \Ch$, which sends a pro-object $X$ to the limit in $\Ch$ of the diagram $X$.\\ 
Let $\Eh$ denote the category of sets and let $\Fh$ be the full subcategory of finite sets. Let $\hEh$ be the category of compact Hausdorff and totally disconnected topological spaces. We may identify $\Fh$ with a full subcategory of $\hEh$ in the obvious way. The limit functor $\lim$: pro-$\Fh \to \hEh$ is an equivalence of categories. Moreover the forgetful functor $\hEh \to \Eh$ admits a left adjoint $\compl:\Eh \to \hEh$ which is called profinite completion.\\ 
We denote by $\hSh$ (resp. $\Sh$) the category of simplicial profinite sets (resp. simplicial sets). The objects of $\hSh$ (resp. $\Sh$) will be called profinite spaces (resp. spaces). The category $\hSh$ was studied for the first time by Morel in \cite{ensprofin}.\\
For a profinite space $X$, we define the set $\Rh(X)$ of simplicial open equivalence relations on $X$. An element $R$ of $\Rh(X)$ is a simplicial profinite subset of the product $X\times X$ such that, in each degree $n$, $R_n$ is an equivalence relation on $X_n$ and an open subset of $X_n\times X_n$. It is ordered by inclusion. For every element $R$ of $\Rh(X)$, the quotient $X/R$ is a simplicial finite set and the map $X \to X/R$ is a map of profinite spaces. The canonical map $X \to \lim_{R\in \Rh(X)} X/R$ is an isomorphism in $\hSh$, cf. \cite{ensprofin}, Lemme 1.\\ 
The profinite completion of sets induces a functor $\compl:\Sh \to \hSh$, which is also called profinite completion. For a space $Z$, its profinite completion can be described as follows. Let $\Rh(Z)$ be the set of simplicial equivalence relations $R$ on $Z$ such that the quotient $Z/R$ is a simplicial finite set. The set $\Rh(Z)$ is ordered by inclusion. The profinite completion is defined as the limit of the $Z/R$ for all $R\in \Rh(Z)$, i.e. $\hat{Z} := \lim_{R\in \Rh(Z)} Z/R$. Profinite completion of spaces is again left adjoint to the forgetful functor $|\cdot|:\hSh \to \Sh$ which sends a profinite space to its underlying simplicial set. \\
Let $X$ be a profinite space and $\pi$ a topological abelian group. The continuous cohomology $H^{\ast}(X;\pi)$ of $X$ with coefficients in $\pi$ is defined as the cohomology of the complex $C^{\ast}(X;\pi)$ of continuous cochains of $X$ with values in $\pi$, i.e. $C^n(X;\pi)$ denotes the set $\Hom_{\mathrm{cont}}(X_n,\pi)$ of continuous maps $\alpha:X_n \to \pi$ and the differentials $\delta^n:C^n(X;\pi)\to C^{n+1}(X;\pi)$ are the morphisms associating to $\alpha$ the map $\sum_{i=0}^{n+1}(-1)^i\alpha \circ d_i$, where $d_i$ denotes the $i$th face map $X_{n+1}\to X_n$. If $\pi$ is a finite abelian group and $Z$ a simplicial set, then the cohomology $H^{\ast}(Z;\pi)$ of the space $Z$ and the continuous cohomology $H^{\ast}(\hat{Z};\pi)$ of the profinite completion $\hat{Z}$ are canonically isomorphic. 
\begin{convention}
Above and in the rest of the paper we do not use a special notation for continuous cohomology. For a profinite space and a topological coefficient group, cohomology will always mean continuous cohomology. 
\end{convention}
If $G$ is an arbitrary profinite group, we may still define the first cohomology of $X$ with coefficients in $G$ following Morel's idea in \cite{ensprofin} p. 355. The functor $(\hSh)^{\mathrm{op}}\to \Eh$, $X\mapsto \Hom_{\hEh}(X_0,G)$, is represented in $\hSh$ by the profinite space $EG$, whose set of $n$-simplices is $EG_n=G^{n+1}$, the $(n+1)$-fold product of $G$. We define the $1$-cocycles $Z^1(X;G)$ to be the set of continuous maps $f:X_1 \to G$ such that $f(d_0x)f(d_2x)=f(d_1x)$ for every $x \in X_1$. The functor $(\hSh)^{\mathrm{op}}\to \Eh$, $X\mapsto Z^1(X;G)$ is represented by a profinite space $BG=EG/G$, whose set of $n$-simplices is $BG_n=G^{n}$, the $n$-fold product of $G$. Furthermore, there is a map 
$$\delta:\Hom_{\hSh}(X,EG) \to Z^1(X;G)\cong \Hom_{\hSh}(X,BG)$$ 
which sends $f:X_0 \to G$ to the $1$-cocycle $x\mapsto \delta f(x)=f(d_0x)f(d_1x)^{-1}$. We denote by $B^1(X;G)$ the image of $\delta$ in $Z^1(X;G)$ and we define the pointed set $H^1(X;G)$ to be the quotient $Z^1(X;G)/B^1(X;G)$. Finally, if $X$ is a profinite space, we define $\pi_0X$ to be the coequalizer in $\hEh$ of the diagram $d_0,d_1:X_1 \rightrightarrows X_0$.\\ 
The profinite fundamental group of $X$ is defined via covering spaces in the spirit of Grothendieck. There is a universal profinite covering space $(\tilde{X},x)$ of $X$ at a vertex $x \in X_0$. Then $\pi_1(X,x)$ is defined to be the group of automorphisms of $(\tilde{X},x)$ over $(X,x)$. As the limit of the finite automorphism groups of the finite Galois coverings of $(X,x)$ the group $\pi_1(X,x)$ is naturally a profinite group. For details we refer the reader to \cite{profinhom}.     
\begin{defn}\label{defnwe}
A morphism $f:X\to Y$ in $\hSh$ is called\\
{\rm (1)} a weak equivalence if the induced map $f_{\ast}:\pi_0(X) \to \pi_0(Y)$ is an isomorphism of profinite sets, $f_{\ast}:\pi_1(X,x) \to \pi_1(Y,f(x))$ is an isomorphism of profinite groups for every vertex $x\in X_0$, and $f^{\ast}:H^q(Y;\Mh) \to H^q(X;f^{\ast}\Mh)$ is an isomorphism for every local coefficient system $\Mh$ of finite abelian groups on $Y$ for every $q\geq 0$;\\
{\rm (2)} a cofibration if $f$ is a levelwise monomorphism;\\
{\rm (3)} a fibration if it has the right lifting property with respect to every cofibration that is also a weak equivalence. 
\end{defn}
The category $\hSh$ has a simplicial structure (see also \cite{dehon}, \S 1.2). Let $X$ and $Y$ be profinite spaces. The mapping space $\map_{\hSh}(X,Y)$ is defined as the simplicial set whose set of $n$-simplices is given as the set of maps 
$$\map_{\hSh}(X,Y)_n=\Hom_{\hSh}(X\times \Delta[n], Y)$$
and whose simplicial structure is induced by the cosimplicial structure of the standard simplex $[n] \mapsto \Delta[n]$. This defines a functor 
$$\map_{\hSh}(-,-):\hSh^{\mathrm{op}}\times \hSh \to \Sh.$$
The tensor and cotensor structure on the category $\hSh$ is defined as follows. 
Let $K$ be a finite simplicial set, i.e. a simplicial set with only finitely many nondegenerate simplices. This implies, in particular, that $K$ is a simplicial finite set. Let $X$ be a profinite space. The tensor object $X\otimes K\in \hSh$ is defined as the levelwise defined product of simplicial profinite sets $X\times K$. For the cotensor object, we recall that any profinite space $X$ is canonically  isomorphic to a limit $\lim_{\beta}X_{\beta}$ of simplicial finite sets. This implies that the set of maps $\Hom_{\hSh}(K, X)$ inherits a natural structure as a profinite set given by the limit of finite sets
$$\lim_{\beta} \Hom_{\hSh}(K,X_{\beta}).$$
The cotensor or function object in $\hSh$ is defined as the profinite space $\hom_{\hSh}(K,X)\in \hSh$ whose set of $n$-simplices is given by the profinite set of maps
$$\hom_{\hSh}(K,X)_n = \Hom_{\hSh}(K\times \Delta[n], X).$$ 
If $K$ is an arbitrary simplicial set, it is isomorphic to the filtered colimit over its finite simplicial subsets $K_{\alpha}$. For a profinite space $X$, we define the tensor object $X\otimes K$ to be the colimit in $\hSh$ of the profinite spaces $X\otimes K_{\alpha}$ (see also \cite{dehon}, \S 1.2), i.e. 
$$X\otimes K := \colim_{\alpha} X\times K_{\alpha} ~ \mathrm{in}~\hSh.$$
The function object is defined to be the limit in $\hSh$ of the profinite spaces $\hom_{\hSh}(K_{\alpha},X)$, i.e. 
$$\hom_{\hSh}(K,X):= \lim_{\alpha} \hom_{\hSh}(K_{\alpha},X) ~ \mathrm{in}~\hSh.$$
Let $X$ and $Y$ be profinite spaces and let $K$ be a simplicial set. Then mapping space, tensor and function objects are connected by the natural bijections
$$\map_{\hSh}(X\otimes K, Y) \cong \map_{\Sh}(K, \map_{\hSh}(X,Y))$$
and 
$$\map_{\hSh}(Y, \hom_{\hSh}(K, X)) \cong \map_{\Sh}(K, \map_{\hSh}(X,Y))$$
where $\map_{\Sh}(-,-)$ denotes the mapping space functor on $\Sh$. This defines the structure of a simplicial category on $\hSh$ in the sense of \cite{homalg}, II \S\S 1+2.\\
%
The following theorem was proven in \cite{gspaces}, Theorem 2.3. 
\begin{theorem}\label{modelstructure}
The classes of weak equivalences, cofibrations and fibrations, as defined above, provide $\hSh$ with the structure of a fibrantly generated left proper simplicial model category. We denote the homotopy category by $\hHh$.
\end{theorem}
We consider the category $\Sh$ of simplicial sets with the usual model structure of \cite{homalg}. We denote its homotopy category by $\Hh$. Then the next result was shown in \cite{profinhom}.
\begin{prop}\label{adjcompletion}
{\rm 1.} The levelwise completion functor $\compl: \Sh \to \hSh$ preserves weak equivalences and cofibrations.\\
{\rm 2.} The forgetful functor $|\cdot|:\hSh \to \Sh$ preserves fibrations and weak equivalences between fibrant objects.\\
{\rm 3.} The induced completion functor $\compl: \Hh \to \hHh$ and the right derived functor $R|\cdot|:\hHh \to \Hh$ form a pair of adjoint functors.
\end{prop}
\subsection{Pointed profinite spaces}
Let $\hShp$ be the category of pointed profinite spaces, i.e. profinite spaces $X$ with a chosen map from the one-point space to $X$. Profinite completion of spaces induces in the obvious way a profinite completion functor $\compl:\Shp \to \hShp$, where $\Shp$ denotes the category of pointed spaces. \\
For pointed profinite spaces $X$ and $Y$, let $X\vee Y$ be the wedge of $X$ and $Y$ over the base point. The smash product $X\wedge Y$ is defined to be the quotient in $\hShp$ 
$$X\wedge Y:= (X\times Y)/(X\vee Y).$$
The category $\hShp$ also has a simplicial structure (cf. again \cite{dehon}, \S 1.2). Let $X$ and $Y$ be pointed profinite spaces. The mapping space $\map_{\hShp}(X,Y)$ is defined as the simplicial set whose set of $n$-simplices is given as the set of maps 
$$\map_{\hShp}(X,Y)_n=\Hom_{\hShp}(X\wedge \Delta[n]_+, Y)$$
where $\Delta[n]_+$ denotes the standard simplicial set with an additional disjoint base point.
This defines a functor 
$$\map_{\hShp}(-,-):\hShp^{\mathrm{op}}\times \hShp \to \Sh.$$
The tensor and cotensor structure are defined as follows. Let $K$ be a finite simplicial set, and $X$ a pointed profinite space. The tensor object $X\otimes K\in \hShp$ is defined as the smash product $X\wedge K_+$. The function object in $\hShp$ is defined as the pointed profinite space $\hom_{\hShp}(K,X)\in \hShp$, pointed by the constant map to the basepoint of $X$, whose set of $n$-simplices is given by the profinite set of maps
$$\hom_{\hShp}(K,X)_n = \Hom_{\hShp}(K_+\wedge \Delta[n]_+, X).$$ 
If $K$ is an arbitrary simplicial set and $X$ a pointed profinite space, we define the tensor object $X\otimes K$ to be the colimit in $\hShp$ of the profinite spaces $X\wedge K_{\alpha+}$ where $K_{\alpha}$ runs through the finite simplicial subsets of $K$ (see also \cite{dehon}, \S 1.2). 
The function object $\hom_{\hShp}(K,X)$ is defined to be the limit in $\hShp$ of the pointed profinite spaces $\hom_{\hShp}(K_{\alpha},X)$. \\
Let $X$ and $Y$ be pointed profinite spaces and let $K$ be a simplicial set. Then mapping space, tensor and function objects are connected by the natural bijections
$$\map_{\hShp}(X\otimes K, Y) \cong \map_{\Sh}(K, \map_{\hShp}(X,Y))$$
and 
$$\map_{\hShp}(Y, \hom_{\hShp}(K, X)) \cong \map_{\Sh}(K, \map_{\hShp}(X,Y)).$$
This defines the structure of a simplicial category on $\hShp$.\\
If the finite simplicial set $K$ is already equipped with a base point and $X$ is a pointed profinite space, we also denote by $\homp(K,X)\in \hShp$ the pointed profinite space whose set of $n$-simplices is given by the profinite set of maps
$$\hom_{\hShp}(K,X)_n = \Hom_{\hShp}(K\wedge \Delta[n]_+, X).$$
Moreover, if $X$ is a pointed profinite space and $K$ is an arbitrary pointed simplicial set, we define the pointed profinite space $\homp(K,X)\in \hShp$ to be the limit in $\hShp$
$$\homp(K,X)= \lim_{\alpha} \homp(K_{\alpha},X)$$
where $K_{\alpha}$ runs through the pointed finite simplicial subsets of $K$ such that $K$ is isomorphic to $\colim_{\alpha}K_{\alpha}$ as a pointed simplicial set.
If $K$ and $L$ are pointed simplicial sets and $X$ a pointed profinite space, we also have the following natural isomorphism
\begin{equation}\label{homhom}
\homp(K, \homp(L,X)) \cong \homp(K \wedge L,X) \cong \homp(L, \homp(K,X)).
\end{equation} 
\begin{example}
As an example of function objects, let $S^1$ be the simplicial circle, i.e. the quotient $S^1=\Delta[1]/\partial \Delta[1]$ of the standard simplex $\Delta[1]$ by its boundary. The pointed simplicial set $S^1$ is finite in each degree, i.e. it is a simplicial finite set and hence also an object in $\hShp$.  Taking the smash product with $S^1$ defines a functor $\hShp \to \hShp$, $X\mapsto S^1 \wedge X$. It is left adjoint to the functor $\hShp \to \hShp$ defined by sending a pointed profinite space $X$ to the function object $\hom_{\hShp}(S^1,X)$ in $\hShp$. We denote $\hom_{\hShp}(S^1,X)$ also by $\Omega X$. Hence for pointed profinite spaces $X$ and $Y$ there is a natural isomorphism
$$\map_{\hShp}(S^1\wedge X, Y) \cong \map_{\hShp}(X, \Omega Y).$$
\end{example}
As an under-category of $\hSh$, $\hShp$ inherits a model structure. We call a map in $\hShp$ a weak equivalence (cofibration, fibration) if its underlying map in $\hSh$ is a weak equivalence (cofibration, fibration). It follows from the general theory of model categories, as explained for example in \cite{hoveybook}, Proposition 1.1.8 and the dual of Lemma 2.1.21, that Theorem \ref{modelstructure} implies the following result.
\begin{theorem}\label{pointedmodelstructure}
The classes of weak equivalences, cofibrations and fibrations provide $\hShp$ with the structure of a fibrantly generated left proper simplicial model category. We denote the homotopy category by $\hHhp$.
\end{theorem}
The Quillen adjunction of Proposition \ref{adjcompletion} can be rephrased in terms of mapping spaces. Let $K$ be a pointed simplicial set and $X$ be a fibrant pointed profinite space. Then there is a natural isomorphism of fibrant simplicial sets
$$\map_{\hShp}(\hat{K}, X) \cong \map_{\Shp}(K, |X|).$$
Moreover, if $K=S^1$ is the simplicial circle, there is an isomorphism in $\Shp$
\begin{equation}
|\Omega X|=|\homp(S^1,X)| \cong \map_{\hShp}(S^1,X) \cong \map_{\Shp}(S^1, |X|) =: \Omega |X|.\end{equation}
\begin{remark}\label{cylinder}
Let $X$ be a pointed profinite space. The maps $d_0, d_1:\Delta[0] \rightrightarrows \Delta[1]$ and $s^0: \Delta[1] \to \Delta[0]$ induce maps $X\vee X \to X \wedge \Delta[1]_+$ and $X\wedge \Delta[1]_+ \to X$ and make $X\wedge \Delta[1]_+$ into a cylinder object for $X$ in $\hShp$. For the maps $X\to X\wedge \Delta[1]_+$ induced by $d_0$ and $d_1$ are cofibrations and weak equivalences and the induced map $X \wedge \Delta[1]_+ \to X$ is a weak equivalence in $\hShp$. Hence if $Y$ is a fibrant pointed profinite space, there is a natural bijection 
$$\pi_0\mapp(X,Y) \cong \Hom_{\hHhp}(X,Y).$$
Moreover, let $K$ be a pointed simplicial set, $X$, $Y$ pointed profinite spaces, and $Y$ fibrant in $\hShp$. If we denote the homotopy category of pointed simplicial sets by $\Hhp$, then there are natural bijections
$$\Hom_{\hHhp}(K\otimes X, Y) \cong \Hom_{\hHhp}(X, \homp(K,Y)),$$
$$\Hom_{\hHhp}(K\otimes X, Y) \cong \Hom_{\Hhp}(K, \map_{\hShp}(X,Y)),~\mathrm{and}$$
$$\Hom_{\hHhp}(\hat{K}, Y) \cong \Hom_{\Hhp}(K, |Y|).$$
\end{remark}
\subsection{An explicit fibrant replacement functor}
As indicated by Morel for pro-$p$-completion of spaces in \cite{ensprofin}, \S 2.1, p. 367, there is an explicit construction for the fibrant replacement in $\hSh$. This construction is based on the work of Quillen \cite{quillen2} and Dwyer-Kan \cite{dk}. We refer the reader to \cite{completion} for more details.\\
First, let $X$ be a reduced simplicial finite set. We denote by $\Gamma X$ the free loop group construction of $X$. When we apply the profinite completion of groups levelwise to the simplicial group $\Gamma X$, we obtain a simplicial profinite group denoted by $\hat{\Gamma}X$. Its profinite classifying space $\bar{W}\hat{\Gamma}X$ is a fibrant object in $\hSh$. It is equipped with a canonical map $\eta: X \to \bar{W}\hat{\Gamma}X$ and one can show that this is a weak equivalence of profinite spaces. The crucial point is that free groups are good in the sense of Serre which we will discuss below.\\
Second, let $X$ be an arbitrary simplicial finite set. We can generalize the previous argument by applying to $X$ the free loop groupoid construction $\Gamma X$ (see \cite{dk} or \cite{gj}, V \S 7, p. 322). It is defined in degree $n$ to be the free groupoid on generators $x: x_1 \to x_0$ with $x \in X_{n+1}$, subject to the relations $s_0 x_0 = 1_{x_0}$, $x_0 \in X_n$. The objects of this groupoid are just the vertices of X. One can define a functor $\Gamma X_n \to \Gamma X_m$ for each ordinal number morphism $[m] \to [n]$. Then we apply the degreewise profinite completion functor to groupoids. We call a groupoid $H$ finite if the set of objects of $H$ is finite and the group of automorphisms of any object is finite as well. The profinite completion of a groupoid is then the limit as a groupoid of the filtered system of its quotient groupoids which are finite groupoids in this sense (see also \cite{ensprofin}, p. 367). We obtain a simplicial profinite groupoid $\hat{\Gamma}X$. We apply again the classifying space functor $\bar{W}$ for simplicial groupoids of \cite{gj}, V, to get the fibrant profinite space $\bar{W}\hat{\Gamma}X$ and a canonical map $\eta: X \to \bar{W}\hat{\Gamma}X$ in $\hSh$.\\
Finally, let $X$ be an arbitrary profinite space. It is isomorphic in $\hSh$ to the limit $\lim_Q X/Q$ where $Q$ runs through the simplicial open equivalence relations on $X$, i.e. those simplicial relations such that $Q_n$ is an open subset of $X_n\times X_n$ for each $n$. We define its free simplicial profinite groupoid by 
$$\hat{G\Gamma}X := \lim_Q \hat{\Gamma}(X/Q).$$
Then we apply $\bar{W}$ to get a fibrant profinite space $RX$ equipped with a canonical map 
$$\eta: X\to RX:=\bar{W}\hat{\Gamma}X=\lim_Q \bar{W}\hat{\Gamma}(X/Q).$$
As before one can show that $\eta$ is a weak equivalence. Hence $\eta$ gives a fibrant replacement functor in $\hSh$. In particular, we can write $RX$ as a limit 
$$RX=\lim_{Q} \lim_{U_Q}\bar{W}(\Gamma(X/Q)/U_Q)$$
where $U_Q$ runs through the normal subgroupoids of $\Gamma(X/Q)$ such that the quotient $\Gamma(X/Q)/U_Q$ is a simplicial finite groupoid. Hence we can decompose a fibrant profinite space into a limit of simplicial finite sets which are fibrant in $\hSh$. 
Moreover, after taking Postnikov sections we can write $RX$ as a limit 
$$RX=\lim_{Q,n} \lim_{U_Q} \cosk_n \bar{W}(\Gamma(X/Q)/U_Q)$$
of simplicial finite sets which are also finite spaces, i.e. simplicial sets whose homotopy groups are all finite and only a finite number of them are nontrivial.\\
Finally, let $X$ be a pointed profinite space with basepoint given by a map $p:\ast \to X$ from the discrete simplicial set $\ast$ with a single vertex. The decomposition of $X$ as a limit of simplicial finite sets $X\cong \lim_QX/Q$ is then also a decomposition as a limit of pointed simplicial finite sets, since $p$ induces compatible maps $\ast \to X/Q$ which provide basepoints for each $X/Q$. Hence the map
$$\eta: X\to RX$$
provides in fact a fibrant replacement in the model structure of pointed profinite spaces such that $RX$ is a limit of pointed simplicial finite sets which are fibrant in $\hShp$. 
\subsection{Homotopy groups of profinite spaces}
Let $X$ be a pointed profinite space. The profinite fundamental group has been defined above. For $n \geq 2$, we define the $n$th homotopy group of $X$ as follows.
\begin{defn}\label{defhomotopygroups}
Let $X$ be a pointed profinite space and let $RX$ be a fibrant replacement of $X$ in the above model structure on $\hShp$. For $n\geq 2$, we define the $n$th profinite homotopy group of $X$  to be the group
$$\pi_n(X):=\pi_0(\Omega^n(RX)).$$
\end{defn}
As shown in \cite{profinhom}, this is consistent with the definition of $\pi_1$ as well. For $n\geq 2$, the higher homotopy groups $\pi_nX$ are abelian groups, since $\Omega^nX$ is a group object in $\hHhp$. 
\begin{lemma}\label{underlyingpi}
Let $X$ be a fibrant pointed profinite space. The profinite homotopy groups of $X$ defined in Definition \ref{defhomotopygroups} is naturally isomorphic to the usual homotopy group of the underlying fibrant simplicial set $|X|$. 
\end{lemma}
\begin{proof}
By Proposition \ref{adjcompletion}, the forgetful functor $|\cdot|:\hShp \to \Shp$ preserves fibrant objects and weak equivalences between fibrant objects. Since $X$ is a fibrant pointed profinite space, for every $n \geq 0$ we have an isomorphism of groups (respectively sets for $n=0$)
$$\pi_n(X) = \pi_0(\Omega^n(X)) \cong \pi_0(\Omega^n(|X|)) \cong \pi_n(|X|)$$
where $\pi_n(|X|)$ denotes the $n$th homotopy group of the simplicial set $|X|$.
\end{proof}
\begin{remark}
The assumption in the previous lemma that $X$ is fibrant in $\hShp$ of course crucial. If we start with an arbitrary pointed profinite space $Y$, then the homotopy group $\pi_n(|Y|)$ of its underlying space $|Y|$ is in general not isomorphic to the homotopy group $\pi_n(Y)$ of Definition \ref{defhomotopygroups}. Since the notion of weak equivalences in $\hShp$ is different from the one in $\Shp$, the fibrant replacement in $\hShp$ does not preserve the homotopy type of $|Y|$ in $\Hhp$
\end{remark}

Furthermore, the homotopy groups of Definition \ref{defhomotopygroups} carry a natural profinite structure. To simplify notations, let $X$ be fibrant in $\hShp$. Using the explicit fibrant replacement functor defined in the previous section, we can assume that $X$ is given as the limit $\lim_{\beta} X_{\beta}$ of pointed simplicial finite sets which are fibrant in $\hShp$. 
Since the underlying space of a fibrant object in $\hShp$ is a fibrant simplicial set, we see that the spaces underlying $X$ and each underlying space of $X_{\beta}$ are fibrant as simplicial sets. Furthermore, since $\hShp$ is a simplicial model category, $\homp(S^1, X_{\beta})$ is a simplicial finite set which is fibrant in $\hShp$. Hence 
$$\homp(S^1, X)=\homp(S^1, \lim_{\beta} X_{\beta}) = \lim_{\beta} \homp(S^1, X_{\beta})$$
is a cofiltered limit of fibrant simplicial finite sets. The fact that $\pi_n(X)$ is a profinite group is now implied by the following proposition due to Dehon and Lannes \cite{dehonlannes}, Proposition 1.1.3. 
\begin{prop}\label{pi0}
Let $\{ X_{\beta} \}_{\beta}$ be a cofiltered diagram of fibrant simplicial finite sets. Then the canonical map 
$$\pi_0(\lim_{\beta} X_{\beta}) \to \lim_{\beta} (\pi_0 X_{\beta})$$
is a bijection.
\end{prop}
%
%
\subsection{Homotopy limits of pointed profinite spaces}
%
%

Homotopy limits of diagrams of profinite spaces can be defined using the construction of Bousfield and Kan in \cite{bouskan}, XI. Let $I$ be a small category and $BI$ be the nerve of $I$. Recall that this is the simplicial set whose $n$-simplices are sequences
$$u=(i_0 \stackrel{\alpha_1}{\leftarrow} \ldots \stackrel{\alpha_n}{\leftarrow} i_n)$$
of morphisms in $I$. 
Moreover, for an object $i$ in $I$, one can form the over-category $I/i$ whose objects are maps $i_0 \to i$. The nerve $B(I/i)$ of this over-category is the simplicial set whose $n$-simplices are sequences
$$i \stackrel{\alpha}{\leftarrow} i_0 \stackrel{\alpha_1}{\leftarrow} \ldots \stackrel{\alpha_n}{\leftarrow} i_n$$ 
of morphisms in $I$. \\
Let $X(-):I\to \hShp$ be a functor from $I$ to pointed profinite spaces. 
The existence of cotensor objects in $\hShp$ allows to define the pointed profinite space 
$$\homp(B(I/-),X(-))\in \hShp$$ 
as the equalizer in $\hShp$ of the diagram of pointed profinite spaces
$$\prod_{i\in I} \homp(B(I/i),X(i)) \rightrightarrows \prod_{\alpha:i\to i' \in I} \homp(B(I/i),X(i'))$$
where the two maps are induced by the maps in $\hShp$
$$\homp(B(I/i),X(i)) \stackrel{X(\alpha)}{\longrightarrow} \homp(B(I/i), X(i'))$$
and 
$$\homp(B(I/i'),X(i')) \stackrel{B(I/\alpha)}{\longrightarrow} \homp(B(I/i), X(i')).$$ 
In analogy to \cite{bouskan}, XI \S 3, we make the following definition.
\begin{defn}\label{defholim}
Let $I$ be small category and $X(-): I \to \hShp$ be a functor. The homotopy limit of $X(-)$ is defined to be the pointed profinite space 
$$\holim_{i\in I} X(i):=\homp(B(I/-),X(-)).$$ 
\end{defn}
%
\begin{lemma}\label{holiminv}
Let $X(-)$ be a functor from a small category $I$ to pointed profinite spaces. If $X(i)$ is fibrant for every $i\in I$, then $\holim_{i\in I}X(i)$ is fibrant in $\hShp$.\\
Let $f(-):X(-) \to Y(-)$ be a natural transformation of functors from the small category $I$ to the full subcategory of fibrant pointed profinite spaces. Let 
$$f:=\holim_{i\in I}f(i): \holim_{i\in I}X(i) \to \holim_{i\in I}Y(i)$$
be the induced map on homotopy limits. If $f(i): X(i) \to Y(i)$ is a weak equivalence in $\hShp$ for every $i\in I$, then $f$ is a weak equivalence in $\hShp$.
\end{lemma}
\begin{proof}
This is proved in \cite{hirsch}, Theorems 19.4.1 and Theorem 19.4.2. See also Remark 19.1.6 and Proposition 16.6.6 in \cite{hirsch} that show that the definition of homotopy limits in \cite{hirsch}, \S 19, coincides with the construction in Definition \ref{defholim}. 
\end{proof}
\begin{lemma}\label{holimlemma}
Let $X(-): I \to \hShp$ be a small cofiltering diagram such that each $X(i)$ is connected and fibrant in $\hShp$. Then there is a natural isomorphism of homotopy groups of underlying simplicial sets for every $q\geq 0$
$$\pi_q(|\holim_{i\in I} X(i)|) \cong \lim_{i\in I} \pi_q(|X(i)|).$$
\end{lemma}
\begin{proof}
Since the right Quillen functor $|-|: \hShp \to \Shp$ of forgetting the profinite structure commutes with limits and cotensor objects, it also commutes with forming the homotopy limit of the diagram $X(-):I\to \hShp$, i.e. there is an isomorphism $\Shp$
$$|\holim_{i\in I} X(i)| \cong \holim_{i\in I} |X(i)|.$$
By \cite{bouskan} XI \S 7.1, since each $|X(i)|$ is a fibrant pointed simplicial set, there is a spectral sequence involving derived limits
\begin{equation}\label{bouskanss}
E_2^{s,t}=\left \{ \begin{array}{l@{\quad \quad}l}
                     \lim_I^s\pi_t (|X(i)|) & \mathrm{for}~0\leq s \leq t\\
                     0 & \mathrm{else}
                     \end{array} \right.
\end{equation}  
converging to $\pi_{s+t}\holim_i |X(i)|$. As we pointed out in Lemma \ref{underlyingpi}, the homotopy groups $\pi_t(|Y|)$ of the pointed simplicial set $|Y|$ underlying a fibrant pointed profinite space $Y$ satisfy 
$$\pi_t(|Y|)\cong \pi_t(Y)$$
where the right hand group is the profinite homotopy group of the fibrant pointed profinite space $Y$ of Definition \ref{defhomotopygroups}. This implies that the groups $\pi_t|X(i)|$ in (\ref{bouskanss}) are profinite groups for all $t$ and all $i$. Since the functor sending a small cofiltering diagram to its inverse limit is exact on the category of profinite groups (\cite{ribes} Proposition 2.2.4), spectral sequence (\ref{bouskanss}) degenerates to a single row and implies the desired isomorphism.
\end{proof}
\subsection{Comparison with the work of Artin-Mazur, Morel and Sullivan}
Let $\Hh_{\mathrm{fin}}$ denote the full subcategory of $\Hh$ of finite spaces, i.e. of simplicial finite sets whose homotopy groups are all finite and only a finite number of them are nontrivial. The morphisms are homotopy classes of maps. 
Let $\Hh_0$ denote the full subcategory of connected spaces. Artin and Mazur showed in \cite{artinmazur} that, for every space $X\in \Hh_0$, the functor 
$$\Hhf \to \Eh,~F\mapsto [X,F],$$ 
is pro-representable in $\Hhf$. The representing pro-object $\hat{X}^{\mathrm{AM}}\in \mathrm{pro}-\Hhf$ is called the (Artin-Mazur) profinite completion of $X$. 
Then Sullivan showed in \cite{sullivan} that the underlying diagram in $\Hhf$ of $\hat{X}^{\mathrm{AM}}$ has a limit $\hat{X}^{\mathrm{Su}}$ in $\Hh$, which is called the profinite completion of Sullivan (see also \cite{ensprofin}).\\ 
The previously constructed fibrant replacement functor in $\hSh$ for the model structure of Theorem \ref{modelstructure} provides a rigid model of the Artin-Mazur and Sullivan profinite completion and extends it to nonconnected spaces. This rigidification is obtained as follows.\\
Given a profinite space $Z$, we have seen above that we can decompose it into an inverse system of fibrant profinite spaces with finite homotopy groups. Together with taking finite Postnikov sections, this yields a decomposition of $Z$ into an inverse system of finite spaces in $\Sh$. By applying this to the completion $\hat{X}$ of a space $X$, we get a functor 
$$T: \Sh \to \hSh \to \mathrm{pro}-\Sh_{\mathrm{fin}}$$ 
where $\Sh_{\mathrm{fin}}$ is the full subcategory of $\Sh$ of finite spaces. Now we can consider this functor on the homotopy level 
$$T: \Hh \to \hHh \to \mathrm{pro}-\Hh_{\mathrm{fin}}.$$ 
It follows immediately from the previous results on profinite spaces and \cite{artinmazur}, Theorem 4.3, that $T|_{\Hh_0}$ is isomorphic to the Artin-Mazur completion functor. Moreover, this implies that the fibrant replacement of $\hat{X}$ in $\hSh$ is a rigid model for the Sullivan completion of $X$, i.e. that $|R\hat{X}|$ is canonically isomorphic to $\hat{X}^{\mathrm{Su}}$ in $\Hh$. Hence $|R\hat{X}|$ provides a rigid model for profinite completion. \\
In \cite{ensprofin} Morel proved that there is a model structure on $\hSh$ for each prime number $p$ in which the weak equivalences are the maps that induce isomorphisms on continuous $\Z/p$-cohomology. The generating fibrations are given by the canonical maps $L(\Z/p,n) \to K(\Z/p,n+1)$, $K(\Z/p,n)\to \ast$, and trivial fibrations are given by the maps $L(\Z/p,n)\to \ast$ for every $n\geq 0$ (see also \cite{goersscomp} and \cite{profinhom}). The fibrant replacement functor $R^p$ of \cite{ensprofin} yields a rigid version of the pro-$p$-finite-completion of Artin-Mazur and Sullivan. \\ 
An extended discussion of this comparison and further aspects of profinite completion of spaces can be found in \cite{completion}.
\subsection{Completion of spaces versus completion of groups}
Beside the completion functor of spaces $\compl: \Sh \to \hSh$ there is a well-known profinite completion functor for discrete groups. For lack of better notation we will also denote the profinite completion of a group $G$ by $\hat{G}$. It is equipped with a natural homomorphism $G\to \hat{G}$ which is universal among continuous homomorphisms from $G$ to profinite groups.\\ 
Given a pointed space $X\in \Shp$, the homotopy groups of Definition \ref{defhomotopygroups} of its profinite completion $\hat{X}\in \hShp$ are profinite groups. Hence the induced map $\pi_tX \to \pi_t\hat{X}$ factors through the group completion of $\pi_tX$, i.e. there is a commutative diagram
$$\xymatrix{
\pi_tX\ar[dr] \ar[rr] & &\pi_t\hat{X}\\ 
  & \widehat{\pi_tX}\ar[ur]_{\varphi_t}. & }
 $$
It is a fundamental question if completion of spaces and groups commute. Unfortunately, $\varphi_t$ is not an isomorphism in general. A similar phenomenon is well-known for group completion and cohomology. In \cite{serre}, this led Serre to call a group $G$ good if the induced map $\psi:H^*(\hat{G};M)\to H^*(G;M)$ between continuous and discrete group cohomology is an isomorphism for every finite $G$-module $M$. It turns out that after modifying slightly the notion of good groups by considering the action of the fundamental group on the higher homotopy groups, one gets a sufficient condition such that the completion of spaces commutes with the one of groups.\\ 
Following \cite{sullivan}, for a pointed space $X$, we call $\pi_1:=\pi_1X$ a good fundamental group, if it is a good group in the sense of Serre above such that  
$$H^i(\hat{\pi}_1;M)\cong H^i(\pi_1;M)$$ 
is a finite group for all finite $\pi_1$-modules $M$ and all $i\geq0$. \\
Let $\pi_n:=\pi_nX$, $n \geq 2$, be a higher homotopy group of $X$. It carries a canonical action of $\pi_1$. Let $\Ph$ be the filtered set of finite $\pi_1$-quotients of $\pi_n$. We denote by $\hat{\pi}^{\pi_1}_n:=\lim_{Q\in\Ph}Q$ the $\pi_1$-completion of $\pi_n$. This is, in particular, a profinite group on which $\pi_1$ acts continuously. The $\pi_1$-module $\pi_n$ is called a good higher homotopy group if
$$H^i(\pi_n;A)\cong H^i(\hat{\pi}^{\pi_1}_n;A)$$ 
is an isomorphism of finite groups for all finite coefficient groups $A$ and all $i\geq0$. With these definitions we reinterpret the following result of Sullivan \cite{sullivan}, Theorem 3.1.
\begin{theorem}\label{profinitecompletion}
Let $X$ be a connected pointed space.\\ 
(a) The canonical map $\varphi_1: \widehat{\pi_1X} \to \pi_1\hat{X}$ is an isomorphism of profinite groups, i.e. $\pi_1X \to \pi_1\hat{X}$ equals group completion.\\
(b) If $X$ has a good fundamental group and good higher homotopy groups, then $\varphi_t: \widehat{\pi_tX} \to \pi_t\hat{X}$ is an isomorphism of profinite groups for every $t$.
\end{theorem}
\begin{proof}
The first assertion follows immediately from the definition of the profinite fundamental group of a profinite space via profinite covering spaces with finite fibres and the analogous description of the usual fundamental group of $X$ (see \cite{profinhom}). The second assertion follows from Theorem 3.1 in \cite{sullivan} together with the fact discussed above that $|R\hat{X}|$ is a rigid model for the profinite completion of Sullivan, i.e.\,that $|R\hat{X}|$ is canonically isomorphic to $\hat{X}^{\mathrm{Su}}$ in $\Hh$. 
\end{proof}

%
\subsection{A concrete profinite replacement functor}
The result of Sullivan shows that profinite completion of spaces with finite homotopy groups induces a weak equivalence of underlying spaces. For a slightly smaller class of spaces, we have a more concrete description of a profinite model. \\ 
We recall that a space $X$ is called simple if it has an abelian fundamental group and if the action of the fundamental group on the higher homotopy groups is trivial. 
\begin{lemma}\label{finitereplacement}
Let $X$ be a connected simple pointed space whose homotopy groups are all finite. Then there is a pointed simplicial finite set which is fibrant in $\hShp$ and a pointed map $X \to FX$ which is a weak equivalence of underlying simplicial sets. In particular, it induces an isomorphism $\pi_*X \cong \pi_*FX$ of homotopy groups of the underlying simplicial sets.\\ 
The assignment $X\mapsto FX$ is functorial in the sense that given a map $g:X\to Y$ between connected simple pointed spaces with finite homotopy groups, there is a map $F(g)$ in $\hShp$ such that the following diagram of underlying simplicial sets commutes
$$\xymatrix{
X\ar[d] \ar[r]^g & Y \ar[d] \\
FX \ar[r]_{F(g)} & FY.}$$
\end{lemma}
\begin{proof}
We can assume that $X$ is a fibrant simplicial set. Let $\pi_n:=\pi_nX$ be the $n$th finite homotopy group of $X$. We define the pointed profinite space $FX$ as the limit in $\hShp$ of a specific Postnikov tower $\{X(n)\}_{n\geq 1}$ of $X$. Let $\cosk_nX\in\Shp$ be the $n$th coskeleton of $X$. 
It is equipped with a natural pointed maps $X\to \cosk_nX$ and $\cosk_{n+1}X\to \cosk_nX$ for each $n\geq 0$. The fibrant pointed space $X$ is isomorphic to the limit $\lim_n \cosk_nX$. 
Moreover, the map $\cosk_nX \to \cosk_{n-1}X$ sits in a homotopy pullback square in $\Shp$
\begin{equation}\label{cosknpb}
\xymatrix{
\cosk_nX \ar[d] \ar[r] & WK(\pi_n,n) \ar[d]^{q_n} \\
\cosk_{n-1} \ar[r]_{k_n} & K(\pi_n,n+1).}
\end{equation} 
The map $q_n$ is the universal bundle over the simplicial group $K(\pi_n,n+1)$. For a simplicial group $\Gamma$, the contractible space $W\Gamma$ is defined by 
$$(W\Gamma)_n = \Gamma_n \times \Gamma_{n-1} \times \ldots \times \Gamma_0.$$ 
The map $k_n$ is the $k$-invariant defined by a class 
$$[k_n]\in H^{n+1}(\cosk_{n-1}X; \pi_n)$$
in the cohomology of $\cosk_{n-1}X$ with coefficients in $\pi_n$. \\
Now we define the pointed profinite spaces $X(n)$ together with pointed maps $j_n:\cosk_nX \to X(n)$ which are weak equivalences of underlying pointed simplicial sets. For $n=1$, we define $X(1):=B\pi_1X$. It is a pointed simplicial finite set and a fibrant object in $\hShp$. Choosing any weak equivalence $\cosk_1X \to B\pi_1X$ of simplicial sets provides a map $j_1: \cosk_1X \to X(1)$. \\
For $n\geq 2$, assume we are given $X(n-1)$ and a map $j_{n-1}:\cosk_{n-1} \to X(n-1)$. Up to homotopy there is a factorization in the category of pointed simplicial sets
$$\xymatrix{
\cosk_{n-1}X \ar[d]^{k_n} \ar[r]^{j_{n-1}} & X(n-1) \ar[dl]\\
K(\pi_n,n+1). & }
$$
The space $X(n)$ and the map $X(n) \to X(n-1)$ is then defined as the pullback of the diagram of simplicial finite sets
\begin{equation}\label{Xnpb}
\xymatrix{
X(n) \ar[d] \ar[r] &  WK(\pi_n,n) \ar[d]^{q_n} \\
X(n-1) \ar[r]_{k_n} & K(\pi_n,n+1).}
\end{equation} 
Since $\pi_n$ is a finite group, $K(\pi_n,n+1)$ and $WK(\pi_n,n)$ are finite in each degree. Moreover, the map  
$$q_n: WK(\pi_n,n) \to K(\pi_n,n+1)$$ 
is a fibration in $\hShp$ by \cite{completion}, Theorem 3.7. 
Hence the pullback of (\ref{Xnpb}) can be constructed in $\hShp$ and $X(n)$ is a simplicial finite set which is a fibrant object in $\hShp$.  
Since the map 
$$\cosk_nX \to WK(\pi_n,n) \times_{K(\pi_n,n+1)} \cosk_{n-1}X$$ 
is a weak equivalence, the induced map $j_n: \cosk_nX \to X(n)$ is a weak equivalence of underlying simplicial sets. 
We define $FX$ to be the limit
$$FX:=\lim_n X(n).$$ 
Since the finite set of $m$-simplices of $X(n)$ is isomorphic to the finite set of $m$-simplices of $X(n-1)$ for $m\leq n-1$, $FX$ is a simplicial finite set, and $FX$ is a fibrant object in $\hShp$, since it is the limit of a tower of fibrations in $\hShp$.  
Since the two maps $X\to \lim_n \cosk_nX$ and $\lim_n \cosk_nX \to \lim_n X(n)$ are weak equivalences of underlying simplicial sets, the associated map $\varphi:X \to FX$ is a weak equivalence of underlying simplicial sets and hence induces an isomorphism $\pi_*X \cong \pi_*FX$. Finally, we deduce the last assertion of the lemma from the fact that all the constructions involved in order to define $FX$ are functorial. 
\end{proof}

Lemma \ref{finitereplacement} provides a functorial construction $F$ that associates a weakly equivalent profinite space to a connected simple pointed space with finite homotopy groups. Since we also want to apply this construction to spectra, we need to understand its behavior with respect to taking loop spaces.
\begin{lemma}\label{Floop}
{\rm (a)} Let $X$ and $Y$ be connected simple pointed spaces with finite homotopy groups and let $f:X\to Y$ be a weak equivalence of pointed simplicial sets. Then the induced map $F(f)$ is a weak equivalence in $\hShp$.\\
{\rm (b)} Let $W$ be a connected simple pointed space with finite homotopy groups. Then there is a natural weak equivalence $F(\Omega X)\stackrel{\sim}{\to} \Omega(FX)$ in $\hShp$.\\
{\rm (c)} If $V\to \Omega W$ is a weak equivalence between connected simple pointed spaces with finite homotopy groups, then the composite map $FV\to F(\Omega W)\to \Omega(FW)$ is a weak equivalence in $\hShp$.
\end{lemma}
\begin{proof}
(a) Let $\{X(n)\}_n$ and $\{Y(n)\}_n$ be the two towers of fibrations corresponding to $X$ and $Y$, respectively, used in the proof of Lemma \ref{finitereplacement} in order to construct $FX$ and $FY$. For each $n$, the map $X(n)\to Y(n)$ of simplicial finite sets induced by $f$ is a weak equivalence in $\Shp$ and in $\hShp$. 
Hence the map of limits of the towers of fibrations 
$$FX=\lim_nX(n) \to \lim_nY(n)=FY$$
is a weak equivalence in $\hShp$ as well.\\
(b) This follows now from the fact that the map $K(\pi, n) \to \Omega(K(\pi,n+1))$ is a weak equivalence in $\Shp$ and $\hShp$ for every finite group $\pi$ and every $n$.\\
(c) This is a consequence of (a) and (b).
\end{proof}

\subsection{Profinite $G$-spaces}
Let $G$ be a fixed profinite group. Let $S$ be a profinite set on which $G$ acts continuously, i.e.\, the group $G$ is acting on $S$ via a continuous map $\mu:G \times S \to S$. In this situation we say that $S$ is a profinite $G$-set. If $X$ is a profinite space and $G$ acts continuously on each $X_n$ such that the action is compatible with the face and degeneracy maps, then we call $X$ a profinite $G$-space. We denote by $\hShg$ the category of profinite $G$-spaces. The morphism are given by $G$-equivariant maps of profinite spaces. \\ 
While a discrete $G$-space $Y$ is characterized as the colimit over the fixed point spaces $Y^U$ over all open subgroups, a profinite $G$-space X is the limit over its orbit spaces $X_U$. More explicitly, for an open subgroup $U$ of $G$, let $X_U$ be the quotient space under the action by $U$, i.e.\,the quotient $X/\sim$ with $x \sim y$ in $X$ if both are in the same orbit under $U$. It is easy to show that the canonical map $\phi:X \to \lim_U X_U$ is a homeomorphism, where $U$ runs through the open subgroups of $G$.\\
Moreover, we have the following lemma which was proven in \cite{completion}, Lemma 4.2. 
\begin{lemma}\label{5.6.4}
Let $G$ be a profinite group and $X$ a profinite $G$-space. Then $X$ has a decomposition 
$$X =\lim_{Q_G \in \Rh_G(X)} X/Q_G$$
in $\hShg$ as an inverse limit of simplicial finite $G$-quotient sets  where $\Rh_G(X)$ denotes the set of $G$-invariant simplicial open equivalence relations of $X$.
\end{lemma}
%
The category of profinite $G$-spaces has a simplicial structure. Let $X$ and $Y$ be profinite $G$-spaces. The mapping space $\map_{\hShg}(X,Y)$ is defined as the simplicial set whose set of $n$-simplices is given as the set of $G$-equivariant maps 
$$\map_{\hShg}(X,Y)_n=\Hom_{\hShg}(X\times \Delta[n], Y)$$
where $\Delta[n]$ has trivial $G$-action. This defines a functor 
$$\map_{\hShg}(-,-):\hShg^{\mathrm{op}}\times \hShg \to \Sh.$$
Let $K$ be a finite simplicial set, and $X$ a profinite $G$-space. We consider $K$ as a profinite $G$-space with trivial $G$-action. The tensor object $X\otimes K\in \hShg$ is defined as the product $X\times K$ in $\hShg$ with the diagonal $G$-action. The function object in $\hShg$ is defined as the profinite $G$-space $\hom_{\hShg}(K,X)\in \hShg$ whose set of $n$-simplices is given by the profinite set of maps
$$\hom_{\hShg}(K,X)_n = \Hom_{\hSh}(K\times \Delta[n], X)$$
on which $G$ acts continuously via its continuous action on the target $X$.\\ 
If $K$ is an arbitrary simplicial set, isomorphic to the colimit $\colim_{\alpha}K_{\alpha}$ of its finite simplicial subsets $K_{\alpha}$, and $X$ a profinite $G$-space, the tensor object $X\otimes K$ is the colimit in $\hShg$ of the profinite $G$-spaces $X\times K_{\alpha}$. 
The function object $\hom_{\hShp}(K,X)$ is defined to be the limit in $\hShp$ of the pointed profinite spaces $\hom_{\hShp}(K_{\alpha},X)$. \\
Let $X$ and $Y$ be profinite $G$-spaces and let $K$ be a simplicial set. Then mapping space, tensor and function objects are connected by the natural bijections
$$\map_{\hShg}(X\otimes K, Y) \cong \map_{\Sh}(K, \map_{\hShg}(X,Y))$$
and 
$$\map_{\hShg}(Y, \hom_{\hShg}(K, X)) \cong \map_{\Sh}(K, \map_{\hShg}(X,Y)).$$
This defines the structure of a simplicial category on $\hShp$.\\
In order to get a model structure on $\hShg$ one can find explicit sets of generating fibrations and trivial fibrations. They arise naturally by considering continuous $G$-actions on the corresponding generating sets for the model structure on $\hSh$. The following result was proven in \cite{gspaces}.
\begin{theorem}
There is a fibrantly generated left proper simplicial model structure on the category of profinite $G$-spaces such that a map $f$ is a weak equivalence (fibration) in $\hShg$ if and only if its underlying map is a weak equivalence (fibration) in $\hSh$. A map $f:X\to Y$ is a cofibration in $\hShg$ if and only if $f$ is a levelwise injection, and the action of $G$ on $Y_n -f(X_n)$ is free for each $n \geq 0$. We denote its homotopy category by $\hHhg$. 
\end{theorem}
Let $\hShpg$ be the category of pointed profinite $G$-spaces. The objects of $\hShpg$ are profinite $G$-spaces that are equipped with a base point that is fixed under $G$. The morphism in $\hShpg$ are the morphisms of profinite $G$-spaces that preserve the base points. Let $X$ and $Y$ be pointed profinite $G$-spaces. The smash product $X\wedge Y$ is again a pointed profinite $G$-space on which $G$ acts via the diagonal action. The mapping space $\map_{\hShp}(X,Y)$ is defined as the simplicial set whose set of $n$-simplices is given as the set of maps 
$$\map_{\hShpg}(X,Y)_n=\Hom_{\hShpg}(X\wedge \Delta[n]_+, Y)$$
where $\Delta[n]_+$ is considered as a pointed profinite $G$-space with trivial $G$-action. 
This defines a functor 
$$\map_{\hShpg}(-,-):\hShpg^{\mathrm{op}}\times \hShpg \to \Sh.$$
Let $K$ be a finite simplicial set and $X$ a pointed profinite $G$space. The tensor object $X\otimes K\in \hShpg$ is defined as the smash product $X\wedge K_+$ where $K_+$ is considered as a pointed profinite $G$-space with trivial $G$-action. The function object in $\hShpg$ is defined as the pointed profinite $G$-space $\hom_{\hShpg}(K,X)\in \hShpg$ whose set of $n$-simplices is given by the profinite set of maps
$$\hom_{\hShpg}(K,X)_n = \Hom_{\hShp}(K_+\wedge \Delta[n]_+, X)$$
on which $G$ acts continuously via its action on the target $X$.\\ 
If $K$ is an arbitrary simplicial set isomorphic to the colimit $\colim_{\alpha}K_{\alpha}$ of its finite simplicial subsets $K_{\alpha}$, and $X$ a pointed profinite $G$-space, we define the tensor object $X\otimes K$ to be the colimit in $\hShpg$ of the profinite $G$-spaces $X\wedge K_{\alpha}$. 
The function object $\hom_{\hShpg}(K,X)$ is defined to be the limit in $\hShpg$ of the pointed profinite $G$-spaces $\hom_{\hShpg}(K_{\alpha},X)$. \\
Let $X$ and $Y$ be pointed profinite spaces and let $K$ be a simplicial set. Then mapping space, tensor and function objects are connected by the natural bijections
$$\map_{\hShpg}(X\otimes K, Y) \cong \map_{\Sh}(K, \map_{\hShpg}(X,Y))$$
and 
$$\map_{\hShpg}(Y, \hom_{\hShpg}(K, X)) \cong \map_{\Sh}(K, \map_{\hShpg}(X,Y)).$$
Again, if the finite simplicial set $K$ is already equipped with a base point and $X$ is a pointed profinite $G$-space, we also denote by $\hompg(K,X)\in \hShpg$ the pointed profinite $G$-space whose set of $n$-simplices is given by the profinite set of maps
$$\hom_{\hShpg}(K,X)_n = \Hom_{\hShp}(K\wedge \Delta[n]_+, X)=\lim_{\alpha}\Hom_{\hShp}(K_{\alpha}\wedge \Delta[n]_+, X)$$
on which $G$ acts continuously via its action on the target. 
%
%
\begin{example}\label{exGomega}
We consider the simplicial circle $S^1$ as a pointed simplicial finite set with trivial $G$-action. Taking the smash product with $S^1$ defines a functor $\hShpg \to \hShpg$, $X\mapsto S^1 \wedge X$. It is left adjoint to the functor $\hShpg \to \hShpg$ defined by sending a pointed profinite $G$-space $X$ to the function object $\hom_{\hShpg}(S^1,X)$ in $\hShpg$. In particular, the profinite loop space $\Omega X$ inherits form $X$ the structure as a pointed profinite $G$-space. 
\end{example}
As in the nonequivariant case, the category of pointed profinite $G$-spaces inherits a model structure from $\hShg$, since $\hShpg$ is an under-category of $\hShg$.
\begin{theorem}\label{Gmodelunstable}
The category $\hShpg$ has the structure of a fibrantly generated left proper simplicial model category for which a map is a weak equivalence (cofibration, fibration) if and only if its underlying unpointed map is a weak equivalence (cofibration, fibration) in $\hShg$. We denote the homotopy category by $\hHhpg$.
\end{theorem}
\begin{remark}\label{Gcylinder}
Let $X$ be a cofibrant pointed profinite $G$-space. As discussed in Remark \ref{cylinder}, the pointed profinite $G$-space $X\wedge \Delta[1]_+$ is a cylinder object for $X$ in $\hShpg$. For the maps $X\to X\wedge \Delta[1]_+$ are trivial cofibrations in $\hShpg$, and the map $X \wedge \Delta[1]_+ \to X$ is a weak equivalence in $\hShpg$. Hence if $X$ is a cofibrant profinite $G$-space and $Y$ is a fibrant pointed profinite $G$-space, there is a natural bijection 
$$\pi_0\map_{\hShpg}(X,Y) \cong \Hom_{\hHhpg}(X,Y).$$
\end{remark}
%
%
Since weak equivalences and fibrations in $\hShg$ are determined by their underlying maps in $\hSh$, one can hope that the fibrant replacement functor constructed for $\hSh$ above also yields a fibrant replacement in $\hShg$. In fact, this is the case. The point is that the loop groups and loop groupoids are free. If $G$ acts continuously on a profinite set $S$, then $G$ also acts  continuously on the free profinite group on this set. Hence the construction of the fibrant replacement functor in $\hShg$ is the same as for $\hSh$. We just use Lemma \ref{5.6.4} above to decompose a profinite $G$-space $X$ into a limit of simplicial finite $G$-sets. With a suitable notion of continuous group action on a groupoid we obtain a $G$-equivariant fibrant replacement in $\hShg$ (for more details see \cite{completion}, \S 4.4).\\
Since we will be working with pointed profinite $G$-spaces, we would also like to have a fibrant replacement in $\hShpg$. For a pointed profinite $G$-space $X\in \hShpg$, the basepoint $p:\ast \to X$ and the decomposition of $X$ as a limit $X\lim_{Q_G}X/Q_G$ as pointed simplicial finite $G$-sets provides each $X/Q_G$ with a choice of basepoint which is compatible in the system $\{X/Q_G\}$. Then \cite{completion}, Theorem 4.12, implies the following pointed version.
\begin{theorem}\label{Gfibrep}
Let $G$ be a profinite group and let $X$ be a pointed profinite $G$-space. The map 
$$\eta: X \to \bar{W}\hat{\Gamma}X=\lim_{Q_G \in \Rh_G(X)} \bar{W}\hat{\Gamma}(X/Q_G)$$ 
is a weak equivalence in $\hShpg$. The pointed profinite $G$-space $\bar{W}\hat{\Gamma}X$ is fibrant in $\hShpg$. Hence $\eta$ defines a functorial fibrant replacement in $\hShpg$. Moreover, $X$ is weakly equivalent in $\hShpg$ to a limit of simplicial finite $G$-sets which are fibrant in $\hShpg$
$$X \simeq \lim_{Q_G,U} \bar{W}(\Gamma(X/Q_G))/U.$$ 
\end{theorem}
\subsection{Homotopy groups of pointed profinite $G$-spaces}
\begin{defn}\label{Gdefhomotopygroups}
Let $X$ be a pointed profinite $G$-space and let $R_GX$ be a fibrant replacement of $X$ in the model structure on $\hShpg$. Then we define the $n$th profinite homotopy group of $X$ for $n \geq 2$ to be the group
$$\pi_n(X):=\pi_0(\Omega^n(R_GX)).$$
\end{defn}
\begin{remark}
Let $X$ be a pointed profinite $G$-space. Since weak equivalences and fibrations in $\hShpg$ are determined by the underlying maps in $\hShp$, the underlying homotopy groups of $X$ of Definition \ref{Gdefhomotopygroups} coincide with the homotopy groups of Definition \ref{defhomotopygroups} of $X$ after forgetting the $G$-action on $X$. 
Moreover, the homotopy groups of Definition \ref{Gdefhomotopygroups} coincide with the homotopy groups of the underlying fibrant pointed simplicial set $|R_GX|$.  
\end{remark}
\begin{prop}\label{Gmodulepi}
Let $X$ be a pointed profinite $G$-space. The set of connected components of $X$ is a profinite $G$-set. The profinite fundamental group of $X$ is a profinite $G$-group. The $n$th homotopy group of $X$ is a profinite $G$-module for $n\geq 2$.  
\end{prop}
\begin{proof}
Let $R_G$ be the functorial fibrant replacement functor of Theorem \ref{Gfibrep}. 
Since $R_GX$ is the limit $\lim_{\beta}(R_GX)_{\beta}$ of simplicial finite $G$-sets which are fibrant in $\hShpg$, we obtain that
$$\hompg(S^1, R_GX)=\hompg(S^1, \lim_{\beta} (R_GX)_{\beta}) = \lim_{\beta} \hompg(S^1, (R_GX)_{\beta})$$
is a cofiltered limit of fibrant simplicial finite $G$-sets. Moreover, the underlying profinite space of $\hompg(S^1, R_GX)$ is equal to the fibrant profinite space $\homp(S^1, R_GX)$. Hence in order to prove the assertion, it suffices by Proposition \ref{pi0} to show the assertions for a simplicial finite $G$-set which is fibrant in $\hShpg$. \\
So we can assume that $X$ is a simplicial finite $G$-set which is fibrant in $\hShpg$. The set $X_0$ of $0$-simplices and the set $X_1$ of $1$-simplices of $X$ are finite discrete $G$-sets. Hence they can be written as colimits of the fixed points under the open subgroups of $G$, i.e. 
$$X_0 = \colim_U X_0^U, ~ X_1 = \colim_U X_1^U$$
where $U$ runs through the open subgroups of $G$. The set $\pi_0X$ of connected components is defined as the colimit of the diagram 
$$\xymatrix{
X_1 \ar[d]_{d_0} \ar[r]^{d_1} & X_0 \ar[d] \\
X_0 \ar[r] & \pi_0X.}$$
Since $\pi_0$ commutes with filtered colimits, we also get $\pi_0X=\colim_U \pi_0(X^U)$. Since $X$ is a fibrant simplicial finite set and since the fixed points can be viewed as a cofiltered limit, Proposition \ref{pi0} shows that we have $\pi_0(X^U)\cong (\pi_0X)^U$. Hence we obtain 
$$\pi_0X=\colim_U (\pi_0X)^U,$$
i.e. that $G$ acts continuously on the finite discrete set $\pi_0X$. Since $\hompg(S^1, X)$ is a fibrant simplicial finite $G$-set if $X$ is a fibrant simplicial finite $G$-set, the assertions on the fundamental group and the higher homotopy groups follow from the result on $\pi_0$. 
\end{proof}
\subsection{Homotopy limits of pointed profinite $G$-spaces}
Homotopy limits of diagrams of pointed profinite $G$-spaces can be constructed in the same way as for pointed profinite spaces. Let $I$ be a small category and let $X(-):I\to \hShpg$ be a functor from $I$ to pointed profinite $G$-spaces. 
We define the pointed profinite $G$-space $\hompg(B(I/-),X(-))$ as the equalizer in $\hShpg$ of the diagram of pointed profinite $G$-spaces
$$\prod_{i\in I} \hompg(B(I/i),X(i)) \rightrightarrows \prod_{\alpha:i\to i' \in I} \hompg(B(I/i),X(i')).$$
%
\begin{defn}\label{Gdefholim}
Let $I$ be small category and $X(-): I \to \hShpg$ be a functor. The homotopy limit of $X(-)$ is defined to be the pointed profinite $G$-space 
$$\holim_{i\in I} X(i):=\hompg(B(I/-),X(-)).$$ 
\end{defn}
%
\begin{lemma}\label{Gholiminv}
Let $X(-)$ be a functor from a small category $I$ to pointed profinite $G$-spaces. If $X(i)$ is fibrant for every $i\in I$, then $\holim_{i\in I}X(i)$ is fibrant in $\hShpg$.\\
Let $f(-):X(-) \to Y(-)$ be a natural transformation of functors from the small category $I$ to the full subcategory of fibrant pointed profinite $G$-spaces. Let 
$$f:=\holim_{i\in I}f(i): \holim_{i\in I}X(i) \to \holim_{i\in I}Y(i)$$
be the induced map on homotopy limits. If $f(i): X(i) \to Y(i)$ is a weak equivalence in $\hShpg$ for every $i\in I$, then $f$ is a weak equivalence in $\hShpg$.
\end{lemma}
\begin{proof}
As for pointed profinite spaces, this follows from  \cite{hirsch}, Theorems 19.4.1 and Theorem 19.4.2. 
\end{proof}
%
%
%
\begin{lemma}\label{Gholimlemma}
Let $X(-): I \to \hShpg$ be a small cofiltering diagram such that each $X(i)$ is connected and fibrant in $\hShpg$. Then the natural isomorphism of homotopy groups of underlying simplicial sets
$$\pi_q(|\holim_{i\in I} X(i)|) \cong \lim_{i\in I} \pi_q(|X(i)|)$$
of Lemma \ref{holimlemma} is an isomorphism of profinite $G$-groups for $q\geq 0$.  
\end{lemma}
\begin{proof}
The construction of the spectral sequence computing the homotopy groups of the homotopy limit is functorial. Hence the isomorphism of Lemma \ref{holimlemma} for the underlying diagram of pointed profinite spaces is in fact a $G$-equivariant isomorphism of profinite $G$-groups. 
\end{proof}

\subsection{A concrete equivariant profinite replacement}  
In this section we discuss a functor that sends certain $G$-spaces to profinite $G$-spaces. We call a simplicial set $X$ a $G$-space, if $X$ is a simplicial object in the category of $G$-sets for the abstract underlying group of $G$. In order to simplify the argument and in order to make the construction more concrete we will make two assumptions, on the one hand that the profinite group $G$ is strongly complete and on the other hand that the spaces have finite homotopy groups. \\
A profinite group $G$ is called strongly complete in \cite{ribes}, if every subgroup of finite index is  open in $G$, or, in other words, if $G= \hat{G}$ as profinite groups. The profinite completion of an abstract group is itself strongly complete. But in general there are subgroups of finite index which are not open in the given topology. The significance of this property for us, is that for a strongly complete profinite group $G$, every finite set $S$ with a $G$-action is a continuous  discrete $G$-set. \\
Serre conjectured that every topologically finitely generated profinite group $G$ is strongly complete, where topologically finitely generated means that $G$ contains a dense subgroup which is finitely generated as a group. He showed this conjecture for finitely generated pro-$p$-groups. Recently, Nikolov and Segal proved the full conjecture for every finitely generated profinite group in \cite{niksegal}. The proof relies on the classification of finite simple groups. The most important examples for us are provided by the extended Morava stabilizer group $G_n$ and all of its closed subgroups. For $G_n$ is a finitely generated profinite group and hence strongly complete. \\
The following lemma is an equivariant version of Lemma \ref{finitereplacement} which is of fundamental importance for this paper. 
\begin{lemma}\label{Gfinitereplacement}
Let $G$ be a strongly complete profinite group and $X$ be a connected simple pointed $G$-space whose homotopy groups are all finite. Then there is a simplicial finite $G$-set $F_GX$ which is a fibrant object in $\hShpg$ and a
natural $G$-equivariant pointed map $\varphi:X \to F_GX$ which is a weak equivalence of underlying simplicial sets. In particular, it induces an isomorphism $\pi_*X \cong \pi_*F_GX$ of homotopy groups of the underlying simplicial sets.  \\
The assignment $X\mapsto F_GX$ is functorial in the sense that given a map $h:X\to Y$ between connected simple pointed simplicial $G$-sets with finite homotopy groups, there is a map $F(h)$ in $\hShpg$ such that the following diagram of underlying simplicial $G$-sets commutes
$$\xymatrix{
X\ar[d] \ar[r]^h & Y \ar[d] \\
F_GX \ar[r]_{F(h)} & F_GY.}$$
\end{lemma}
\begin{proof}
Let $X\to X'$ be a functorial fibrant replacement in $\Shp$. The map $X\to X'$ is $G$-equivariant and a weak equivalence of underlying simplicial sets. Hence we can assume that the underlying simplicial set of $X$ is a fibrant object in $\Shp$. \\ 
Let $\pi_n:=\pi_nX$ be the $n$th finite homotopy group of $X$. Since $X$ is a $G$-space, each $\pi_n$ is a $G$-module. Moreover, since $G$ is strongly complete, $G$ acts continuously on each finite discrete group $\pi_n$. We define the profinite $G$-space $F_GX$ again as the limit in $\hShpg$ of a specific Postnikov tower $\{X(n)\}_{n\geq 1}$ of $X$. Let $\cosk_nX\in\Shp$ be the $n$th coskeleton of $X$. Since $\cosk_n$ is a functor, the space $\cosk_nX$ is a simplicial $G$-set and the associated natural pointed maps $X\to \cosk_nX$ and $\cosk_{n}X\to \cosk_{n-1}X$ are $G$-equivariant for each $n\geq 0$.  
Moreover, the map $\cosk_nX \to \cosk_{n-1}X$ sits in a homotopy pullback square of $G$-equivariant maps
\begin{equation}\label{Gcosknpb}
\xymatrix{
\cosk_nX \ar[d] \ar[r] &  WK(\pi_n,n) \ar[d]^{q_n} \\
\cosk_{n-1}X \ar[r]_{k^G_n} & K(\pi_n,n+1).}
\end{equation} 
The map $k^G_n$ is the $G$-equivariant $k$-invariant defined by a class 
$$[k_n^G]\in H^{n+1}_G(\cosk_{n-1}X; \pi_n)$$
in the $G$-equivariant cohomology of $\cosk_{n-1}X$ (see e.g. \cite{goerss}, p. 207-208). \\
Now we define profinite $G$-spaces $X(n)$ together with $G$-equivariant pointed maps $j_n:\cosk_nX \to X(n)$ which are weak equivalences of underlying pointed simplicial sets. For $n=1$, we define $X(1):=B\pi_1X$. It is a pointed simplicial finite $G$-set and a fibrant object in $\hShpg$. Choosing any $G$-equivariant map $\cosk_1X \to B\pi_1X$ which is a weak equivalence of simplicial sets provides a map $j_1: \cosk_1X \to X(1)$. (This $G$-equivariant choice is possible, since the category of simplicial $G$-sets, for the underlying abstract group of $G$, can be equipped with a model structure in which weak equivalences and fibrations are determined by their underlying non-equivariant maps.)\\ 
For $n\geq 2$, assume we are given $X(n-1)$ and a map $j_{n-1}:\cosk_{n-1} \to X(n-1)$. Up to $G$-equivariant homotopy there is a factorization in the category of pointed simplicial $G$-sets
$$\xymatrix{
\cosk_{n-1}X \ar[d]^{k_n^G} \ar[r]^{j_{n-1}} & X(n-1) \ar[dl]\\
K(\pi_n,n+1). & }
$$
The space $X(n)$ and the map $X(n) \to X(n-1)$ is then defined as the pullback of the diagram of simplicial finite $G$-sets
\begin{equation}\label{GXnpb}
\xymatrix{
X(n) \ar[d] \ar[r] &  WK(\pi_n,n) \ar[d]^{q_n} \\
X(n-1) \ar[r]_{k_n} & K(\pi_n,n+1).}
\end{equation} 
Since $\pi_n$ is a finite $G$-module, the spaces $K(\pi_n,n+1)$ and $WK(\pi_n,n)$ are simplicial finite $G$-sets. Moreover, since $G$ is strongly complete, the action of $G$ on all the finite sets involved is continuous and the map  
$$q_n: WK(\pi_n,n) \to K(\pi_n,n+1)$$ 
is a fibration in $\hShpg$ by Theorem \ref{Gmodelunstable}, and \cite{completion} Proposition 3.7. 
Hence the pullback of (\ref{GXnpb}) can be constructed in $\hShpg$ and $X(n)$ is a simplicial finite discrete $G$-set which is a fibrant object in $\hShpg$.  
Since the map 
$$\cosk_nX \to WK(\pi_n,n) \times_{K(\pi_n,n+1)} \cosk_{n-1}X$$ 
is a weak equivalence, we obtain an induced weak equivalence $j_n: \cosk_nX \to X(n)$ of underlying simplicial sets. 
Now we can define $F_GX$ to be the limit
$$F_GX:=\lim_n X(n).$$
Since the finite discrete $G$-set of $m$-simplices of $X(n)$ is isomorphic to the finite discrete $G$-set of $m$-simplices of $X(n-1)$ for $m\leq n-1$, $F_GX$ is a simplicial object of finite discrete $G$-sets. Moreover, $F_GX$ is a fibrant object in $\hShpg$, since it is the limit of a tower of fibrations in $\hShpg$. Furthermore, since the $G$-equivariant maps $X\to \lim_n \cosk_nX$ and $\lim_n \cosk_nX \to \lim_n X(n)$ are weak equivalences of underlying simplicial sets, the associated map $\varphi:X \to F_GX$, which is functorial and hence $G$-equivariant, is a weak equivalence of underlying simplicial sets. In particular, it induces an isomorphism $\pi_*X \cong \pi_*F_GX$. \\
The functoriality follows as in the non-equivariant case from the fact that all constructions used to define $F_GX$ are functorial.
\end{proof}
As in the non-equivariant case, we have to know how $F_G$ behaves with respect to taking loop spaces. 
\begin{lemma}\label{GFloop} 
Let $G$ be a strongly complete profinite group. \\
{\rm (a)} Let $X$ and $Y$ be connected simple pointed simplicial $G$-sets with finite homotopy groups and let $f:X\to Y$ be a $G$-equivariant map that is a weak equivalence of underlying pointed simplicial sets. Then the induced map $F_G(f)$ is a weak equivalence in $\hShpg$.\\
{\rm (b)} Let $W$ be a connected simple pointed simplicial $G$-set with finite homotopy groups. Then there is a natural weak equivalence $F_G(\Omega X)\stackrel{\sim}{\to} \Omega(F_GX)$ in $\hShpg$.\\
{\rm (c)} If $V\to \Omega W$ is a $G$-equivariant map between connected simple pointed simplicial $G$-sets with finite homotopy groups which is a weak equivalence of underlying simplicial sets, then the composite map $F_GV\to F_G(\Omega W)\to \Omega(F_GW)$ is a weak equivalence in $\hShpg$.
\end{lemma}
\begin{proof}
Since weak equivalences in $\hShpg$ are determined by their underlying maps in $\hShp$, the lemma follows from the corresponding results in Lemma \ref{Floop} and the construction of $F_G$.
\end{proof}

%
\section{Profinite $G$-spectra}
\subsection{Profinite spectra}
Now we want to stabilize the model structure of profinite spaces. Since the simplicial circle $S^1=\Delta^1/\partial \Delta^1$ is a simplicial finite set and hence an object in $\hShp$, we may stabilize $\hShp$ by considering sequences of pointed profinite spaces together with connecting maps for the suspension. 
\begin{defn} 
A profinite spectrum $X$ consists of a sequence of pointed profinite spaces $X_n \in \hShp$ and maps $\sigma_n:S^1 \wedge X_n \to X_{n+1}$ in $\hShp$ for $n\geq0$. A morphism $f:X \to Y$ of profinite spectra consists of maps $f_n:X_n \to Y_n$ in $\hShp$ for $n\geq0$ such that $\sigma_n(1\wedge f_n)=f_{n+1}\sigma_n$. We denote by $\hSp$ the corresponding category of profinite spectra.
\end{defn}
As for profinite spaces, the word ``profinite'' in the term ``profinite spectrum'' does not mean that we look at inverse systems of finite spectra in the usual sense. Instead we look at spectra that are built out of simplicial profinite sets. \\ 

The category of profinite spectra is a simplicial category. Let $X$ and $Y$ be profinite spectra. The mapping space $\map_{\hSp}(X,Y)$ is defined as the simplicial set whose set of $n$-simplices is given as the set of maps 
$$\map_{\hSp}(X,Y)_n=\Hom_{\hSp}(X\wedge \Delta[n]_+, Y)$$
where $X\wedge\Delta[n]_+$ is the levelwise smash product of $X$ with the pointed simplicial finite set $\Delta[1]_+$. This defines a functor 
$$\map_{\hSp}(-,-):\hSp^{\mathrm{op}}\times \hSp \to \Sh.$$
Let $K$ be a simplicial set, and $X$ a profinite spectrum. The tensor object $X\otimes K\in \hSp$ is defined as the spectrum whose $n$th pointed profinite space is $X_n\wedge K_+$. The function object in $\hSp$ is defined as the profinite spectrum $\hom_{\hSp}(K,X)\in \hSp$ whose $n$th pointed profinite space is given by 
$$(\hom_{\hSp}(K,X))_n = \hom_{\hShp}(K, X_n).$$ 
Recall that the structure $S^1\wedge X_n \to X_{n+1}$ is determined by its adjoint 
$$X_n \to \Omega X_{n+1}.$$ 
Hence the structure map of $\hom_{\hSp}(K,X)$ is determined by the map
$$\hom_{\hShp}(K, X_n) \to \hom_{\hShp}(K, \Omega X_{n+1}) \cong \Omega (\homp(K, X_{n+1}))$$
given by adjunction (\ref{homhom}).\\
%
%
%
Let $X$ and $Y$ be profinite spectra and let $K$ be a simplicial set. We have the natural bijections
$$\map_{\hSp}(X\otimes K, Y) \cong \map_{\Sh}(K, \map_{\hSp}(X,Y))$$
and 
$$\map_{\hSp}(Y, \hom_{\hSp}(K, X)) \cong \map_{\Sh}(K, \map_{\hSp}(X,Y)).$$
%
If the simplicial set $K$ is already equipped with a base point and $X$ is a profinite spectrum, we also denote by $\hom_{\hSp}(K,X)\in \hSp$ the profinite spectrum whose $n$th space is $\hom_{\hShp}(K,X_n)$. \\
%
%
Observing that the functor $\hSp \to \hSp$, $X \mapsto S^1\wedge X$ is a left Quillen endofunctor, the stable model structure on $\hSp$ is constructed as in \cite{etalecob} in two steps as a dual version of the stabilization techniques in \cite{hovey}. First we need a projective model structure.
\begin{defn}\label{defstrict}
A map $f$ in $\hSp$ is a projective weak equivalence (projective fibration) if each map $f_n$ is a weak equivalence (fibration) in $\hShp$. A map $i$ is a projective cofibration if it has the left lifting property with respect to all projective trivial fibrations. 
\end{defn}
The following proposition is proved as in \cite{hovey}, Proposition 1.14.
\begin{prop}\label{projcofs}
A map  $i:A \to B$ in $\hSp$ is a projective (trivial) cofibration if and only if $i_0:A_0 \to B_0$ and the induced maps  $j_n:A_n \amalg_{S^1\wedge A_{n-1}} S^1\wedge B_{n-1}\to B_n$ for $n\geq 1$ are (trivial) cofibrations in $\hShp$. 
\end{prop}
\begin{prop}\label{thmstrict}
The projective weak equivalences, projective fibrations and projective cofibrations define a left proper fibrantly generated simplicial model structure on $\hSp$. 
\end{prop}
\begin{proof} 
That we obtain a left proper fibrantly generated model structure can be proven in essentially the same way as Theorem 1.13 in \cite{hovey}. In order to show the factorization axiom one uses a cosmall object argument and the fact that $\hShp$ is fibrantly generated. It remains to prove that this model structure is simplicial. We have defined mappings spaces and tensor and cotensor objects for $\hSp$ above. Let $i: A\to B$ be a cofibration of finite simplicial sets and $p: X\to Y$ a projective fibration in $\hSp$. We have to show that the induced map
$$(i^*, p_*): \hom_{\hSp}(B, X) \to \hom_{\hSp}(A, X) \times_{\hom_{\hSp}(A, Y)} \hom_{\hSp}(B, Y)$$
is a projective fibration in $\hSp$, which is trivial if either $i$ or $p$ is trivial. For $n\geq 0$, the $n$th map $(i^*, p_*)_n$ is given by the map of pointed profinite spaces
$$(i^*, p_{n*}): \hom_{\hShp}(B, X_n) \to \hom_{\hShp}(A, X_n) \times_{\hom_{\hShp}(A, Y_n)} \hom_{\hShp}(B, Y_n).$$
The model structure on $\hShp$ is simplicial. This implies that $(i^*, p_{n*})$ is a fibration, since $i$ is a cofibration and $p_n$ is a fibration. Moreover, $(i^*, p_{n*})$ is a weak equivalence if either $i$ or $p_n$ is a weak equivalence. Since projective weak equivalences and projective fibrations are determined levelwise, this shows that $(i^*, p_*)$ is a projective fibration which is a trivial projective fibration if either $i$ or $p$ is trivial.
\end{proof}

In the second step we enlarge the class of weak equivalences by localizing the projective model structure further. 
\begin{defn}
A profinite spectrum $E\in \hSp$ is called an $\Omega$-spectrum if each $E_n$ is fibrant in $\hShp$ and the adjoint structure maps 
$$E_n \to \Omega E_{n+1}=\hom_{\hShp}(S^1,E_{n+1})$$ 
are weak equivalences in $\hShp$ for all $n\geq 0$. \\
A map $f:X\to Y$ of profinite spectra is called \\
\begin{itemize}
\item a (stable) equivalence if it induces a weak equivalence of mapping spaces 
$$\map_{\hSp}(Y,E) \to \map_{\hSp}(X,E)$$ 
for all $\Omega$-spectra $E$; 
\item a (stable) cofibration if and only if it is a projective cofibration; 
\item a (stable) fibration if it has the right lifting property with respect to all maps that are stable equivalences and stable cofibrations.
\end{itemize} 
\end{defn}
The following result was shown in \cite{profinhom}, Theorem 2.36. The proof is based on the dual of the general stabilization machinery for model structures by Hovey \cite{hovey} and the localization for fibrantly generated simplicial model structures in \cite{etalecob}, Theorem 6 (see also \cite{etalecob}, Theorem 14).
\begin{theorem}\label{stable}
The classes of stable equivalences, stable fibrations and stable cofibrations define a stable simplicial model structure on profinite spectra. The fibrant objects in $\hSp$ are exactly the $\Omega$-spectra. 
We denote the corresponding homotopy category by $\hSHh$. In particular, the functor $S^1\wedge \cdot: \hSHh \to \hSHh$ is an equivalence of categories.
\end{theorem}
Let $\Sp$ be the stable model category of Bousfield-Friedlander spectra of \cite{bousfried}.
\begin{prop}\label{stablecompletion}
The levelwise profinite completion functor 
$$\compl:\Sp \to \hSp$$ 
preserves stable cofibrations and stable equivalences. The forgetful functor 
$$|\cdot|:\hSp \to \Sp$$ 
preserves fibrations and weak equivalences between fibrant objects.
In particular, $\compl$ induces a functor on the homotopy categories and the pair $(\compl, |\cdot|)$ is a Quillen pair of adjoint functors.  
\end{prop}
\begin{proof}
Let $i:A \to B$ be a cofibration in $\Sp$. Since $\compl:\Shp \to \hShp$ preserves cofibrations and pushouts as a left Quillen functor, the maps $\hat{i}_0$ and $\hat{j}_n$ are cofibrations in $\hShp$. Hence $\hat{i}$ is a cofibration in $\hSp$. \\
Let $Y$ be a pointed profinite space. As a right adjoint the forgetful functor for pointed profinite spaces $|\cdot|: \hShp \to \Shp$ commutes with cotensors. Hence we have  $$|\hom_{\hShp}(S^1, Y)| = \hom_{\Shp}(S^1, |Y|).$$
Moreover, $|\cdot|: \hShp \to \Shp$ sends fibrant pointed profinite spaces to fibrant pointed spaces and preserves weak equivalences between fibrant objects. Hence if $E$ is an $\Omega$-spectrum in $\hSp$, then $|E|$ is an $\Omega$-spectrum in $\Sp$. \\
Let $f: X\to Y$ be a stable equivalence in $\Sp$ and let $E$ be an $\Omega$-spectrum in $\hSp$. Since profinite completion is left adjoint to the forgetful functor we get a commutative diagram of simplicial sets
\begin{equation}\label{mapsquare}
\xymatrix{
\map_{\hSp}(\hat{Y}, E) \ar[r] \ar[d]_{\cong} & \map_{\hSp}(\hat{X},E) \ar[d]^{\cong} \\
\map_{\Sp}(Y, |E|) \ar[r] & \map_{\Sp}(X, |E|)}
\end{equation}
whose vertical maps are isomorphisms. Since the map $f$ is a stable equivalence in $\Sp$ it induces a weak equivalence of mapping spaces 
$$\map_{\Sp}(Y, F) \to \map_{\Sp}(X,F)$$ 
for every $\Omega$-spectrum $F$ in $\Sp$. (This follows for example from \cite{hovey}, Corollary 3.5.) Since $|E|$ is an $\Omega$-spectrum in $\Sp$, this shows that the lower horizontal map in (\ref{mapsquare}) is a weak equivalence. Since the vertical maps are isomorphisms, this implies that the upper horizontal map in (\ref{mapsquare}) is a weak equivalence as well. Since this holds for every $\Omega$-spectrum $E$ in $\hSp$, the map $\hat{f}$ is a stable equivalence in $\hSp$. This finishes the proof that profinite completion preserves stable cofibrations and stable equivalences. Since the forgetful functor is the right adjoint of profinite completion, this implies that $|\cdot|$ preserves fibrations and stable equivalences between fibrant objects. 
\end{proof}

\subsection{Homotopy groups of profinite spectra}
For a positive integer $k\geq 0$, there is an evaluation functor $\Ev_k:\hSp \to \hShp$, $X \mapsto X_k$, sending a profinite spectrum $X$ to its $k$th pointed profinite space $X_k$. It is a right Quillen functor for the stable model structure on $\hSp$. Its left Quillen adjoint is the functor $\Sigma^k(-): \hShp \to \Sp$ defined by $(\Sigma^kY)_m = S^{m-k}Y$ if $m\geq k$ and $(\Sigma^kY)_m = *$ otherwise, where $S^{m-k}Y$ denotes the pointed profinite space obtained by smashing the pointed profinite space $Y$ $(m-k)$-times with $S^1$. \\
Let $S^0$ be the discrete pointed simplicial finite set generated by two points as $0$-simplices one of which is the basepoint. 
We define the sphere spectrum $\Sph^0$ to be the profinite spectrum $\Sigma^0(S^0)$ whose $m$th pointed profinite space is $S^m$, i.e. $S^1$ smashed with itself $m$-times. \\
For a negative integer $n< 0$, we define the $n$th suspension $\Sph^n$ of the sphere spectrum to be the profinite spectrum $\Sigma^{-n}(S^0)$. Its $m$th space is $S^{m+n}$ if $m\geq -n$ and just a point otherwise.  \\
For a positive integer $n>0$, we define the $n$th suspension $\Sph^n$ of the sphere spectrum to be the profinite spectrum $\Sigma^0(S^n)$ whose $m$th space is $S^{m+n}$ for every $m\geq 0$. \\
Via the model structure of Theorem \ref{stable} we can define homotopy groups of a profinite spectrum $X$. Let $RX$ denote a functorial fibrant replacement of $X$ in the stable model structure on $\hSp$ and, for an integer $n$, let $\Sph^n$ be the $n$th suspension of the sphere spectrum. We define $\pi_nX$ to be the group 
$$\pi_nX:=\Hom_{\hSHh}(\Sph^n, RX).$$ 
\begin{prop}\label{stablepis}
{\rm (a)} Let $X$ be a fibrant profinite spectrum, i.e. an $\Omega$-spectrum in $\hSp$, and $n\in \Z$. Let $k\geq 0$ be any natural number such that $n+k\geq0$. Then there is a natural isomorphism of groups
$$\pi_nX\cong \pi_{n+k}X_{k}$$
where $\pi_{n+k}X_{k}$ denotes the profinite homotopy group of the $k$th pointed profinite space of $X$. \\
{\rm (b)} The homotopy groups of a fibrant profinite spectrum $X$ are isomorphic to the homotopy groups of the underlying $\Omega$-spectrum $|X|$ in $\Sp$.\\
{\rm (c)} The homotopy groups of every spectrum are abelian profinite groups. 
\end{prop}
\begin{proof}
(a) If $n\geq 0$, then the Quillen adjoint pair $(\Sigma^0, \Ev_0)$ yields a natural isomorphism 
$$\pi_nX=\Hom_{\hSHh}(\Sph^n, X) = \Hom_{\hSHh}(\Sigma^0(S^n), X) \cong \Hom_{\hHhp}(S^n, X_0).$$
The right hand side is naturally isomorphic to $\pi_0(\Omega^n(X_0))=\pi_n(X_0)$. Since the maps $X_0 \to \Omega^k(X_k)$ are weak equivalences of fibrant objects in $\hShp$ for every $k\geq 0$, this proves the assertion for $n\geq 0$. \\
For $n< 0$, the Quillen adjoint pair $(\Sigma^{-n}, \Ev_{-n})$ yields a natural isomorphism 
$$\pi_nX=\Hom_{\hSHh}(\Sph^n, X) = \Hom_{\hSHh}(\Sigma^{-n}(S^0), X) \cong \Hom_{\hHhp}(S^0, X_{-n}).$$
The right hand side is naturally isomorphic to $\pi_0(X_{-n})$. Reindexing and using the fact that $X$ is an $\Omega$-spectrum gives $\pi_nX=\pi_{n+k}(X_{k})$ for any $k\geq 0$ such that $n+k\geq 0$. This finishes the proof of (a). \\
(b) Since the underlying spectrum $|X|$ is an $\Omega$-spectrum in $\Sp$, the $n$th homotopy group of $|X|$ is given by the abelian group $\pi_{n+k}(|X_k|)$ for any $k\geq 0$ such that $n+k\geq0$. By Lemma \ref{underlyingpi}, there is a natural isomorphism 
$$\pi_{n+k}(|X_k|) \cong \pi_{n+k}(X_k).$$
The assertion now follows from (a).\\
(c) This follows from the fact that the homotopy groups of a profinite space are profinite groups and that $\Omega^2Y$ is an abelian group object for any pointed profinite space $Y$. 
\end{proof}

We can also define generalized cohomology groups with coefficients in profinite spectra. Let $X$ be a fibrant profinite spectrum and let $Z$ be a spectrum in $\Sp$. The spectrum $Z$ is isomorphic to a colimit of suspended sphere spectra $\Sph^{n_{\alpha}}$ for integers $n_{\alpha}\in \Z$ indexed by a filtered category. We define the $k$th generalized cohomology group of $Z$ with coefficients in $X$ to be the limit
\begin{equation}\label{defgencoh}
X^kZ:=\lim_{\alpha} \Hom_{\hSHh}(\Sph^{n_{\alpha}}, X[k])
\end{equation}
where $X[k]$ denotes the $\Omega$-spectrum in $\hSp$ whose $m$th space is $X_{m-k}$ for $m\geq k$ and a point otherwise. Since each abelian group $\Hom_{\hSHh}(\Sph^{n_{\alpha}}, X[k])=\pi_{n_{\alpha}}(X[k])$ carries a profinite structure by Lemma \ref{stablepis}, we deduce that the group $X^kZ$ is also a profinite abelian group.\\
Moreover, since each space in the sphere spectrum $\Sph^n$ is a pointed simplicial finite set for any integer $n$, the Quillen adjunction of Proposition \ref{stablecompletion} between profinite completion and the forgetful functor yields a natural isomorphism of abelian groups
$$\Hom_{\hSHh}(\Sph^{n}, X)\cong \Hom_{\SHh}(\Sph^{n}, |X|).$$
Hence the $k$th generalized cohomology group of $Z$ with coefficients in the profinite spectrum $X$ coincides with the $k$th generalized cohomology group of $Z$ with coefficients in the underlying $\Omega$-spectrum $|X|$ in $\Sp$, i.e. there is a natural isomorphism 
\begin{equation}\label{compgen}
X^kZ = \lim_{\alpha} \Hom_{\hSHh}(\Sph^{n_{\alpha}}, X[k]) \cong \lim_{\alpha} \Hom_{\SHh}(\Sph^{n_{\alpha}}, |X[k]|) \cong |X|^kZ.
\end{equation}
%
%
%
\subsection{Homotopy limits of profinite spectra}
Let $I$ be a small category and let $X(-)$ be a functor from $I$ to the full subcategory of $\Omega$-spectra in $\hSp$. For each $n\geq0$ and each $i\in I$, we have a fibrant pointed profinite space $X_n(i):=X(i)_n$. So for every $n\geq 0$, this defines an $I$-diagram $X_n(-)$ of fibrant pointed profinite spaces. The homotopy limit $\holim_{i\in I} X_n(i)$ in $\hShp$ is again a fibrant pointed profinite space by Lemma \ref{holiminv}. Since the homotopy limit is defined using cotensors, there is a natural isomorphism
$$\holim_{i\in I}\Omega (X_n(i)) \cong \Omega (\holim_{i\in I} X_n(i)).$$
Since each $X(i)$ is an $\Omega$-spectrum in $\hSp$ and since $\holim_{i\in I}$ preserves weak equivalences between fibrant objects by Lemma \ref{holiminv}, for each $n\geq 0$ we obtain a weak equivalence in $\hShp$
$$\holim_{i\in I} X_n(i) \stackrel{\sim}{\to} \holim_{i\in I} \Omega X_n(i) \cong \Omega \holim_{i\in I} X_n(i).$$
Hence together with these structure maps the sequence of fibrant pointed profinite spaces $\holim_{i\in I} X_n(i)$ defines an $\Omega$-spectrum in $\hSp$ that we denote by $\holim_{i\in I} X(i)$ and call the homotopy limit of the diagram $X(-)$ (see \cite{thomason}, \S 5, for the analogous story for $\Sp$). 
\begin{lemma}\label{stableholiminv}
Let $X(-) \to Y(-)$ be a natural transformation of functors from a small category $I$ to the full subcategory of $\Omega$-spectra in $\hSp$. The induced map 
$$\holim_{i\in I}X(i) \to \holim_{i\in I}Y(i)$$ 
is an equivalence of $\Omega$-spectra in $\hSp$. 
\end{lemma}
\begin{proof}
We already know that $\holim$ sends a small diagram of $\Omega$-spectra in $\hSp$ to an $\Omega$-spectrum in $\hSp$. Since $\holim$ is constructed termwise and since stable equivalences between $\Omega$-spectra are exactly the projective equivalences in $\hSp$, the assertion follows from the corresponding result Lemma \ref{holiminv} for $\hShp$.
\end{proof}

\begin{lemma}\label{stableholimlemma}
Let $X(-): I \to \hSp$ be a small cofiltering diagram of $\Omega$-spectra in $\hSp$. There is a natural isomorphism of homotopy groups of underlying spectra for every $q\in \Z$
$$\pi_q(|\holim_{i\in I} X(i)|) \cong \lim_{i\in I} \pi_q(|X(i)|).$$
\end{lemma}
\begin{proof}
Since the profinite spectrum $\holim_{i\in I} X(i)$ is constructed levelwise and is an $\Omega$-spectrum in $\hSp$, the assertion follows from Proposition \ref{stablepis} and Lemma \ref{holimlemma}. 
\end{proof}
%
%
\subsection{Spectra with finite homotopy groups}
The functors given by taking profinite completion of spectra and taking profinite completion of abelian groups are related to each other in a similar way as the corresponding profinite completion functor for spaces. For the purpose of this paper, the following concrete result will be sufficient.
\begin{theorem}\label{stablefinitecompletion}
Let $X\in \Sp$ be a spectrum whose homotopy groups are all finite groups. 
Then there is a natural map 
$$X \to F^sX$$ 
of spectra from $X$ to a profinite spectrum $F^sX$ built of simplicial finite sets such that $F^sX$ is fibrant in $\hSp$ and $X\to F^sX$ is a stable equivalence of underlying spectra. In particular, it  induces an isomorphism $\pi_*X \cong \pi_*|F^sX|$ of homotopy groups of underlying spectra. \\
The assignment $X\mapsto F^sX$ is functorial in the sense that given a map $g:X\to Y$ between spectra with finite homotopy groups, there is a map $F^s(g)$ in $\hSp$ such that the following diagram of underlying spectra commutes
$$\xymatrix{
X\ar[d] \ar[r]^g & Y \ar[d] \\
F^sX \ar[r]_{F^s(g)} & F^sY.}$$
\end{theorem}
\begin{proof}
Let $X$ be a spectrum whose homotopy groups are finite. We can assume that $X$ is an $\Omega$-spectrum in $\Sp$, i.e. that each pointed space $X_n$ is fibrant in $\Shp$ and $X_n \to \Omega(X_{n+1})=\hom_{\Shp}(S^1, X_{n+1})$ is a weak equivalence for all $n\geq 0$. This implies for all $k\geq 0$
$$\pi_k(X_n) \cong \pi_{k}(\Omega(X_{n+1})) \cong \pi_{k+1}(X_{n+1}).$$
The $i$th stable homotopy group of the spectrum $X$ is given as any of the isomorphic homotopy groups $\pi_{i+n}(X_n)$ for positive $n$ such that $i+n \geq0$. Since the homotopy groups of $X$ are finite, each of the groups $\pi_k(X_n)$ is a finite group for all $k, n \geq 0$. \\
Since $X$ is a spectrum, we can assume that each $X_n$ is a connected simple pointed space.
Hence we can apply the finite replacement $F$ of Lemma \ref{finitereplacement} to each space $X_n$ and obtain the profinite spectrum $F^sX$. By Lemma \ref{Floop}, the induced maps 
$$FX_n \to \Omega(FX_{n+1})$$
are weak equivalences in $\hShp$. Since each $FX_n$ is fibrant in $\hShp$, this implies that the levelwise fibrant profinite spectrum $F^sX$ is already an $\Omega$-spectrum in $\hSp$. 
Furthermore, the $i$th stable homotopy group of the profinite spectrum $F^sX$ is given by any of the isomorphic finite groups 
$$\pi_{i+n}(FX_n)\cong \pi_{i+n}(X_n)$$ 
for positive $n$ such that $i+n \geq0$. This shows that the map $X\to F^sX$ is a weak equivalence of underlying spectra and induces an isomorphism $\pi_*X \cong \pi_*F^sX$. Functoriality follows from the functoriality of $F$ in Lemma \ref{finitereplacement}.
\end{proof}
%
%
\subsection{Profinite $G$-spectra}
Now let $G$ be a profinite group. We consider $S^1$ as a simplicial finite set with trivial $G$-action.
\begin{defn}\label{Gspectra}
A profinite $G$-spectrum $X$ is a sequence of pointed profinite $G$-spaces $\{X_n\}$ together with maps $S^1\wedge X_n \to X_{n+1}$ of pointed profinite $G$-spaces for each $n\geq 0$. A map of profinite $G$-spectra $X \to Y$ is a collection of maps $X_n \to Y_n$ in $\hShpg$ compatible with the structure maps of $X$ and $Y$. We denote the category of profinite $G$-spectra by $\hSpg$.
\end{defn}
In \cite{gspaces} a slightly different notion of profinite $G$-spectra was introduced by using a cofibrant replacement of $S^1$ in $\hShpg$. This notion turns out to be less appropriate for our purposes. Therefore, we consider Definition \ref{Gspectra} as the correct notion of profinite $G$-spectra. \\
The category of profinite $G$-spectra is a simplicial category. Let $X$ and $Y$ be profinite $G$-spectra. The mapping space $\map_{\hSpg}(X,Y)$ is defined as the simplicial set whose set of $n$-simplices is given as the set of maps 
$$\map_{\hSpg}(X,Y)_n=\Hom_{\hSpg}(X\wedge \Delta[n]_+, Y)$$
where $\Delta[n]_+$ is considered as a simplicial finite $G$-set with trivial $G$-action. This defines a functor 
$$\map_{\hSpg}(-,-):\hSpg^{\mathrm{op}}\times \hSpg \to \Sh.$$
Let $K$ be a simplicial set and $X$ a profinite $G$-spectrum. The tensor object $X\otimes K\in \hSpg$ is defined as the profinite $G$-spectrum whose $n$th pointed profinite $G$-space $X_n\wedge K_+$ where we consider $K_+$ as a pointed simplicial set with trivial $G$-action. The function object in $\hSpg$ is defined as the profinite spectrum $\hom_{\hSpg}(K,X)\in \hSpg$ whose $n$th pointed profinite $G$-space is given by 
$$(\hom_{\hSpg}(K,X))_n = \hom_{\hShpg}(K, X_n).$$ 
The structure map of $\hom_{\hSpg}(K,X)$ is determined by the $G$-equivariant map
$$\hom_{\hShpg}(K, X_n) \to \hom_{\hShpg}(K, \Omega X_{n+1}) \cong \Omega (\hom_{\hShpg}(K, X_{n+1})).$$
%
%
%
Let $X$ and $Y$ be profinite $G$-spectra and let $K$ be a simplicial set. We have natural bijections
$$\map_{\hSpg}(X\otimes K, Y) \cong \map_{\Sh}(K, \map_{\hSpg}(X,Y))$$
and 
$$\map_{\hSpg}(Y, \hom_{\hSpg}(K, X)) \cong \map_{\Sh}(K, \map_{\hSpg}(X,Y)).$$
%
If the simplicial set $K$ is already equipped with a base point and $X$ is a profinite $G$-spectrum, we also denote by $\hom_{\hSpg}(K,X)\in \hSp$ the profinite $G$-spectrum whose $n$th space is $\hom_{\hShpg}(K,X_n)$. 
%
%
As for profinite spectra we would like to construct a stable model structure on $\hSpg$. Therefore, we have to check the following lemma. 
\begin{lemma}\label{endofunctor}
The functor $X\mapsto S^1\wedge X$ is a left Quillen endofunctor of $\hShpg$. 
\end{lemma}
\begin{proof}
We saw in Example \ref{exGomega} that smashing with $S^1$ is a left adjoint functor whose right adjoint is given by the functor $\hShpg \to \hShpg$, $X \to \Omega X$. 
Since weak equivalences in $\hShpg$ are determined by the underlying maps in $\hShp$, it follows that the functor $X\mapsto S^1 \wedge X$ from $\hShpg$ to itself preserves weak equivalences. In order to see that it preserves cofibrations we recall that a map $f:X\to Y$ in $\hShpg$ is a cofibration if and only if $f$ is a levelwise injection and the action of $G$ on $Y_n -f(X_n)$ is free for each $n \geq 0$. We know that the smash product preserves levelwise injections. Now $G$ acts diagonally on the smash product $S^1\wedge Y$. This implies that if the $G$-action is free on $Y_n -f(X_n)$, then the $G$-action on $(S^1\wedge Y)_n - (S^1\wedge f)((S^1 \wedge X)_n)$ is free as well. Hence $X\mapsto S^1 \wedge X$ is a left adjoint endofunctor that preserves weak equivalences and cofibrations in $\hShpg$. 
\end{proof}

As for profinite spectra, we start with a projective model structure.
\begin{defn}\label{Gdefstrict}
A map $f$ in $\hSpg$ is a projective weak equivalence (projective fibration) if each map $f_n$ is a weak equivalence (fibration) in $\hShpg$. A map $i$ is a projective cofibration if it has the left lifting property with respect to all projective trivial fibrations. 
\end{defn}
The following proposition follows again as in \cite{hovey}, Proposition 1.14.
\begin{prop}\label{Gprojcofs}
A map  $i:A \to B$ in $\hSpg$ is a projective (trivial) cofibration if and only if $i_0:A_0 \to B_0$ and the induced maps  $j_n:A_n \amalg_{S^1\wedge A_{n-1}} S^1\wedge B_{n-1}\to B_n$ for $n\geq 1$ are (trivial) cofibrations in $\hShpg$. 
\end{prop}
\begin{prop}\label{Gthmstrict}
The projective weak equivalences, projective fibrations and projective cofibrations define a left proper fibrantly generated simplicial model structure on $\hSpg$. 
\end{prop}
\begin{proof} 
That we obtain a left proper fibrantly generated model structure can be proven in essentially the same way as Theorem 1.13 in \cite{hovey}. In order to show the factorization axiom one uses a cosmall object argument and the fact that $\hShpg$ is fibrantly generated. It remains to prove that this model structure is simplicial. We have defined tensor and cotensor objects for $\hSp$ above. Let $i: A\to B$ be a cofibration of finite simplicial sets and $p: X\to Y$ a projective fibration in $\hSpg$. We have to show that the induced map
$$(i^*, p_*): \hom_{\hSpg}(B, X) \to \hom_{\hSpg}(A, X) \times_{\hom_{\hSpg}(A, Y)} \hom_{\hSpg}(B, Y)$$
is a projective fibration in $\hSpg$, which is trivial if either $i$ or $p$ is trivial. For $n\geq 0$, the $n$th map $(i^*, p_*)_n$ is given by the map of pointed profinite spaces
$$(i^*, p_{n*}): \hom_{\hShpg}(B, X_n) \to \hom_{\hShpg}(A, X_n) \times_{\hom_{\hShpg}(A, Y_n)} \hom_{\hShpg}(B, Y_n).$$
The model structure on $\hShpg$ is simplicial. This implies that $(i^*, p_{n*})$ is a fibration, since $i$ is a cofibration and $p_n$ is a fibration. Moreover, $(i^*, p_{n*})$ is a weak equivalence if either $i$ or $p_n$ is a weak equivalence. Since projective weak equivalences and projective fibrations are determined levelwise, this shows that $(i^*, p_*)$ is a projective fibration which is a trivial projective fibration if either $i$ or $p$ is trivial.
\end{proof}

In the second step we enlarge the class of weak equivalences by localizing the projective model structure. 
\begin{defn}\label{Gomega}
A profinite $G$-spectrum $E\in \hSpg$ is called an $\Omega$-spectrum if each $E_n$ is fibrant in $\hShpg$ and the adjoint structure maps 
$$E_n \to \Omega E_{n+1}=\hom_{\hShpg}(S^1,E_{n+1})$$ 
are weak equivalences in $\hShpg$ for all $n\geq 0$. A map $f:X\to Y$ of profinite $G$-spectra is called 
\begin{itemize}
\item a (stable) equivalence if it induces a weak equivalence of mapping spaces 
$$\map_{\hSpg}(Y,E) \to \map_{\hSpg}(X,E)$$ 
for all $\Omega$-spectra $E$ in $\hSpg$; 
\item a (stable) cofibration if and only if it is a projective cofibration; 
\item a (stable) fibration if it has the right lifting property with respect to all maps that are stable equivalences and stable cofibrations. 
\end{itemize}
\end{defn}
\begin{theorem}\label{Gmodelspectra}
The classes of stable equivalences, fibrations, and cofibrations of Definition \ref{Gomega} provide $\hSpg$ with a stable simplicial model structure. 
The fibrant profinite $G$-spectra are exactly the $\Omega$-spectra in $\hSpg$. Its underlying profinite spectra are $\Omega$-spectra in $\hSp$. A map in $\hSpg$ is a stable equivalence between $\Omega$-spectra if and only if it is a projective weak equivalence. 
We denote its homotopy category by $\hSHhg$.
\end{theorem}
\begin{proof}
In order to show that there is a stable model structure on $\hSpg$ we apply again the dual methods of Hovey \cite{hovey}. We know from Theorem \ref{Gmodelunstable} that $\hShpg$ satisfies the necessary properties. From Lemma \ref{endofunctor} we know that smashing with $S^1$ is a left Quillen endofunctor on $\hSpg$. By Proposition \ref{Gthmstrict}, we know that the projective model structure on $\hSpg$ is simplicial. By localizing the projective model structure with respect to the $\Omega$-spectra in $\hSpg$, we obtain a stable model structure on $\hSpg$. \\
The $\Omega$-spectra are the fibrant objects in the stable model structure, since they are the local objects. Since weak equivalences and fibrations in $\hShpg$ are determined by their underlying maps in $\hShp$, this shows that the fibrant objects in $\hSpg$ are exactly the profinite $G$-spectra whose underlying spectrum is fibrant in $\hSp$. \\
The assertion on stable equivalences between $\Omega$-spectra in $\hSpg$ follows from the general theory of Bousfield localizations.
%
%
%
\end{proof}

\begin{cor}\label{Komega}
Let $K$ be a closed subgroup of the profinite group $G$. If $X$ is an $\Omega$-spectrum in $\hSpg$, then its restriction to a profinite $K$-spectrum is also an $\Omega$-spectrum in $\hSpk$.
\end{cor}
\begin{proof}
Since weak equivalences and fibrations for $\hShpg$ and $\hShpk$ are determined by their underlying maps in $\hShp$, the assertion follows from the definition of an $\Omega$-spectrum.
\end{proof}

\begin{prop}\label{stableforgetG}
(1) The forgetful functor $\hSpg \to \hSp$ sends $\Omega$-spectra to $\Omega$-spectra and preserves stable equivalences between $\Omega$-spectra. \\
(2) The composition of forgetful functors $\hSpg \to \hSp \to \Sp$, which we also denote by $|\cdot|:\hSpg \to \Sp$, sends $\Omega$-spectra to $\Omega$-spectra and preserves stable equivalences between $\Omega$-spectra. 
\end{prop}
\begin{proof}
(1) In the model structures on $\hSpg$ and $\hSp$, stable equivalences between $\Omega$-spectra are exactly the projective equivalences. By Theorem \ref{Gmodelspectra}, the underlying profinite spectrum of an $\Omega$-spectrum in $\hSpg$ is an $\Omega$-spectrum in $\hSp$. Since projective equivalences are maps that are levelwise weak equivalences, the assertion follows from the fact that the forgetful functor $\hShpg \to \hShp$ preserves weak equivalences between fibrant objects. \\
(2) The second assertion follows from the first and from Proposition \ref{stablecompletion}.
\end{proof}

\subsection{Homotopy groups of fibrant profinite $G$-spectra}
\begin{defn}\label{Gdefofstablepis}
Let $X$ be an $\Omega$-spectrum in $\hSpg$. Its underlying profinite spectrum is fibrant in $\hSp$, and, for $n\in \Z$, we define the $n$th homotopy group of $X$ to be the $n$th homotopy group of its underlying fibrant profinite spectrum.
\end{defn}

Note that in Theorem \ref{Gmodelspectra} we did not show that the stable equivalences in $\hSpg$ are determined by their underlying maps in $\hSp$. If we start with an arbitrary profinite $G$-spectrum $Y$, we do not claim that the homotopy groups of the underlying profinite spectrum 
are isomorphic to the homotopy groups of $R_GY$, where $R_G$ denotes a functorial fibrant replacement functor in $\hSpg$. The point is that, since $Y$ is not fibrant, the fibrant replacement functors in $\hSp$ and $\hSpg$ may send it to profinite spectra which are not stably equivalent in $\hSp$. \\
But if we start with a fibrant profinite $G$-spectrum $X$, i.e. an $\Omega$-spectrum in $\hSpg$, as in Definition \ref{Gdefofstablepis}, then the homotopy groups of $X$ as a profinite $G$-spectrum are canonically isomorphic to the homotopy groups of its underlying $\Omega$-spectrum in $\hSp$ by Proposition \ref{stableforgetG}.\\
%
Moreover, the $G$-action on the $\Omega$-spectrum $X$ in $\hSpg$ induces a $G$-action on each homotopy group of $X$. 
\begin{prop}\label{Gstablepis}
Let $X$ be an $\Omega$-spectrum in $\hSpg$. Then each homotopy group $\pi_nX$ is a profinite $G$-module for every $n\in \Z$. 
\end{prop}
\begin{proof}
Since the underlying spectrum of $X$ is an $\Omega$-spectrum in $\hSp$, Lemma \ref{stablepis} shows that $\pi_nX$ is isomorphic to $\pi_{n+k}X_k$ for any $k\geq 0$ such that $n+k \geq 0$. By Proposition \ref{Gmodulepi}, the groups $\pi_{n+k}X_k$ are profinite $G$-modules. 
\end{proof}
\begin{remark}\label{Gremgencoh}
Let $X$ be again an $\Omega$-spectrum in $\hSpg$. For a spectrum $Z\in \Sp$, the generalized cohomology groups $X^kZ$ of (\ref{defgencoh}) of $Z$ with coefficients in the underlying fibrant profinite spectrum of $X$, inherit a $G$-action from $X$. By Proposition \ref{Gstablepis}, each 
$$\Hom_{\hSHh}(\Sph^n, X)= \pi_nX$$
is a profinite $G$-module. Hence (\ref{defgencoh}) provides $X^kZ$ with the structure of a profinite $G$-module. Moreover, we may consider (\ref{compgen}) as an isomorphism of profinite $G$-modules between $X^kZ$ and $|X|^kZ$. 
\end{remark}
%
%
\subsection{Homotopy limits of fibrant profinite $G$-spectra}
Let $I$ be a small category and let $X(-)$ be a functor from $I$ to the full subcategory of $\Omega$-spectra in $\hSpg$. The homotopy inverse limit of the diagram $X(-)$ in $\hSpg$ is again defined levelwise for each space. For each $n\geq0$ and each $i\in I$, the pointed profinite $G$-space $X_n(i):=X(i)_n$ is fibrant. So for every $n\geq 0$, the homotopy limit $\holim_{i\in I} X_n(i)$ in $\hShpg$ is a fibrant pointed profinite $G$-space by Lemma \ref{Gholiminv} and there is a natural isomorphism in $\hShpg$
$$\holim_{i\in I}\Omega (X_n(i)) \cong \Omega (\holim_{i\in I} X_n(i)).$$
Since each $X(i)$ is an $\Omega$-spectrum in $\hSpg$ and since $\holim_{i\in I}$ preserves weak equivalences between fibrant objects by Lemma \ref{Gholiminv}, for each $n\geq 0$, we obtain a weak equivalence in $\hShpg$
$$\holim_{i\in I} X_n(i) \stackrel{\sim}{\to} \holim_{i\in I} \Omega X_n(i) \cong \Omega \holim_{i\in I} X_n(i).$$
Hence together with these structure maps the sequence of fibrant pointed profinite $G$-spaces $\holim_{i\in I} X_n(i)$ defines a $\Omega$-spectrum in $\hSpg$ that we denote by $\holim_{i\in I} X(i)$ and call the homotopy limit of the diagram $X(-)$. 
\begin{lemma}\label{Gstableholiminv}
Let $X(-) \to Y(-)$ be a natural transformation of functors from a small category $I$ to the full subcategory of $\Omega$-spectra in $\hSpg$. Then the induced map 
$$\holim_{i\in I}X(i) \to \holim_{i\in I}Y(i)$$ 
is an equivalence of $\Omega$-spectra in $\hSpg$. 
\end{lemma}
\begin{proof}
We already know that $\holim$ sends a small diagram of $\Omega$-spectra in $\hSpg$ to an $\Omega$-spectrum in $\hSpg$. Since $\holim$ is constructed termwise and since stable equivalences between $\Omega$-spectra are exactly the projective equivalences in $\hSpg$, the assertion follows from the corresponding result for $\hShpg$ given in Lemma \ref{Gholiminv}.
\end{proof}

\begin{lemma}\label{Gstableholimlemma}
Let $X(-): I \to \hSpg$ be a small cofiltering diagram of $\Omega$-spectra in $\hSpg$. Then the isomorphism of Lemma \ref{stableholimlemma} 
$$\pi_q(|\holim_{i\in I} X(i)|) \cong \lim_{i\in I} \pi_q(|X(i)|)$$
is an isomorphism of profinite $G$-modules for every $q\in \Z$.
\end{lemma}
\begin{proof}
Since the $\Omega$-spectrum $\holim_{i\in I} X(i)$ in $\hSpg$ is constructed levelwise, the assertion follows from Lemma \ref{Gholimlemma}.
\end{proof}
%
%
\subsection{$G$-spectra with finite homotopy groups}
\begin{theorem}\label{Gstablefinitecompletion}
Let $G$ be a strongly complete profinite group. Let $X\in \Sp$ be a spectrum such that each space $X_n$ is a pointed $G$-space. We assume that the homotopy groups of $X$ are all finite groups. Then there is a $G$-equivariant map 
$$\varphi^s: X \to F^s_GX$$ 
of spectra from $X$ to a profinite $G$-spectrum $F^s_GX$ built of simplicial finite discrete $G$-sets such that $F^s_GX$ is fibrant in $\hSpg$ and $\varphi^s$ is a stable equivalence of underlying spectra. In particular, $\varphi^s$ induces an isomorphism $\pi_*X \cong \pi_*|F^s_GX|$ of the homotopy groups of underlying spectra.\\
The assignment $X\mapsto F_G^sX$ is functorial in the sense that given a map $h:X\to Y$ between spectra whose spaces are pointed $G$-spaces and whose homotopy groups are finite, there is a map $F_G^s(h)$ in $\hSpg$ such that the following diagram of underlying spectra commutes
$$\xymatrix{
X\ar[d] \ar[r]^h & Y \ar[d] \\
F_G^sX \ar[r]_{F_G^s(h)} & F_G^sY.}$$
\end{theorem}
\begin{proof}
The proof is similar to the one of Theorem \ref{stablefinitecompletion} using the replacement functor of Lemma \ref{Gfinitereplacement}. We can assume that $X$ is an $\Omega$-spectrum in $\Sp$. 
Since the homotopy groups of $X$ are finite, each of the groups $\pi_k(X_n)$ is a finite group for all $k, n \geq 0$. Moreover, since $X$ is a spectrum, we can assume that each $X_n$ is a simple connected pointed $G$-space. 
Hence we can apply the functor $F_G$ of Lemma \ref{Gfinitereplacement} to each space $X_n$ and obtain the profinite spectrum $F^s_GX$. 
By Lemma \ref{GFloop}, the induced maps 
$$F_GX_n \to \Omega(F_GX_{n+1})$$
are weak equivalences in $\hShpg$. Since each $F_GX_n$ is fibrant in $\hShpg$, this implies that the profinite $G$-spectrum $F_G^sX$ is an $\Omega$-spectrum in $\hSp$. 
The $i$th stable homotopy group of the $\Omega$-spectrum $F^s_GX$ is given by any of the isomorphic finite groups 
$$\pi_{i+n}(F_GX_n)\cong \pi_{i+n}(X_n)$$ 
for positive $n$ such that $i+n \geq0$. This shows that the map $X\to F^s_GX$ induces an isomorphism $\pi_*X \cong \pi_*|F^s_GX|$. Functoriality of $F^s_G$ follows from the functoriality of $F_G$.
\end{proof}
\section{Homotopy fixed point spectra}
\subsection{Continuous equivariant mapping spectra}
Let $Y$ be a profinite space and $W$ be a pointed profinite space. The functor $\hSh \to \hShp$, $Y\to Y_+$, defined by adding a disjoint basepoint, is the left adjoint of the functor that forgets the basepoint. Hence there is a natural isomorphism of simplicial sets
\begin{equation}\label{pointedad}
\map_{\hShp}(Y_+,W)\cong \map_{\hSh}(Y,W).
\end{equation}
The spaces on both sides of (\ref{pointedad}) are pointed by the map that factors through $\ast \to W$. Hence (\ref{pointedad}) is in fact an isomorphism of pointed simplicial sets. We will use the notation $\Map(Y,W)$ for the pointed simplicial set $\map_{\hShp}(Y_+,W)$ together with its basepoint $Y\to \ast \to W$. This defines a functor
$$\Map(-,-): \hSh \times \hShp^{\op} \to \Shp.$$
\begin{defn}\label{Mapdefn}
Let $G$ be a profinite group. Let $Y$ be a profinite $G$-space and $W$ be a pointed profinite $G$-space. We define $\Mapg(Y,W)$ to be the pointed simplicial set $\map_{\hShpg}(Y_+,W)$ pointed by the map $Y_+\to \ast \to W$. This defines a functor
$$\Mapg(-,-): \hShg \times \hShpg^{\op} \to \Shp.$$
\end{defn}
When $Y$ is a profinite $G$-space and $W$ is a pointed profinite $G$-space, we can equip the pointed simplicial set $\Map(Y,W)$ with a $G$-action by $(gf)(y):= gf(yg^{-1})$. With this $G$-action on $\Map(Y,W)$, $\Mapg(Y,W)$ is the pointed space of $G$-fixed points of the pointed  space $\Map(Y,W)$. 
\begin{lemma}\label{mapfib}
Let $Y$ be a cofibrant profinite $G$-space and $f:V\to W$ be a fibration between pointed profinite $G$-spaces. Then 
$$\Mapg(Y,f): \Mapg(Y,V)\to \Mapg(Y,W)$$ 
is a fibration of pointed simplicial sets.
\end{lemma}
\begin{proof}
Since $Y$ is cofibrant and $\hShpg$ is a simplicial model category, the map 
$$\map_{\hShpg}(Y_+,V) \to \map_{\hShg}(Y_+,W)$$ 
is a fibration of simplicial sets. Thus the $\Mapg(Y,f)$ is a fibration of pointed spaces.
\end{proof}

\begin{lemma}\label{mapwe}
Let $Y$ be a cofibrant profinite $G$-space, $f:V\to W$ be a weak equivalence between fibrant pointed profinite $G$-spaces. 
Then 
$$\Mapg(Y,f): \Mapg(Y,V)\to \Mapg(Y,W)$$ 
is a weak equivalence of fibrant pointed simplicial sets.
\end{lemma}
\begin{proof}
The profinite space $Y$ is a cofibrant object in $\hShg$ and $V$ and $W$ are fibrant objects in $\hShpg$ by assumption. Since $\hShpg$ is a simplicial model category, the induced map $\map_{\hShpg}(Y_+,V) \to \map_{\hShpg}(Y_+,W)$ is a weak equivalence of fibrant simplicial sets. Hence the pointed map $\Mapg(Y,f)$ is a weak equivalence of fibrant pointed spaces. 
\end{proof}
\begin{lemma}\label{mapfiber}
Let $Y$ be a cofibrant profinite $G$-space. Let $X(-):I\to \hShpg$ be a small diagram of fibrant pointed profinite $G$-spaces.  
Then there is a natural isomorphism of fibrant pointed simplicial sets
$$\Mapg(Y,\holim_{i\in I}X(i)) \cong \holim_{i\in I}\Mapg(Y,X(i))$$
where $\holim_{i\in I}\Mapg(Y,X(i))$ denotes the homotopy limit in $\Shp$ of the small diagram $\Mapg(Y,X(-)):I \to \Shp$ of pointed simplicial sets. In particular, $\Mapg(Y,-)$ preserves homotopy fibers.
\end{lemma}
\begin{proof}
Since $Y$ is a cofibrant object in $\hShg$, $\Mapg(Y,-)$ preserves fibrations. Moreover, being a right adjoint functor, $\Mapg(Y,-)$ preserves products and limits and sends cotensors to cotensors in the sense that there is a natural isomorphism of pointed simplicial sets
$$\Mapg(Y, \hompg(K, W)) \cong \hom_{\Shp}(K, \Mapg(Y,W))$$
for every simplicial set $K$, where the right hand space is equal to the pointed space $\map_{\Shp}(K, \Mapg(Y,W))$ whose basepoint is the map having constant image the basepoint of $\Mapg(Y,W)$. Thus $\Mapg(Y,-)$ sends the equalizer of the diagram of pointed profinite $G$-spaces
$$\prod_{i\in I} \hompg(B(I/i),X(i)) \rightrightarrows \prod_{\alpha:i\to i' \in I} \hompg(B(I/i),X(i')),$$
which is by definition $\holim_{i\in I}X(i)$, to the equalizer in $\Shp$ of the diagram of pointed spaces
$$
\prod_{i\in I} \hom_{\Shp}(B(I/i),\Mapg(Y,X(i))) \rightrightarrows \prod_{\alpha:i\to i' \in I} \hom_{\Shp}(B(I/i), \Mapg(Y,X(i'))).
$$
Since the equalizer of the last diagram is by definition $\holim_{i\in I}\Mapg(Y,X(i))$, this proves the assertion. The statement on homotopy fibers is a special case of the first assertion. 
\end{proof}

Now we turn our attention to continuous mapping spectra. 
\begin{defn}\label{stableMapdefn}
For a profinite space $Y$ and a profinite spectrum $X$, we denote by $\Map(Y, X)$ the spectrum whose $n$th space is given by the pointed simplicial set $\Map(Y,X_n)$. This defines a functor
$$\Map(-,-): \hSh \times \hSp^{\op} \to \Sp.$$
\end{defn}
\begin{defn}\label{stableMapgdefn}
Let $Y$ be a profinite $G$-space and $X$ a profinite $G$-spectrum. We define $\Mapg(Y, X)$ to be the spectrum whose $n$th space is given by the pointed simplicial set $\Mapg(Y,X_n)$ defined in Definition \ref{Mapdefn}. The structure maps are defined as follows. The compatibilities of mapping spaces and cotensors provide an isomorphism 
$$\Mapg(Y,\homp(S^1,X_{n}))\cong \hom_{\Shp}(S^1, \Mapg(Y,X_{n})).$$ 
The space on the left hand side is $\Mapg(Y,\Omega(X_n))$ and the space on the right hand side is $\Omega(\Mapg(Y,X_{n}))$. Hence the map $X_n\to \Omega X_{n+1}$ defines a map 
\begin{equation}\label{Mapstructuremap}
\Mapg(Y,X_n)) \to \Mapg(Y,\Omega X_{n+1}) \cong \Omega(\Mapg(Y,X_{n+1})).
\end{equation}
This provides a functor
$$\Mapg(-,-): \hShg \times \hSpg^{\op} \to \Sp.$$
\end{defn}
\begin{remark}\label{stableMapgactionrem}
Let $Y$ be a profinite $G$-space and $X$ a profinite $G$-spectrum. If we equip again $\Map(Y,X_n)$ with the above $G$-action, then we can consider $\Map(Y,X)$ as a spectrum which is built out of pointed spaces with a $G$-action and $\Mapg(Y,X)$ is the spectrum of fixed points of $\Map(Y,X)$. 
\end{remark}
\begin{lemma}\label{mapomega}
Let $Y$ be a cofibrant profinite $G$-space. If $X$ is an $\Omega$-spectrum in $\hSpg$, then $\Mapg(Y,X)$ is an $\Omega$-spectrum in $\Sp$.
\end{lemma}  
\begin{proof}
By Lemma \ref{mapwe} and the definition of the structure maps of $\Mapg(Y,X)$ in (\ref{Mapstructuremap}), we can conclude that if $X$ is an $\Omega$-spectrum in $\hSpg$, then $\Mapg(Y,X)$ is an $\Omega$-spectrum in $\Sp$.
\end{proof}

\begin{lemma}\label{stablemapwe}
Let $Y$ be a cofibrant profinite $G$-space. The functor $\Mapg(Y,-)$ sends stable equivalences between profinite $\Omega$-spectra in $\hSpg$ to stable equivalences between $\Omega$-spectra in $\Sp$. 
\end{lemma}
\begin{proof}
Let $f: X\to X'$ be a stable equivalence between profinite $\Omega$-spectra in $\hSpg$. By Lemma \ref{mapomega}, $\Mapg(Y,X)$ and $\Mapg(Y,X')$ are $\Omega$-spectra in $\Sp$. Since stable equivalences between $\Omega$-spectra in $\hSpg$ and $\Sp$ are exactly the projective (i.e. levelwise) equivalences, the assertion now follows from Lemma \ref{mapwe}.
\end{proof}

\begin{prop}\label{stablemapfiber}
Let $Y$ be a cofibrant profinite $G$-space. Let $X(-):I\to \hSpg$ be a small diagram of $\Omega$-spectra in $\hSpg$.   
Then there is a natural isomorphism of $\Omega$-spectra in $\Sp$
$$\Mapg(Y,\holim_{i\in I}X(i)) \cong \holim_{i\in I}\Mapg(Y,X(i))$$
where $\holim_{i\in I}\Mapg(Y,X(i))$ denotes the homotopy limit in $\Sp$ of the small diagram $\Mapg(Y,X(-)):I \to \Sp$ of $\Omega$ -spectra in $\Sp$. \\
In particular, the functor $\Mapg(Y,-)$ preserves homotopy fibers of maps between $\Omega$-spectra.   
\end{prop}
\begin{proof}
Since $\Mapg(Y,-)$ and homotopy limits are defined levelwise, the first assertion follows from Lemma \ref{mapfiber}. The second assertion is a special case of the first one.
\end{proof}

\subsection{Skeleta and coskeleta}
For $n\geq 0$, let $\sk_n: \Sh \to \Sh$ be the $n$th skeleton functor for simplicial sets. It is the left adjoint of the coskeleton functor $\cosk_n: \Sh \to \Sh$. For a simplicial set $Z$, the simplicial set $\sk_nZ$ is given by the subspace of $X$ generated by simplices of degree $\leq n$. The $k$-simplices of the $n$th coskeleton $\cosk_nZ$ are given by the set $\Hom_{\Sh}(\sk_n\Delta[k],Z)$. If $Z=Y$ is a profinite space, then $\cosk_nY$ inherits the structure of a profinite space. For $\sk_n\Delta[n]$ is a finite simplicial set, and hence the set 
$$\Hom_{\Sh}(\sk_n\Delta[k],|Y|)\cong \Hom_{\hSh}(\sk_n\Delta[k], Y)$$
inherits the structure as a profinite set.\\
Moreover, if $Y$ is a profinite $G$-space, then $\Hom_{\hSh}(\sk_n\Delta[k], Y)$ is a limit of finite discrete $G$-sets, since $\sk_n\Delta[n]$ is a finite simplicial set with trivial $G$-action. 
Thus the well-known constructions for skeleta and coskeleta yield a pair of adjoint endofunctors $(\sk_n,\cosk_n)$ on $\hShg$. \\
In terms of mapping spaces, this adjunction translates into the natural isomorphism of pointed simplicial sets
\begin{equation}\label{skcosk}
\Map_G(\sk_nW, Y)\cong \Map_G(W,\cosk_nY)
\end{equation}
for every profinite $G$-space $W$ and every pointed profinite $G$-space $Y$. The basepoint of $\cosk_nY$ is given by the basepoint of $Y$. \\
Now let $X$ be an $\Omega$-spectrum in $\hSpg$. For a given integer $n$, the $n$th coskeleton of $X$ is defined to be the profinite $G$-spectrum whose $k$th space is the pointed profinite $G$-space $\cosk_{n+k}X_k$, i.e. the $(n+k)$th coskeleton of the $k$th space $X_k$ of $X$. Each $\cosk_{n+k}X_k$ is a fibrant pointed profinite $G$-space and the induced map of pointed profinite $G$-spaces 
$$\cosk_{n+k}X_k \to \Omega(\cosk_{n+k+1}X_{k+1})$$
is a weak equivalence in $\hShpg$, since $\cosk_n$ preserves fibrant objects and weak equivalences between fibrant objects. Thus $\cosk_nX$ is an $\Omega$-spectrum in $\hSpg$ if $X$ is. \\
\begin{lemma}\label{spectrapost}
Let $X$ be an $\Omega$-spectrum in $\hSpg$. The tower 
$$\ldots \to \cosk_2X \to \cosk_1X \to \cosk_0X \to \cosk_{-1}X \to \ldots$$
is a Postnikov tower for $X$, i.e. $X=\lim_n \cosk_nX$, $\pi_qX\to \pi_q(\cosk_nX)$ is an isomorphism if $q\leq n$, and $\pi_q(\cosk_nX)=0$ if $q>n$. In particular, the fiber $F(n)$ of $\cosk_nX \to \cosk_{n-1}X$ is an Eilenberg-MacLane spectrum in $\hSpg$ whose only nontrivial homotopy group is $\pi_nF(n)\cong \pi_nX$. 
\end{lemma}
\begin{proof}
Let $n$ and $q$ be integers, and $k$ be any positive integer such that $q+m\geq 0$. Since $X$ and $\cosk_nX$ are $\Omega$-spectra in $\hSpg$, the $q$th profinite homotopy group of $X$ is given by the abelian profinite group $\pi_{q+k}X_k$, and the $q$th homotopy group of $\cosk_nX$ is given by $\pi_{q+k}((\cosk_nX)_k)=\pi_{q+k}(\cosk_{n+k}X_k)$. Hence $\pi_q(\cosk_nX)$ is isomorphic to  $\pi_qX$ if $q+k\leq n+k$, i.e. if $q \leq n$, and $\pi_q(\cosk_nX)=0$ if $q+k > n+k$, i.e. if $q>n$. \\
Since $\cosk_nX$ and $\cosk_{n-1}X$ are $\Omega$-spectra in $\hSpg$, the fiber $F(n)$ of the map $\cosk_nX \to \cosk_{n-1}X$ is an $\Omega$-spectrum in $\hSpg$ by Proposition \ref{stablemapfiber}. For $k\geq 0$, the $k$th space $F(n)_k$ is the fiber of $\cosk_{n+k}X_k \to \cosk_{n-1+k}X_k$. This fiber is the profinite $G$-space $K(\pi_{n+k}X_k,n+k)$. The structure map is given by the natural equivalence of pointed profinite $G$-spaces $K(\pi_{n+k}X_k,n+k) \to \Omega(K(\pi_{n+k}X_k,n+k+1))$. Since $X$ is an $\Omega$-spectrum in $\hSpg$, the profinite $G$-module $\pi_{n+k}X_k$ is the $n$th homotopy group of $X$. Hence the fiber $F(n)$ is an Eilenberg-MacLane spectrum whose homotopy groups $\pi_mF(n)$ vanish for $m\neq n$ and whose $n$th homotopy group is isomorphic to $\pi_nX$.  
\end{proof}

\subsection{Homotopy fixed point spectra}
As an example of a continuous mapping spectrum, let $EG$ be a contractible profinite $G$-space with a levelwise free $G$-action, in other words a cofibrant profinite $G$-space which is weakly equivalent to a point. 
\begin{defn}\label{homotopyfixedpoints}
Let $X \in \hSpg$ be a profinite $G$-spectrum and let $R_G$ be a fixed functorial fibrant replacement in $\hSpg$. We define the homotopy fixed point spectrum $X^{hG}$ of $X$ to be the function spectrum of continuous $G$-equivariant maps from $EG$ to $R_GX$, i.e. 
$$X^{hG}:=\Map_G(EG,R_GX).$$ 
\end{defn}
\begin{prop}
If $f:X\to Y$ is a stable equivalence of profinite $G$-spectra, then the induced map $f^{hG}:X^{hG} \to Y^{hG}$ is a stable equivalence between $\Omega$-spectra in $\Sp$.
\end{prop}
\begin{proof}
If $f$ is a stable equivalence in $\hSpg$, then $R_G(f)$ is a stable equivalence between $\Omega$-spectra in $\hSpg$. Now the assertion follows from Lemma \ref{mapomega}, Lemma \ref{stablemapwe}, and the fact that $EG$ is cofibrant in $\hShg$. 
\end{proof}

\begin{remark}
%
1. One should note that the homotopy fixed points of a profinite $G$-spectrum are not the homotopy limit of the $G$-action in the sense of the construction of the previous section. The homotopy limit construction would treat $G$ as an abstract group, or rather the category defined by $G$, and would forget the profinite topology. But for the homotopy fixed points we want to remember the topology of $G$. This is why we use the explicit functor $\Map_G(EG,-)$. \\
2. The reader may wonder why we do not look for homotopy fixed point spectra that are profinite spectra themselves. The goal of our construction is to obtain the descent spectral sequence of Theorem \ref{stablefixeddescent}. Since the cohomology groups $H^s(G;\pi_tX)$ are not profinite groups in general, there is no reason to expect that $\pi_{*}(X^{hG})$ to be profinite. 
\end{remark}
Before we prove our main result about continuous homotopy fixed point spectra, let us choose a concrete model for $EG$. We let $EG$ be the profinite $G$-space given in degree $n$ by the $(n+1)$-fold product $G^{n+1}$ of copies of $G$. Its $i$th face map is given by 
\begin{equation}\label{standardcosimp1}
d^i: G^{n+1} \to G^{n}, (g_1,\ldots, g_{n+1}) \mapsto (g_1, \ldots, \hat{g_{i+1}}, \ldots, g_{n+1})
\end{equation}
where the hat means that the component $g_{i+1}$ is omitted. We let $g\in G$ act on $G$ by $h\mapsto hg^{-1}$, for all $h\in G$, and let $G$ act on $EG_n$ via the diagonal action on $G^{n+1}$. This action is free, and the quotient $EG/G$ is the usual classifying space $BG$ for $G$ given in degree $n$ by $G^n$ with face maps
$$\bar{d}^i(g_1, \ldots, g_{n+1}) \left \{ \begin{array}{l@{\quad: \quad}l}
                                          (g_2, \ldots, g_{n+1}) & i=0\\
                                          (g_1, \ldots, g_ig_{i+1},\ldots,g_{n+1}) & 1\leq i \leq n \\
                                         (g_1,\ldots, g_{n}) & i=n+1.
                                          \end{array} \right. 
$$ 
\begin{theorem}\label{stablefixeddescent}
Let $G$ be a profinite group and $X$ a profinite $G$-spectrum. There is a homotopy fixed point spectral sequence whose $E_2^{s,t}$-term is the $s$th continuous cohomology of $G$ with coefficients the profinite $G$-module $\pi_tX$: 
$$E_2^{s,t}=H^s(G;\pi_tX) \Rightarrow \pi_{t-s}(X^{hG}).$$
This spectral sequence converges completely to $\pi_*(X^{hG})$ if ${\lim}^1_r E_r^{s,t}=0$ for all $s$ and $t$.
\end{theorem}
\begin{proof} 
After applying the fibrant replacement functor $R_G$ we can assume that $X$ is fibrant in $\hSpg$. Let $\sk_nEG$ be the $n$th skeleton of $EG$ in $\hShg$. The induced $G$-equivariant map $\sk_{n-1}EG \to \sk_nEG$ is a cofibration of profinite $G$-spaces (since we only add free copies of $G$ in dimension $n$ as nondegenerate simplices). Filtering $EG$ by its finite skeleta we obtain a tower of spectra
\begin{equation}\label{sktower}
\ldots \to \Map_G(\sk_nEG,X) \to \Map_G(\sk_{n-1}EG,X) \to \ldots
\end{equation}
$$\{X(n):=\Map_G(\sk_nEG,X)\}_n$$
whose limit is isomorphic to $X^{hG}$. 
Since the skeleton $\sk_nEG$ is a cofibrant profinite $G$-space for every $n\geq 0$, every spectrum $\Map_G(\sk_nEG,X)$ is an $\Omega$-spectrum in $\Sp$ by Lemma \ref{mapomega}, and each map 
$$\Map_G(\sk_nEG,X) \to \Map_G(\sk_{n-1}EG,X)$$ 
is a projective (or, in the terminology used in \cite{bousfried}, a strict) fibration in $\Sp$ by Lemma \ref{mapfib}. Now we can use the fact that the projective fibrations between $\Omega$-spectra are exactly the stable fibrations. For $\Sp$ this was proved in \cite{bousfried}, Lemma A.8. But this is just a special case of the general fact that in a left Bousfield localization $L_{\Ch}\Mh$ of a model category $\Mh$ a map between $\Ch$-local objects is a fibration in $L_{\Ch}\Mh$ if and only if it is a fibration in $\Mh$ by \cite{hirsch}. Thus (\ref{sktower}) is in fact a tower of stable fibrations between $\Omega$-spectra in $\Sp$. \\
By the adjunction of skeleta and coskeleta, tower (\ref{sktower}) is naturally isomorphic to the tower of stable fibrations of $\Omega$-spectra in $\Sp$
\begin{equation}\label{tower}
\ldots \to \Map_G(EG,\cosk_nX) \to \Map_G(EG,\cosk_{n-1}X) \to \ldots
\end{equation}
By Lemma \ref{spectrapost}, the fiber of the map $\cosk_nX\to \cosk_{n-1}X$ is an Eilenberg-MacLane spectrum $H\pi_nX[n]$. Since $\pi_nX$ is a profinite $G$-module, $H\pi_nX[n]$ is a fibrant object of $\hSpg$ whose $k$th space is the pointed profinite $G$-space $K(\pi_nX,k+n)$. 
Now the crucial point is that the functor $\Map_G(EG,-)$ sending $\Omega$-spectra in $\hSpg$ to $\Omega$-spectra in $\Sp$ preserves homotopy fibers by Proposition \ref{stablemapfiber}. Hence, for every $n$, tower (\ref{tower}) induces a long exact sequence of abelian groups
$$\ldots \to \pi_{k+1}(\Mapg(EG,\cosk_{n-1}X))\to \pi_k(\Mapg(EG,F(n))) \to \pi_k(\Mapg(EG,\cosk_nX)) \to \ldots$$
continuing with $\pi_k(\Mapg(EG,\cosk_{n-1}X))$ and so on. This sequence yields exact couples for various $n$ and we obtain a spectral sequence 
$$E_2^{s,t}=\pi_{t-s}(\Map_G(EG,F(t)))\Rightarrow \pi_{t-s}(X^{hG}).$$
Since $F(t)$ is an $\Omega$-spectrum in $\hSpg$ equivalent to $H\pi_tX[t]$, the fiber of the map 
$$\Map_G(EG,\cosk_tX) \to \Map_G(EG,\cosk_{t-1}X)$$
is the $\Omega$-spectrum $\Map_G(EG,H\pi_tX[t])$ in $\Sp$. The $(t-s)$th homotopy group of this spectrum is given by any of the isomorphic abelian groups 
$$\pi_{t-s-k}(\Map_G(EG,K(\pi_{t+k}X_k,t+k)))$$
such that $t+k\geq0$. The following lemma shows that this this group is exactly the continuous cohomology group $H^s(G; \pi_tX)$. Finally, complete convergence follows as in \cite{bouskan}, IX \S 5. 
\end{proof}
To finish the proof of Theorem \ref{stablefixeddescent}, it remains to prove the following fact.
\begin{lemma}\label{ctscoh}
Let $\pi$ be a profinite $G$-module. For every $n \geq0$, there is a natural isomorphism 
$$\pi_q(\Map_G(EG,K(\pi,n)))\cong H^{n-q}(G; \pi)$$
between the homotopy groups of the fibrant pointed space $\Map_G(EG,K(\pi,n))$ and the continuous cohomology of $G$ with coefficients in $\pi$.
\end{lemma}  
\begin{proof}
The weak equivalence $K(\pi,n)\to \Omega(K(\pi,n+1))$ of pointed profinite $G$-spaces induces a weak equivalence 
$$\Map_G(EG,K(\pi,n))\to \Omega(\Map_G(EG,K(\pi,n+1)))$$ 
of pointed simplicial sets. Hence by induction it suffices to show 
$$\pi_0(\Map_G(EG,K(\pi,n)))\cong H^{n}(G; \pi).$$
The fibrant space $K(\pi,n)$ represents continuous cohomology in $\hSh$ with coefficients in the profinite abelian group $\pi$. In particular, this means that we have a natural isomorphism
$$\Hom_{\hSh}(Y, K(\pi,n)) \cong Z^n(Y_n; \pi)$$
where the right hand side denotes the set of continuous cocycles in the cochain complex of continuous maps
$$C^n(Y; \pi)=\Hom_{\hEh}(Y_n,\pi).$$ 
The differential $\delta^n:C^n(Y;\pi)\to C^{n+1}(Y;\pi)$ is the map associating to the continuous map $\alpha:Y_n\to \pi$ the map $\sum_{i=0}^{n+1}(-1)^i\alpha \circ d_i$, where $d_i$ denotes the $i$th face map $Y_{n+1}\to Y_n$.
For $Y=EG$, $C^n(EG;\pi)$ is just the set 
$$C^n(G;\pi)=\Hom_{\hEh}(G^{n+1}; \pi)$$
of continuous maps from the $(n+1)$-fold product of copies of $G$ to $\pi$ and $\delta^n$ sends a map $\alpha: G^{n+1}\to \pi$ to the map given by
$$(g_0, \ldots, g_{n+1}) \mapsto \sum_{i=0}^{n+1}(-1)^i\alpha(g_0, \ldots, \hat{g_i}, \ldots, g_{n+1}).$$ 
The subcomplex of $G$-equivariant continuous maps $C^*_G(G;\pi)\subset C^*(G;\pi)$ is given in degree $n$ by the maps $\alpha$ such that 
$$g\alpha(g_1,\ldots, g_{n+1})=\alpha(g_1g^{-1}, \ldots, g_{n+1}g^{-1}).$$
The homology groups of this cochain complex are the continuous cohomology groups of $G$ with coefficients in the profinite $G$-module $\pi$. 
Since the differentials in $C^*(G;\pi)$ are $G$-equivariant, the group of $n$th cocycles $Z_G^n(G; \pi)$ in the complex $C^*_G(G;\pi)$ is equal to the subgroup of $G$-equivariant maps in $Z^n(G;\pi)$. In terms of maps from $EG$ to $K(\pi,n)$ this means that there is a natural isomorphism
$$\Hom_{\hShg}(EG, K(\pi,n)) \cong Z_G^n(G; \pi).$$
Since the left hand side is just the set of $0$-simplices of $\Map_G(EG,K(\pi,n))$ we obtain an isomorphism
$$(\Map_G(EG,K(\pi,n)))_0 \cong Z_G^n(G; \pi).$$
Now it suffices to remark that, since $K(\pi,n)$ is a fibrant profinite $G$-space, that the equivalence relation 
$$d_0,d_1:(\Map_G(EG,K(\pi,n)))_1 \rightrightarrows (\Map_G(EG,K(\pi,n)))_0$$ 
given by simplicial homotopy, which defines the equalizer of this diagram, induces an isomorphism 
$$\pi_0\Map_G(EG,K(\pi,n)) \cong H^n(G; \pi).$$ 
%
%
%
%
%
\end{proof}

\subsection{Cosimplicial descriptions of the homotopy fixed point spectral sequence}
The homotopy fixed point spectral sequence of Theorem \ref{stablefixeddescent} and the definition of homotopy fixed points can also be given in terms of cosimplicial spectra and total spaces of cosimplicial objects. Let $V^{\bullet}$ be a cosimplicial pointed space, i.e. a cosimplicial object in $\Shp$. The total space $\Tot V^{\bullet}$ of $V^{\bullet}$ (see \cite{bouskan}, X, \S 3) is given by the equalizer in $\Shp$ of the diagram 
$$\prod_{n\geq 0} \homp(\Delta[n],V^n) \rightrightarrows \prod_{\varphi:[n]\to [m]} \homp(\Delta[n],V^m).$$
For $k\geq0$, let $\mathrm{Tot}_k\,V^{\bullet}$ be the equalizer in $\Shp$ of the diagram 
$$\prod_{n\geq 0} \homp(\sk_k\Delta[n],V^n) \rightrightarrows \prod_{\varphi:[n]\to [m]} \homp(\sk_k\Delta[n],V^m).$$
Now let $Z^{\bullet}$ be a cosimplicial spectrum, i.e. a cosimplicial object in $\Sp$. The total spectrum $\Tot Z^{\bullet}$ is the spectrum whose $n$th space is the total space of the cosimplicial space $Z_n^{\bullet}$ (see \cite{thomason}, Definition 5.24). Since $\Tot$ commutes with cotensor objects for spaces, the structure maps of $\Tot Z^{\bullet}$ are given by the map
$$\Tot Z_n^{\bullet} \to \Tot (\Omega Z_{n+1}^{\bullet}) \cong \Omega(\Tot Z_{n+1}^{\bullet}).$$
Since cotensor objects in $\Sp$ are defined levelwise, $\Tot$ also commutes with cotensor objects for spectra. For $k\geq 0$, $\mathrm{Tot}_k\,Z^{\bullet}$ is defined to be the spectrum obtained by applying $\mathrm{Tot}_k$ levelwise.  \\

Now let $W$ be a profinite $G$-space and $X$ a profinite $G$-spectrum. Considering the $n$th profinite set $W_n$ of $W$ as a constant simplicial profinite $G$-space, we can view $W$ also as a simplicial object in $\hShg$. Applying the functor $\Mapg(-,X)$ to $W$ yields a cosimplicial spectrum which we denote by $\Mapg(W^{\bullet},X)$. 
\begin{lemma}\label{Tot}
{\rm (a)} For $W$ and $X$ as above, the total spectrum $\Tot(\Mapg(W^{\bullet},X))$ is naturally isomorphic to the spectrum $\Mapg(W,X)$.\\ 
{\rm (b)} For every $k\geq0$, the spectrum $\mathrm{Tot}_k(\Mapg(W^{\bullet},X))$ is naturally isomorphic to the spectrum $\Mapg(\sk_kW,X)$.
\end{lemma}
\begin{proof}
Since the total spectrum is defined levelwise, it suffices to prove the corresponding assertions when $X$ is a pointed profinite $G$-space. In this case, (a) follows immediately from the description of $\Tot(\Mapg(W^{\bullet},X))$ as the equalizer in $\Shp$ of the diagram
$$\prod_{n\geq 0} \homp(\Delta[n],\Mapg(W^n,X)) \rightrightarrows \prod_{\varphi:[n]\to [m]} \homp(\Delta[n],\Mapg(W^m,X)).$$
This proves (a). Claim (b) follows in the same way by considering the corresponding equalizer diagram for $\mathrm{Tot}_k$. 
\end{proof}
For $W=EG \in \hShg$, we obtain a cosimplicial spectrum $\Mapg(G^{\bullet+1},X)$ whose $n$th spectrum is $\Mapg(G^{n+1},X)$. The previous lemma and the construction of the homotopy fixed point spectral sequence in the proof of Theorem \ref{stablefixeddescent} imply the following result. 
\begin{prop}\label{sscomparison1}
Let $X$ be a fibrant profinite $G$-spectrum. \\
{\rm (a)} There is a natural isomorphism of spectra 
$$X^{hG}=\Mapg(EG,X)\cong \Tot(\Mapg(G^{\bullet+1},X)).$$ 
%
{\rm (b)} The spectral sequence of Theorem \ref{stablefixeddescent} is isomorphic to the spectral sequence associated to the tower of spectra 
$$\{\mathrm{Tot}_k\,(\Mapg(G^{\bullet+1},X))\}_k.$$
\end{prop}

In \cite{thomason}, Definition 5.23, Thomason calls a cosimplicial spectrum $Z$ a cosimplicial fibrant spectrum, if each spectrum $Z^n$ is a fibrant spectrum, and calls $Z$ a fibrant cosimplicial fibrant spectrum if in addition each cosimplicial space $Z_m^{\bullet}$ is Reedy fibrant  (see \cite{hirsch}, 15.3.2, or \cite{bouskan}, X, 4.6).
\begin{lemma}\label{cosimpfibrant}
{\rm (a)} Let $V$ be a fibrant pointed profinite $G$-space. Then $\Mapg(G^{\bullet+1},V)$ is a fibrant cosimplicial space in the Reedy model category structure on cosimplicial spaces.\\
{\rm (b)} Let $X$ be a fibrant profinite $G$-spectrum. Then $\Mapg(G^{\bullet+1},X)$ is a fibrant cosimplicial fibrant spectrum in the sense of \cite{thomason}, Definition 5.23.
\end{lemma}
\begin{proof}
(a) By \cite{hirsch}, Lemma 15.11.10, in order to show that $\Mapg(G^{\bullet+1},V)$ is Reedy fibrant it suffices to show that $G^{\bullet+1}$ is a Reedy cofibrant simplicial profinite $G$-space. This means that the map from the $n$th latching object $L_nG^{\bullet+1}$ to $G^{n+1}$ is a cofibration in $\hShg$. Since the $L_nG^{\bullet+1}$ is a subobject of $G^{n+1}$ in $\hSh$ and since the action of $G$ on $G^{n+1}$ is free, the map $L_nG^{\bullet+1}\to G^{n+1}$ is in fact a cofibration in $\hShg$.\\
Claim (b) follows from (a) and the fact that $\Mapg(G^{n+1},X)$ is an $\Omega$-spectrum for every $n$. 
\end{proof}

Finally, we would like to replace $\Mapg(G^{\bullet+1},X)$ by a cosimplicial spectrum of the form $\Map(G^{\bullet},X)$ given in cosimplicial degree $n$ by the spectrum $\Map(G^n,X)$. \\ 
Recall that the $i$th coface map 
$$d^i: \Mapg(G^{n}, X) \to \Mapg(G^{n+1},X)$$
is given by sending a map $\alpha: G^n\to X$ to the map $d^i(\alpha):G^{n+1}\to X$ that sends the $(n+1)$-tuple $(g_1,\ldots, g_{n+1})$ to $\alpha(g_1, \ldots, \hat{g_{i+1}}, \ldots, g_{n+1})$. \\
\begin{remark}
Here and in the following discussion we abuse notations and describe elements in $\Mapg(G^{n+1},X)$ or $\Map(G^n,X)$ simply by their effects on tuples of elements on $G$ and neglect that a general element in the $k$th set $\Mapg(G^{n+1},X_m)_k$ of the $m$th space is a $G$-equivariant map $G^{n+1}\times \Delta[k] \to X_m$, respectively a map $G^n\times \Delta[k]\to X_m$ in $\Map(G^n,X_m)_k$. Moreover, we will only discuss the coface maps, since they are important for the resulting cochain structures, and omit the calculations for the codegeneracy maps.
\end{remark}
On $\Map(G^{\bullet},X)$ we define the $i$th coface map $\tilde{d}^i: \Map(G^{n-1},X)\to \Map(G^n,X)$ to be the map which sends a map $\beta: G^{n-1}\to X$ to the map 
\begin{equation}\label{cosimpstructure2}
\tilde{d}^i(\beta)(g_1,\ldots,g_n)= \left \{ \begin{array}{l@{\quad: \quad}l}
                                          \beta(g_1,\ldots,\hat{g}_{i+1}, \ldots, g_n) & 0\leq i \leq n-1\\
                                         g^{-1}_{n}\beta(g_1g_n^{-1},\ldots,g_{n-1}g_n^{-1}) & i=n.
                                          \end{array} \right. 
\end{equation}
We then have a morphism 
$$\varphi: \Mapg(G^{\bullet+1},X)\to \Map(G^{\bullet},X)$$ 
of cosimplicial spectra defined as follows. For $\alpha\in \Mapg(G^{n},X)$, we define the map $\varphi(\alpha)\in \Map(G^{n-1},X)$ by
$$\varphi(\alpha)(g_1,\ldots, g_{n-1}) := \alpha(g_1, \ldots, g_{n-1}, 1).$$
The map $\varphi$ has an inverse $\varphi^{-1}$ defined by sending a map $\beta\in \Map(G^{n-1},X)$ to the map $\alpha=\varphi^{-1}(\beta)$ defined by 
$$\varphi^{-1}(\beta)(g_1,\ldots,g_n):=g_n^{-1}\beta(g_1g_n^{-1},\ldots, g_{n-1}g_n^{-1}).$$
Remembering that we let $G$ act on itself by letting $g\in G$ act by $h\mapsto hg^{-1}$, we see that $\varphi^{-1}(\beta)$ is in fact a $G$-equivariant map; for we have for every $g\in G$
$$\begin{array}{rcl}
g\varphi^{-1}(\beta)(g_1,\ldots,g_n) & = & gg_n^{-1}\beta(g_1g_n^{-1},\ldots, g_{n-1}g_n^{-1}) \\
 & = & (g_ng^{-1})^{-1}\beta(g_1g^{-1}(g_ng^{-1})^{-1},\ldots, g_{n-1}g^{-1}(g_ng^{-1})^{-1})\\
 & =& \varphi^{-1}(\beta)(g_1g^{-1},\ldots, g_ng^{-1}).
\end{array}$$
Now one can check
$$\begin{array}{rcl}
\varphi^{-1}(\varphi(\alpha))(g_1,\ldots,g_n) & = & g_n^{-1}\varphi(\alpha)(g_1g_n^{-1},\ldots, g_{n-1}g_n^{-1})\\
& = & g_n^{-1}\alpha(g_1g_n^{-1},\ldots, g_{n-1}g_n^{-1},1)\\
 & = & \alpha(g_1g_n^{-1}g_n,\ldots, g_{n-1}g_n^{-1}g_n, 1g_n)\\
 & = & \alpha(g_1,\ldots, g_n)
\end{array}$$
where the third equality uses that $\alpha$ is $G$-equivariant. 
On the other hand we have
$$\begin{array}{rcl}
\varphi(\varphi^{-1}(\beta))(g_1,\ldots,g_{n-1}) & = & \varphi^{-1}(\beta)(g_1,\ldots, g_{n-1},1)\\
 & = & \beta(g_1,\ldots, g_{n-1}).
\end{array}$$
Moreover, the maps $\varphi$ and $\varphi^{-1}$ are compatible with the coface maps. We start with $\varphi$. For $0\leq i\leq n-1$, we have 
$$\begin{array}{rcl}
\varphi(d^i(\alpha))(g_1,\ldots,g_{n}) & = & d^i(\alpha)(g_1, \ldots,  g_{n}, 1)\\
 & = & \alpha(g_1,\ldots, \hat{g_{i+1}}, \ldots, g_{n},1)\\
  & = & 1\alpha(g_11,\ldots, \hat{g_{i+1}}, \ldots, g_{n}1)\\
 & = & \tilde{d}^{i}(\varphi(\alpha))(g_1,\ldots, g_{n}).
\end{array}$$ 
For $i= n$, we have 
$$\begin{array}{rcl}
\varphi(d^n(\alpha))(g_1,\ldots,g_{n}) & = & d^n(\alpha)(g_1, \ldots,  g_{n}, 1)\\
 & = & \alpha(g_1, \ldots, g_{n})\\
  & = & g_n^{-1}\alpha(g_1g_n^{-1},\ldots, g_{n-1}g_n^{-1},1)\\
    & = & g_n^{-1}\varphi(\alpha)(g_1g_n^{-1},\ldots, g_{n-1}g_n^{-1})\\
 & = & \tilde{d}^{n}(\varphi(\alpha))(g_1,\ldots, g_{n}).
\end{array}$$ 

Now we check $\varphi^{-1}$. For $0\leq i \leq n-1$, we have
$$\begin{array}{rcl}
\varphi^{-1}(\tilde{d}^i(\beta))(g_1,\ldots,g_{n+1}) & = & g_{n+1}^{-1}\tilde{d}^i(\beta)(g_1g_{n+1}^{-1}, \ldots, g_{n-1}g_{n+1}^{-1})\\
 & = & g_{n+1}^{-1}\beta(g_1g_{n+1}^{-1},\ldots, g_{i}g_{n+1}^{-1}, g_{i+2}g_{n+1}^{-1}, \ldots, g_{n-1}g_{n+1}^{-1})\\
 & = & \varphi^{-1}(\beta)(g_1, \ldots, \hat{g_{i+1}}, g_{n+1})\\
 & = & d^i(\varphi^{-1}(\beta))(g_1,\ldots, g_{n+1}).
\end{array}$$ 
For $i=n$, we have
$$\begin{array}{rcl}
\varphi^{-1}(\tilde{d}^n(\beta))(g_1,\ldots,g_{n+1}) & = & g_{n+1}^{-1}\tilde{d}^n(\beta)(g_1g_{n+1}^{-1},\ldots, g_{n}g_{n+1}^{-1}))\\
 & = & g_{n+1}^{-1}(g_{n}g_{n+1}^{-1})^{-1}\beta(g_1g_{n+1}^{-1}(g_{n}g_{n+1}^{-1})^{-1},\ldots, g_{n-1}g_{n+1}^{-1}(g_{n}g_{n+1}^{-1})^{-1})\\
 & = & g_{n}^{-1}\beta(g_1g_{n}^{-1},\ldots, g_{n-1}g_{n}^{-1})\\
 & = & \varphi^{-1}(\beta)(g_1,\ldots,g_{n})\\
 & = & d^{n}(\varphi^{-1}(\beta))(g_1,\ldots, g_{n+1}).
\end{array}$$ 
Together with Proposition \ref{sscomparison1}, this proves the following comparison result.
\begin{prop}\label{sscomparison3}
Let $X$ be a fibrant profinite $G$-spectrum. We equip $\Map(G^{\bullet},X)$ with the cosimplicial structure of (\ref{cosimpstructure2}). The map $\varphi$ is an isomorphism of cosimplicial spectra and induces an isomorphism 
$$X^{hG} \cong \Tot(\Map(G^{\bullet},X)).$$
The map $\varphi$ induces a map of towers of spectra
$$\{\mathrm{Tot}_k\,(\Mapg(G^{\bullet+1},X))\}_k \to \{\mathrm{Tot}_k\,(\Map(G^{\bullet},X))\}_k$$
and an isomorphism of spectral seuences from the spectral sequence of Theorem \ref{stablefixeddescent} to the spectral sequence associated to the tower of spectra 
$$\{\mathrm{Tot}_k\,(\Map(G^{\bullet},X))\}_k.$$ 
The former spectral sequence converges to $\pi_{t-s}(X^{hG})$, and converges completely if ${\lim}^1_r E_r^{s,t}=0$ for all $s$ and $t$.
\end{prop}

\begin{remark}\label{sscomparisonremark}
In order to show that the spectral sequences of the proposition are isomorphic, one could also use the following fact proved in \cite{devinatz} and \cite{devinatzhopkins}. For a profinite $G$-module $M$, let $\Map(G^{\bullet},M)$ be the cochain complex of continuous, given in degree $n$ by the group of continuous maps from $G^n$ to $M$ with differentials $\sum_{i=0}^{n}\tilde{d}^i$, where $\tilde{d}^i$ is defined as in (\ref{cosimpstructure2}) with $X$ replaced by $M$. The cohomology of this cochain complex is isomorphic to the continuous cohomology $H^*(G;M)$ of $G$ with coefficients in $M$ (see \cite{devinatz}, Proof of Theorem 3.1 on page 139). This also shows that the map $\varphi$ induces an isomorphism of spectral sequences from the $E_2$-terms on.
\end{remark}
\subsection{Iterated homotopy fixed point spectra}
Now let $H$ be a closed normal subgroup of $G$. The homotopy fixed points under the action of $G$, $H$ and $G/H$ should be related to each other. As we mentioned above, the homotopy fixed point spectrum $X^{hH}$ is in general not a profinite spectrum anymore. Hence, in general, we cannot ask for a continuous action of $G/H$ on $X^{hH}$. But the homotopy fixed points $X^{hG}$ do respect the continuity of the action. So for comparing $X^{hG}$ with $X^{hH}$ under its induced $G/H$-action, we assume that $K:=G/H$ is a finite group.\\ 
Starting from the model structure of simplicial $K$-sets in \cite{gj}, V \S 2, we can define the category of $K$-spectra $\Spk$ as in Definition \ref{Gspectra}. Then the method of Hovey \cite{hovey} yields again a stable model structure on $\Spk$. As above, a map $f$ in $\Spk$ is a fibration if and only if its underlying map in $\Sp$ is a fibration of Bousfield-Friedlander spectra.\\
The induced map $X^{hG} \to X^{hH}$ factors through the fixed points $(X^{hH})^K$, since $K$ acts trivially on $X^{hG}$. By composing with the canonical map from fixed points $(X^{hH})^K=\Map_K(\ast,X^{hH})$ to homotopy fixed points $(X^{hH})^{hK}=\Map_K(EK, X^{hH})$, we get  the map
$$X^{hG} \to (X^{hH})^K \to (X^{hH})^{hK}.$$
It follows almost from the definitions that this map is an equivalence. We summarize this in the following theorem in which the last assertion follows from the first and the older brother of Theorem \ref{stablefixeddescent} for finite groups.
\begin{theorem}\label{iterated}
Let $X$ be a profinite $G$-spectrum and let $H$ be an normal subgroup of $G$ with finite quotient $K=G/H$. Then the map $X^{hG}\stackrel{\simeq}{\to} (X^{hH})^{hK}$ is a stable equivalence in $\Sp$. Moreover, there is a spectral sequence for iterated homotopy fixed points
$$H^{s}(K;\pi_t(X^{hH})) \Rightarrow \pi_{t-s}(X^{hG}).$$
\end{theorem}
%
%
\section{Morava stabilizer groups and Lubin-Tate spectra}
\subsection{$E_n$ has a model as a profinite $G_n$-spectrum}
We return to our main example discussed in the introduction. For a fixed prime number $p$ and an integer  $n\geq 1$, let $G_n$ denote the extended Morava stabilizer group associated to the height $n$ Honda formal group law $\Gamma_n$ over $\F_{p^n}$. It is a profinite group and can also be described as the group of automorphisms of the Lubin-Tate spectrum $E_n$ in the stable homotopy category. Hopkins and Miller have shown that $G_n$ is in fact the automorphism group of $E_n$ by $\Ah_{\infty}$-maps, cf. \cite{rezk}. Later on, Goerss and Hopkins extended this result to the $E_{\infty}$-setting.  Hence $G_n$ acts on the spectrum-level on $E_n$ by $E_{\infty}$-ring maps, cf. \cite{goersshopkins}. We will show now that there is a canonical model for $E_n$ in the category of profinite $G_n$-spectra.\\
Let $BP$ be the Brown-Peterson spectrum for the fixed prime $p$. Its coefficient ring is $BP_*=\Z_{(p)}[v_1,v_2,\ldots]$, where $v_n$ has degree $2(p^n-1)$. There is a canonical map $r:BP_* \to E_{n\ast}=W(\F_{p^n})[[u_1,\ldots,u_{n-1}]][u,u^{-1}]$ defined by $r(v_i)=u_iu^{1-p^i}$ for $i<n$, $r(v_n)=u^{1-p^n}$ and $r(v_i)=0$ for $i>n$. Let $I$ be an ideal in $BP_*$ of the form $(p^{i_0},v_1^{i_1},\ldots, v_{n-1}^{i_{n-1}})$. Such ideals form a cofiltered system. Its images in $E_{n\ast}$ under $r:BP_* \to E_{n\ast}$ provide each $\pi_tE_n$ with the structure of a continuous profinite $G_n$-module as 
$$
\pi_tE_n \cong \lim_I \pi_tE_n/I\pi_tE_n.
$$
In fact, for $t$ odd these groups vanish, for $t$ even each quotient $\pi_tE_n/I\pi_tE_n$ is a finite discrete $G_n$-module and the decomposition as an inverse limit of these finite discrete $G_n$-modules is $G_n$-compatible.\\
For a cofinal subsystem of such ideals $I$, there are finite generalized Moore spectra $M_I$ with trivial $G_n$-action whose Brown-Peterson homology is $BP_*(M_I)=BP_*/I$ and who have the property $\pi_t(E_n \wedge M_I)\cong \pi_tE_n/I\pi_tE_n$ for all $t$, see \cite{homasa}. Moreover there is a canonical equivalence
$$E_n\simeq \holim_I E_n \wedge M_I.$$
This observation of \cite{homasa} is the starting point for all attempts to view $E_n$ as a continuous $G_n$-spectrum. \\
Although each $\pi_t(E_n \wedge M_I)$ is a finite discrete $G_n$-module, $G_n$ does not act discretely on the spectra $E_n\wedge M_I$, in the sense that there is no open subgroup $U$ of $G_n$ such that the $G_n$-action on the whole spectrum $E_n \wedge M_I$ factors through $G_n/U$. But a slightly less demanding statement holds. Each $E_n\wedge M_I$ and hence $E_n$ have a model in $\mathrm{Sp}(\hSh_{\ast G_n})$, i.e. models that are built out of pointed profinite $G_n$-spaces. 
\begin{theorem}\label{Endecomp}
$E_n$ has a model in the category of continuous profinite $G_n$-spectra, i.e. there is a profinite $G_n$-spectrum $E'_n$ and a $G_n$-equivariant map of spectra 
$$\psi:E_n \to E'_n$$ 
that is a stable equivalence of underlying spectra.
\end{theorem}
\begin{proof}
To simplify the notation, we will set $E_{n,I}:=E_n\wedge M_I$. The profinite group $G_n$ is  finitely generated as a profinite group and is hence strongly complete by the theorem of Nikolov and Segal \cite{niksegal}. Thus, since the homotopy groups of the $G_n$-spectrum $E_{n,I}$ are all finite, we can apply the replacement functor of Theorem \ref{Gstablefinitecompletion} to obtain a profinite $G_n$-spectrum $F^s_{G_n}(E_{n,I})$ built out of fibrant pointed simplicial finite $G_n$-sets and a $G_n$-equivariant map of spectra
$$E_{n,I} \to F^s_{G_n}(E_{n,I})$$ 
which induces isomorphisms 
$$\pi_t(E_{n,I}) = \pi_t(F^s_{G_n}(E_{n,I})) = \pi_t(|F^s_{G_n}(E_{n,I})|)$$
for every $t$ and hence is a stable equivalence of spectra. 
We define the profinite $G_n$-spectrum $E'_n$ to be the homotopy limit in $\mathrm{Sp}(\hSh_{\ast G_n})$ of the $F^s_{G_n}(E_{n,I})$ 
$$E'_n:=\holim_I F^s_{G_n}(E_{n,I})\in \mathrm{Sp}(\hSh_{\ast G_n})$$
over the cofinal subsystem of ideals $I$ for which $M_I$ exists. The profinite $G_n$-spectrum $E'_n$ is equipped with a $G_n$-equivariant map of spectra 
$$\psi:E_n \stackrel{\simeq}{\to} \holim_I E_{n,I} \stackrel{\simeq}{\to} \holim_I F^s_{G_n}(E_{n,I}) = E'_n$$ 
which is a stable equivalence of underlying spectra. This proves the assertion. 
\end{proof}
%
%
The above theorem also implies that $E'_n$ is a fibrant profinite $G$-spectrum for every closed subgroup $G$ of $G_n$. This allows to define the continuous homotopy fixed spectrum of $E_n$ under the action of any closed or open subgroup of $G_n$. 
\begin{defn}\label{Enfixed}
Let $G$ be any closed subgroup of $G_n$. The continuous homotopy fixed point spectrum $E_n^{hG}$ of the Lubin-Tate spectrum is the homotopy fixed point spectrum of $E'_n$ considered via restriction as a profinite $G$-spectrum, i.e. 
$$E_n^{hG}:=\Mapg(EG,E'_n).$$
\end{defn}
%
%
\begin{remark}
Since $E_n$ is $K(n)_*$-local, $E_n'$ is $K(n)_*$-local, and since taking mapping spectra  preserves $K(n)_*$-local objects, $E_n^{hG}$ is a $K(n)_*$-local spectrum.
\end{remark}
\begin{cor}
Let $G$ be a closed subgroup of $G_n$, and $K$ be a normal subgroup of $G$ such that $G/K$ is finite. Then $E_n^{hG}$ is naturally equivalent to $(E_n^{hK})^{hG/K}$. There is a strongly convergent spectral sequence for iterated homotopy fixed points
$$H^{\ast}(G/K;\pi_*E_n^{hK}) \Rightarrow \pi_*E_n^{hG}.$$
\end{cor}
\begin{proof}
This follows from Theorem \ref{iterated} applied to $E'_n$ of Theorem \ref{Endecomp}.
\end{proof}

Before we prove Theorem \ref{descentssintro} of the introduction, we collect some  consequences of our construction. Let $K(n)$ be the $n$th $p$-primary Morava $K$-theory spectrum. Its coefficient ring is given by $K(n)_*=\F_p[v_n,v_n^{-1}]$ where $v_n$ has degree $2(p^n-1)$. Let $\hL=L_{K(n)}$ denote $K(n)_*$-localization in $\Sp$. As in the proof of Theorem \ref{Endecomp}, we choose a cofinal subsystem of these ideals such that the associated finite generalized Moore spectra $M_I$ satisfy $\hL(\Sph^0)=\holim_I M_I$. 
To simplify the notation we will write 
$$\Eni:=F^s_{G_n}(E_n\wedge M_I)$$
for the fibrant profinite $G_n$-spectrum $F^s_{G_n}(E_n\wedge M_I)$ built out of simplicial finite discrete $G$-sets.  
\begin{lemma}\label{holimeni}
For any closed subgroup $G$ of $G_n$, there is an isomorphism of $\Omega$-spectra
$$E_n^{hG} \cong \holim_I (\Eni)^{hG}.$$
\end{lemma}
\begin{proof}
By Proposition \ref{stablemapfiber} we have an isomorphism 
$$\Mapg(EG,\holim_I \Eni) \cong \holim_I \Mapg(EG,\Eni).$$
The first assertion now follows from the fact that we defined $E_n^{hG}$ to be the left hand side of this isomorphism, and that the right hand side is $\holim_I (\Eni)^{hG}$.
\end{proof}
Furthermore, since $G_n$ is a $p$-adic analytic profinite group, it is possible to find a sequence of normal open subgroups of $G_n$
\begin{equation}\label{Uis}
G_n =U_0 \varsupsetneq U_1 \varsupsetneq \cdots \varsupsetneq U_i \varsupsetneq \cdots 
\end{equation}
with $\bigcap_i U_i=\{e\}$. For the rest of the paper we choose a fixed sequence of such open subgroups of $G_n$. 
\begin{lemma}\label{Encolimgui}
For every closed subgroup $G$ of $G_n$ and every ideal $I$, the canonical map 
$$\colim_i (\Eni)^{hGU_i} \stackrel{\sim}{\to}(\Eni)^{hG}$$
is an equivalence of spectra.
\end{lemma}
\begin{proof}
Filtering $EG$ and $EGU_i$ by its subskeleta, induces for each $i$ a map of towers of spectra 
$$\{\Map_{GU_i}(\sk_nEGU_i,\Eni)\}_n \to \{\Map_{G}(\sk_nEG,\Eni)\}_n.$$
This induces a map of convergent spectral sequences $\{E_r^{s,t}(i)\} \to \{E_r^{s,t}\}$ of the form
$$
\xymatrix{
E_2^{s,t}(i) = H^s(GU_i;\pi_t\Eni) \ar[d] \ar@{=>}[r] & \pi_{t-s}(\Eni)^{hGU_i} \ar[d]\\
E_2^{s,t} = H^s(G;\pi_t\Eni)\ar@{=>}[r] & \pi_{t-s}(\Eni)^{hG}.}
$$
As open subgroups of $G_n$, the $GU_i$ have a uniform bound for their cohomological dimension, i.e. there is an $N\in \N$ such that $E_2^{s,t}(i)=0$ for all $i$ whenever $s>N$. By \cite{mitchell}, Proposition 3.3, this implies that the colimit of spectral sequences $\colim_i \{E_r^{s,t}(i)\}$ converges to the colimit of the abutments. Hence we get an induced map of convergent spectral sequences of the form
\begin{equation}\label{sscompEni}
\xymatrix{
\colim_iE_2^{s,t}(i) = \colim_i H^s(GU_i;\pi_t\Eni) \ar[d] \ar@{=>}[r] & \colim_i \pi_{t-s}(\Eni)^{hGU_i} \ar[d]\\
E_2^{s,t} = H^s(G;\pi_t\Eni)\ar@{=>}[r] & \pi_{t-s}(\Eni)^{hG}.}
\end{equation}
Since the $U_i$'s satisfy $\bigcap_i U_i=\{e\}$ and since $\pi_t\Eni$ is a finite discrete $G_n$-module for each $t$, the map of $E_2$-terms of (\ref{sscompEni}) is an isomorphism for all $s$ and $t$. Since both spectral sequences are strongly convergent, this implies that the right hand map of (\ref{sscompEni}) is an isomorphism of finite abelian groups and that the canonical map of spectra 
$$\colim_i (\Eni)^{hGU_i} \to (\Eni)^{hG}$$
is an equivalence for each $I$. 
\end{proof}
%

\begin{prop}\label{holimcolimeni}
For any closed subgroup $G$ of $G_n$, there is an equivalence of $\Omega$-spectra
$$E_n^{hG} \simeq \holim_I \colim_i (\Eni)^{hGU_i}.$$
\end{prop}
\begin{proof}
By Lemma \ref{holimeni}, it suffices to show that $(\Eni)^{hG}$ is equivalent to the colimit $\colim_i 
(\Eni)^{hGU_i}$ for each $I$. This is the content of Lemma \ref{Encolimgui}.
\end{proof}
Let $S$ be a profinite set $S$ and let $S=\lim_{\alpha}S_{\alpha}$ be a presentation of $S$ as a limit of finite sets. We can consider $S$ also as a constant simplicial profinite set with identities as face and degeneracy maps. Then the constructions of the previous chapter give us spectra $\Map(S,E'_n)$ and $\Map(S,\Eni)$ for every $I$. 
The main case of interest for us is the one where $S=G^j$ is the $j$-fold product of a closed subgroup $G$ of $G_n$ for an integer $j\geq 0$.
\begin{lemma}\label{holimmapeni}
There is an isomorphism of $\Omega$-spectra
$$\Map(S,E'_n) \cong \holim_I\Map(S,\Eni).$$
\end{lemma}
\begin{proof}
By mimicking the proof of Proposition \ref{stablemapfiber} for $\Map(S,-)$, considered as a functor from $\hSpg$ to the category of $G$-spectra, we obtain an isomorphism 
$$\Map(S,\holim_I \Eni) \cong \holim_I \Map(S,\Eni).$$
The assertion then follows from the definition of $E'_n$ as $\holim_I\Eni$. 
\end{proof}

In the following, for profinite sets $S$ and $T$, we will denote the set of continuous maps from $S$ to $T$ also by
$$\Map(S,T):=\Hom_{\hEh}(S,T).$$
\begin{lemma}\label{mapgeni}
For each ideal $I$, the homotopy groups of the spectrum $\Map(S,\Eni)$ are given by the isomorphism
$$\pi_*\Map(S,\Eni)\cong \Map(S,\pi_*\Eni):=\Hom_{\hEh}(S,\pi_*\Eni)$$ 
where the right hand side denotes the group of continuous functions from the profinite set $S$ to the finite discrete homotopy groups of $\Eni$. 
\end{lemma}
\begin{proof}
Since $\Eni$ is an $\Omega$-spectrum that consists of simplicial finite sets, the spectrum $\Map(S,\Eni)$ is isomorphic to $\colim_{\alpha} \Map(S_{\alpha},\Eni)$. Taking homotopy groups commutes with this colimit of fibrant spectra by \cite{thomason}, Lemma 5.5.  
Since each $S_{\alpha}$ is a finite set, $\Map(S_{\alpha},\Eni)$ is just a finite product of copies of $\Eni$. This implies that taking homotopy groups also commutes with $\Map(S_{\alpha},-)$, i.e. there is an isomorphism
$$\pi_*\Map(S_{\alpha},\Eni) \cong \Map(S_{\alpha}, \pi_*\Eni).$$
Hence we obtain isomorphisms 
$$\pi_*\Map(S,\Eni) \cong \colim_{\alpha} \Map(S_{\alpha},\pi_*\Eni) \cong \Map(S,\pi_*\Eni).$$
\end{proof}

\begin{prop}\label{mapgen}
The homotopy groups of the spectrum $\Map(S, E'_n)$ are given by the isomorphism
$$\pi_*\Map(S,E'_n)\cong \Map(S,\pi_*E'_n):=\Hom_{\hEh}(S,\pi_*E'_n).$$  
\end{prop}
\begin{proof}
By Lemma \ref{mapgeni}, the assertion holds when we replace $E'_n$ by any $\Eni$. Moreover, $\Map(S,E'_n)$ is isomorphic to the homotopy limit $\holim_I\Map(S,\Eni)$ by Lemma \ref{holimmapeni}. Hence the spectral sequence for the homotopy groups of the homotopy limit of spectra 
$\Map(S,E'_n)$ has the form
$$E_2^{p,q}={\lim_I}^p \Map(S,\pi_q\Eni) \Rightarrow \pi_{q-p}\Map(S,E'_n).$$
Since the functor $\Map(S,-)$ is exact on the category of profinite abelian groups, the terms $E_2^{p,q}$ vanish for $p>0$ and the spectral sequence collapses. For each $q\in \Z$, we obtain an induced isomorphism 
$$\lim_I \Map(S,\pi_q\Eni) \cong \pi_{q}\Map(S,E'_n).$$
We remark that the category of profinite groups is the pro-category of finite groups. This means in particular, that the left hand side satisfies
$$\lim_I \Map(S,\pi_q\Eni) \cong \Map(S,\lim_I\pi_q\Eni).$$
Now it remains to recall $\pi_qE'_n\cong \lim_I\pi_q\Eni$. This proves the assertion.
\end{proof}

\subsection{Comparison with the construction of Devinatz and Hopkins} 
As in the introduction we will adopt the notation of \cite{behrensdavis} and \cite{davis} to denote the Devinatz-Hopkins fixed point spectra of $E_n$  of \cite{devinatzhopkins} by $E_n^{dhG}$. Devinatz and Hopkins define these spectra in two steps. First they define homotopy fixed point spectra $E_n^{dhU}$ for open subgroups $U$ of $G_n$ using the fact that $G_n/U$ is finite and that the expected homotopy type of $E_n^{dhU}\wedge E_n$ is the one of $\Map(G_n/U,E_n)$. This construction depends on the specific properties of $G_n$ as a $p$-adic analytic profinite group, and of the important calculations of the mapping space of self-maps of $E_n$ by Goerss and Hopkins in \cite{goersshopkins}. (The construction of $E'_n$ of course also relies on this calculation which makes it possible to define a model of $E_n$ on which $G_n$ acts by maps of spectra and not only by maps in the stable homotopy category.) \\ 
%
%
%
In a second step, they define the homotopy fixed points for a closed subgroup of $G_n$. Using the fixed sequence (\ref{Uis}), Devinatz and Hopkins define for an arbitrary closed subgroup $G$ of $G_n$
$$E_n^{dhG}:=\hat{L}(\hocolim_i E_n^{dh(U_iG)})$$
where $\hocolim$ denotes the homotopy colimit in the category 
of commutative $\Sph^0$-algebras in the category of $\Sph^0$-modules of \cite{ekmm}.\\ 
%
As remarked in \cite{devinatzhopkins}, p. 5, 
there is a canonical map $E_n^{dhG} \to E_n$. Since this map is $G_n$-equivariant and since $G$ acts trivially on $E_n^{dhG}$, this map factors through $E_n^{dhG}\to E_n^G$, i.e. through the $G$-fixed points of $E_n$. Furthermore, there is a canonical map from the fixed points $E_n^G$ to the continuous homotopy fixed points $E_n^{hG}$ of Definition \ref{Enfixed} given by the composition
$$E_n^G \to (E'_n)^G=\Mapg(\ast,E'_n) \to \Mapg(EG,E'_n)=E_n^{hG}.$$
Composition with the map $E_n^{dhG} \to E_n^G$ yields a map from $E_n^{dhG}$ to $E_n^{hG}$. Hence for any closed (or open) subgroup $G$ of $G_n$, there is a canonical map between the  homotopy fixed point spectra 
\begin{equation}\label{comparison}
E_n^{dhG} \longrightarrow E_n^{hG}.
\end{equation}
\subsection{Proof of Theorem \ref{descentssintro}} 
We have defined the homotopy fixed point spectrum $E_n^{hG}$ for an arbitrary closed subgroup $G$ of $G_n$ in Definition \ref{Enfixed}. This proves the first assertion of (i) in Theorem \ref{descentssintro}. 
The spectral sequence of Theorem \ref{stablefixeddescent} then yields the homotopy fixed point spectral sequence  
$$H^{s}(G;\pi_tE_n) \Rightarrow \pi_{t-s}(E_n^{hG})$$
for any closed subgroup $G$ of $G_n$ starting from the continuous cohomology of $G$ and converging to the homotopy groups of the homotopy fixed points of $E_n$ under $G$. This spectral sequence is natural in $G$. Since $G_n$ is a $p$-adic analytic group, so is $G$ and its continuous cohomology groups with profinite coefficients are also profinite groups, cf. \cite{symondsweigel}. Hence the ${\lim}^1$-terms of the $E_r$-terms all vanish. Moreover, $G$ has finite virtual cohomological dimension and the spectral sequence above is strongly convergent. This proves the first assertion of part (ii) of Theorem \ref{descentssintro}.\\
%

Let $S$ be a profinite set and $S=\lim_{\alpha}S_{\alpha}$ be a presentation of $S$ as a limit of finite sets. 
In preparation of the proof the main theorem, we will define a map of spectra
$$\hL(E_n\wedge \Map(S,E'_n)) \to \Map(S \times G_n,E'_n)$$ 
that induces an isomorphism on homotopy groups. \\
Let $U_i$ be an element of the sequence of open normal subgroups of $G_n$ and let $I$ be an element of the fixed set of ideals such that $\hL(\Sph^0)=\holim_I M_I$. 
For a given $I$, 
let $\mu_I$ be the induced multiplication map $\mu_I: \Eni \wedge \Eni \to \Eni$. 
We define a map of spectra
$$\tau(I,i): E_n^{dhU_i} \wedge \Map(S, E'_n) \to \Map(S \times G_n/U_i, \Eni)$$
as follows. Given a pair $(x,f)$ on the left hand side to the map $S\times G_n/U_i  \to \Eni$ that sends a pair $(s, \bar{h})$ in $S\times G_n/U_i$ to the element $\mu_I(f(s), x\bar{h}^{-1})$ in $\Eni$, where we abuse notations and denote the images of $x$ and $f(s)$ in $\Eni$ by the same symbols. Since $f$ is continuous and $G_n/U_i$ is a finite set, this map is continuous as well. \\ 
%
%
By taking colimits, the maps $\tau(I,i)$ induce a map
$$\colim_{i} E_n^{dhU_i} \wedge \Map(S, E'_n) \to \colim_{i} \Map(S \times G_n/U_i,\Eni).$$
%
%
The right hand spectrum is just $\Map(S\times G_n, \Eni)$, and we obtain a map
$$\tau(I):\colim_{i} E_n^{dhU_i} \wedge \Map(S, E'_n) \to \Map(S \times G_n,\Eni).$$
%
%
Moreover, the maps $\tau(I)$ induce a map 
\begin{equation}\label{holimtau}
\colim_{i} E_n^{dhU_i} \wedge \Map(S,E'_n) \to \holim_I \Map(S\times G_n, \Eni).
\end{equation}
By Lemma \ref{holimmapeni}, the right hand side is isomorphic to $\Map(S\times G_n, E'_n)$. Finally, since $\Map(S\times G_n, E'_n)$ is $K(n)_*$-local, we obtain an induced map from the $K(n)_*$-localization
\begin{equation}\label{tauj}
\tau: \hL(E_n \wedge \Map(S,E'_n)) \to \Map(S\times G_n, E'_n).
\end{equation}
%
%
\begin{theorem}\label{tauprop}
The map $\tau$ induces an isomorphism on homotopy groups. 
\end{theorem}
\begin{proof}
The map $\tau$ induces a homomorphism on homotopy groups of underlying spectra
\begin{equation}\label{pitau}
\pi_*(\tau): \pi_*\hL(E_n\wedge \Map(S,E'_n)) \to \pi_*\Map(S\times G_n, E'_n).
\end{equation}
By Theorem \ref{Endecomp} and Proposition \ref{mapgen}, the right hand side of (\ref{pitau}) is isomorphic to $\Map(S\times G_n,\pi_*E_n)$. By \cite{devinatzhopkins}, (2.3) and page 25, 
and the construction of $\tau$, the homomorphism $\pi_*(\tau)$ is exactly the isomorphism 
$$\pi_*\hL(E_n\wedge \Map(S,E'_n)) \to \Map(S\times G_n,\pi_*E_n).$$
%
\end{proof}

If $S=*$ is just a point, the equivalence $E_n\simeq |E'_n|$ implies the following consequence of the previous proposition.
\begin{cor}
The map of spectra 
$$\hL(E_n \wedge E_n) \to \Map(G_n, E'_n)$$
induces an isomorphism on homotopy groups.
\end{cor}
%
%
We would like to generalize this result a bit further. Therefore, let $Z\in \Sp$ be a spectrum and let $t\in \Z$ be an integer. We denote by $[Z,?]^t$ maps in the stable homotopy category of degree $-t$. Let $Z\cong \colim_{\alpha}Z_{\alpha}$ be a decomposition of $Z$ as a filtered colimit of finite subspectra $Z_{\alpha}$. We recall from \cite{devinatzhopkins} that the $E_n$-cohomology groups $E_n^tZ=[Z,E_n]^t$ of $Z$ are topologized as the profinite abelian groups
$$E_n^tZ=\lim_{\alpha,I} E_n^tZ_{\alpha}/IE_n^tZ_{\alpha}.$$
For each $\alpha$, the above decomposition of $E_n^tZ$ into a limit over the ideals $I$ is induced by the decomposition of $E_n\cong \holim E_n\wedge M_I$. For each $I$, the finite abelian group $E_n^tZ_{\alpha}/IE_n^tZ_{\alpha}$ is isomorphic to $(E_n\wedge M_I)^tZ_{\alpha}$. \\
Furthermore, by the construction of $\Eni$ and by (\ref{compgen}), we have a natural isomorphism
$$(\Eni)^tZ_{\alpha}\cong (E_n\wedge M_I)^tZ_{\alpha}\cong E_n^tZ_{\alpha}/IE_n^tZ_{\alpha}.$$
By definition of $E'_n$ as $\holim_I \Eni$ and by (\ref{compgen}), we obtain a natural isomorphism
$$(E'_n)^tZ_{\alpha} \cong E_n^tZ_{\alpha}/IE_n^tZ_{\alpha}.$$
Hence the map $\psi:E_n \to E'_n$ induces an isomorphism 
\begin{equation}\label{EnE'n}
E_n^tZ \cong (E'_n)^tZ
\end{equation}
of profinite abelian groups. Considering the $G_n$-action on $E_n$, we even know that (\ref{EnE'n}) is an isomorphism of profinite $G_n$-modules. 
%
%
\begin{prop}\label{taupropZ}
Let $G$ be a closed subgroup of $G_n$, $j\geq 0$ a positive integer. For every spectrum $Z\in \Sp$ and every integer $t\in \Z$, the map  
$$\tau^j(Z): [Z,\hL(E_n\wedge \Map(G^j,E'_n))]^t \to \Map(G^j\times G_n, E_n^tZ)$$
induced by $\tau$ is an isomorphism of abelian groups.  
\end{prop}
\begin{proof}
By the definition of the topology on $E_n^tZ \cong (E'_n)^tZ$, it suffices to show the assertion for any finite subspectrum $Z_{\alpha}$ of $Z$. Hence we can assume that $Z$ is a finite spectrum. Moreover, since $\tau$ induces a natural transformation
$$[-, \hL(E_n\wedge \Map(G^j,E'_n))]^* \to \Map(G^j\times G_n, E_n^*(-))$$ 
of cohomology theories satisfying the product axiom, it suffices to prove the assertion for $Z=\Sph^0$. For $\Sph^0$, the assertion follows from Proposition \ref{mapgen} and Theorem \ref{tauprop}.
\end{proof}

%
%
Let $G$ be again a closed subgroup of $G_n$. The cosimplicial spectrum $\Map(G^{\bullet},E'_n)$ with the coface maps defined in (\ref{cosimpstructure2}) defines a cochain complex 
in the stable homotopy category of spectra given in degree $j$ by $\Map(G^j,E'_n)$ with differentials $\tilde{\delta}^i:=\sum_i (-1)^i\tilde{d}^i$. \\
Moreover, the map $E_n^{dhG}\to E_n\to E'_n$ induces an augmentation map 
$$E_n^{dhG} \to E'_n=\Map(G^0,E'_n).$$ 
Thus we obtain a sequence of maps in the stable homotopy category
\begin{equation}\label{dhgseq}
*\to E_n^{dhG} \to \Map(G^0,E'_n) \to \Map(G^1,E'_n) \to \Map(G^2,E'_n) \to \ldots
\end{equation}
Using the terminology of \cite{miller} and \cite{devinatzhopkins} for $K(n)_*$-local $E_n$-Adams resolutions, we can now show the following key result.
\begin{theorem}\label{adamsthm}
Sequence (\ref{dhgseq}) is a $K(n)_*$-local $E_n$-Adams resolution of $E_n^{dhG}$.
\end{theorem}
\begin{proof}
First, we have to show that each $\Map(G^j,E'_n)$ is $E_n$-injective, i.e. that it is a retract of $\hL(Y\wedge E_n)$ for some spectrum $Y$. The isomorphisms in the stable homotopy category
$$\begin{array}{rcl}
\Map(G^j,E'_n) & \cong & \holim_I \colim_i \Map(G/(G\cap U_i),\Eni) \\
 & \cong & \holim_I \colim_i \prod_{G/(G\cap U_i)} \Eni \\
 & \cong & \holim_I \colim_i \prod_{G/(G\cap U_i)} E_n\wedge M_I \\
  & \cong & \holim_I \colim_i (\prod_{G/(G\cap U_i)}\Sph^0) \wedge E_n\wedge M_I \\
   & \cong & \hL(\colim_i (\prod_{G/(G\cap U_i)}\Sph^0) \wedge E_n).
\end{array}
$$
shows that $\Map(G^j,E'_n)$ is a retract of $\hL(\colim_i (\prod_{G/(G\cap U_i)}\Sph^0) \wedge E_n)$. \\
It remains to show that (\ref{dhgseq}) is $E_n$-exact. By \cite{devinatzhopkins}, Appendix A, it suffices to show that the sequence of abelian groups 
\begin{equation}\label{dhgcochainseq}
0\to [Z,\hL(E_n^{dhG}\wedge E_n)]^t \to [Z,\hL(E_n\wedge\Map(G^0,E'_n))]^t \to [Z,\hL(E_n\wedge\Map(G^1,E'_n))]^t \to \ldots
\end{equation}
is exact for every spectrum $Z\in \Sp$. \\
By Proposition \ref{taupropZ}, we have an isomorphism 
$$H^i([Z,\hL(E_n\wedge\Map(G^*,E'_n))]^t)\cong H^i(\Map(G^*\times G_n, E_n^tZ)).$$
%
%
%
By \cite{serre}, Chap. I, \S 1.2, Proposition 1.1, the surjective map $G_n\to G_n/G$ of profinite sets admits a continuous (set-theoretic) section. This implies that there is a homeomorphism $G_n \cong G\times G_n/G$. This induces an isomorphism of cochain complexes
$$\Map(G^*\times G_n, E_n^tZ) \cong \Map(G^*\times G \times G_n/G, E_n^tZ).$$
For the latter cochain complex, the map 
$$q: \Map(G^{*+1}\times G \times G_n/G, E_n^tZ) \to \Map(G^*\times G \times G_n/G, E_n^tZ)$$
defined by
$$q(f)(g_1, \ldots, g_{j},t, \bar{h})=(-1)^{j+1}f(g_1t, \ldots, g_jt, t,1, \bar{h})$$
is a contracting homotopy. This shows  
$$H^i(\Map(G^*\times G_n,E_n^tZ)) = 0 ~\mathrm{for}~i>0.$$
For $i=0$, $H^0(\Map(G^*\times G_n, E_n^tZ))$ is the kernel of the map 
$$\tilde{\delta}^0:\Map(G_n, E_n^tZ) \to \Map(G\times G_n, E_n^tZ)$$
given by the subgroup of $G$-invariant elements 
$$(\Map(G_n, E_n^tZ))^G \subset \Map(G_n, E_n^tZ).$$
%
%
Hence we have an isomorphism 
$$H^0(\Map(G_n\times G^*, E_n^tZ))\cong (\Map(G_n, E_n^tZ))^G.$$
By \cite{devinatzhopkins}, Proposition 6.3, there is an isomorphism
$$[Z,\hL(E_n^{dhG}\wedge E_n)]^t \cong \Map(G_n, E_n^tZ)^G.$$
This shows that sequence (\ref{dhgcochainseq}) is exact.
\end{proof}

\begin{theorem}\label{finalthm}
The homotopy fixed point spectral sequence converging to $E_n^{hG}$ is isomorphic to the $K(n)_*$-local $E_n$-Adams spectral sequence of \cite{devinatzhopkins} converging to $E_n^{dhG}$. The canonical map $E_n^{dhG}\to E_n^{hG}$ is an equivalence of spectra. 
\end{theorem}
\begin{proof}
By Theorem \ref{adamsthm}, sequence (\ref{dhgseq}), which is induced by the cosimplicial spectrum $\Map(G^{\bullet},E'_n)$, is a $K(n)_*$-local $E_n$-Adams resolution of $E_n^{dhG}$. Hence, by \cite{devinatzhopkins}, Proposition A.5, the spectral sequence associated to the tower of fibrations of total spectra 
$$\{\mathrm{Tot}_k\,\Pi^{\bullet}(\Map(G^{\bullet},E'_n))\}_k$$
is isomorphic to a the $K(n)_*$-local $E_n$-Adams spectral sequence for $E_n^{dhG}$, where $\Pi^{\bullet}$ denotes the cosimplicial replacement functor of \cite{bouskan}, XI. Moreover, by \cite{devinatzhopkins}, Proposition A.8, the map
$$E_n^{dhG} \to \holim_{\Delta} \Map(G^{\bullet},E'_n)$$
is an equivalence of spectra. \\
Since $\Map(G^{\bullet},E'_n)\cong \Mapg(G^{\bullet+1},E'_n)$ is a fibrant cosimplicial fibrant spectrum, the canonical map 
$$\Tot \Map(G^{\bullet},E'_n) \to \holim_{\Delta} \Map(G^{\bullet},E'_n)$$
is an equivalence of spectra by \cite{thomason}, Lemma 5.25. Hence we obtain a commutative diagram of maps of spectra 
$$\xymatrix{
E_n^{dhG} \ar[rr] \ar[dr]^{\sim} & &  \Tot \Map(G^{\bullet},E'_n) \ar[dl]_{\sim} \\
 & \holim_{\Delta} \Map(G^{\bullet},E'_n). & }
$$
Since the two bottom maps are equivalences, the map 
$$E_n^{dhG} \to \Tot \Map(G^{\bullet},E'_n)$$
is an equivalence as well. Since the right hand spectrum is isomorphic to $E_n^{hG}$ by Proposition \ref{sscomparison3}, this shows that the map 
$$E_n^{dhG} \to E_n^{hG}$$
is an equivalence of spectra.\\
Furthermore, the map of spectral sequences induced by the map of towers of fibrations 
$$\{\mathrm{Tot}_k\,(\Map(G^{\bullet},E'_n))\}_k \to \{\mathrm{Tot}_k\,\Pi^{\bullet}(\Map(G^{\bullet},E'_n))\}_k$$
is an isomorphism of spectral sequences from $E_2$-terms on (see for example \cite{devinatzhopkins}, Proposition 4.16). Together with the previous conclusions, this proves the assertion. 
\end{proof}
%
%
\begin{cor}
The spectral sequence obtained by mapping a spectrum $Z$ into the tower of fibrations $\{\mathrm{Tot}_k\,(\Map(G^{\bullet},E'_n))\}_k$ has the form
$$E_2^{s,t}=H^s(G;E_n^tZ)\Rightarrow (E_n^{hG})^{t+s}Z.$$
This spectral sequence is strongly convergent and isomorphic to the $K(n)_*$-local $E_n$-Adams spectral sequence of \cite{devinatzhopkins} converging to $(E_n^{dhG})^{*}Z$.
\end{cor}
\begin{proof}
By Theorem \ref{adamsthm} and by \cite{devinatzhopkins}, Proposition A.5, the spectral sequence obtained by mapping a spectrum $Z$ into the tower of fibrations $\{\mathrm{Tot}_k\,(\Map(G^{\bullet},E'_n))\}_k$ is isomorphic to the $K(n)_*$-local $E_n$-Adams spectral sequence of \cite{devinatzhopkins} converging to $(E_n^{dhG})^{*}Z$. By Proposition \ref{sscomparison3}, $[Z,\Tot \Map(G^{\bullet},E'_n)]^*$ is isomorphic to $(E_n^{hG})^*Z$. 
The identification of the abutments then follows from Theorem \ref{finalthm}, since the equivalence $E_n^{dhG}\simeq E_n^{hG}$ implies a natural isomorphism $(E_n^{dhG})^{*}Z \cong (E_n^{hG})^{*}Z$. 
\end{proof}

\bibliographystyle{amsplain}

\end{document}